\documentclass[11pt]{article} 

\usepackage[utf8]{inputenc} 

\usepackage{geometry} 
\usepackage{graphicx} 

\usepackage[parfill]{parskip} 

\usepackage{booktabs} 
\usepackage{array} 
\usepackage{paralist} 
\usepackage{verbatim} 

\usepackage{booktabs, url}
\usepackage{pifont}

\usepackage{amsmath, multirow}
\usepackage{amssymb}
\usepackage{amsthm}
\usepackage{algorithmic}
\usepackage{algorithm}

\usepackage{cite}
\usepackage{fullpage}

\usepackage{algorithmic}
\usepackage{algorithm}
\floatname{algorithm}{Method}

\usepackage{color}

\usepackage{caption}
\usepackage{subcaption}

\newtheorem{remark}{Remark}[section]
\newtheorem{proposition}{Proposition}[section]

\newtheorem{corollary}{Corollary}[section]

\newtheorem{theorem}{Theorem}[section]
\theoremstyle{remark}

\newcommand{\rank}{\text{rank}}
\newcommand{\FSe}{\text{FS}_\varepsilon}
\newcommand{\FS}{\text{FS}}
\newcommand{\FSek}{\text{FS}_{\varepsilon_{k}}}
\newcommand{\RFSe}{\text{R-FS}_{\varepsilon, \delta}}
\newcommand{\PATHe}{\text{PATH-R-FS}_{\varepsilon}}
\newcommand{\ebs}{\textsc{LS-Boost}$(\varepsilon)$\,}
\newcommand{\lasso}{\textsc{Lasso }}
\newcommand{\blasso}{\textsc{BLasso }}
\newcommand{\blassoperiod}{\textsc{BLasso}}
\newcommand{\lassoperiod}{\textsc{Lasso}}

\newcommand{\by}{\mathbf{y}}
\newcommand{\bX}{\mathbf{X}}

\newcommand{\sgn}{\text{sgn}}
\newcommand{\bbR}{\mathbb{R}}

\newcommand{\bbE}{\mathbb{E}}

\newcommand{\LAR}{{\textsc{Lar}} }

\newcommand{\pmin}{\mathrm{pmin}}

\newcommand{\M}{\mathbf}
\newcommand{\eps}{\varepsilon}

\newcommand{\bei}{\begin{itemize}}
\newcommand{\eei}{\end{itemize}}
\newcommand{\bee}{\begin{enumerate}}
\newcommand{\eee}{\end{enumerate}}

\newcommand{\x}{x}
\newcommand{\y}{y}

\newenvironment{myarray}[2][1]
  {\def\arraystretch{#1}\array{#2}}
  {\endarray}

\DeclareMathOperator*{\argmin}{arg\,min}
\DeclareMathOperator*{\argmax}{arg\,max}
\DeclareMathOperator*{\mini}{minimize}

\usepackage{fancyhdr} 
\usepackage[nottoc,notlof,notlot]{tocbibind} 
\usepackage[titles,subfigure]{tocloft} 

\title{A New Perspective on Boosting in Linear Regression via Subgradient Optimization and Relatives}

\author{Robert M. Freund\thanks{MIT Sloan School of Management, 77 Massachusetts Avenue, Cambridge, MA   02139
({mailto:  rfreund@mit.edu}).  This author's research is supported by AFOSR Grant No. FA9550-11-1-0141 and the MIT-Chile-Pontificia Universidad Católica de Chile Seed Fund.}
\and Paul Grigas\thanks{MIT Operations Research Center, 77 Massachusetts Avenue, Cambridge, MA   02139
({mailto:  pgrigas@mit.edu}).  This author's research has been partially supported through an NSF Graduate Research Fellowship and the MIT-Chile-Pontificia Universidad Católica de Chile Seed Fund.}
\and Rahul Mazumder\thanks{Department of Statistics, Columbia University, New York, NY 10027. The author's research has been funded by Columbia University's startup fund and a grant from the Betty Moore-Sloan foundation.
({mailto:  rm3184@columbia.edu}).}}

\begin{document}
\maketitle

\begin{abstract}

In this paper we analyze boosting algorithms~\cite{Friedman99,FHT00,LARS} in linear regression from a new perspective:  that of modern first-order methods in convex optimization.  We show that classic boosting algorithms in linear regression, namely the incremental forward stagewise algorithm ($\FSe$) and least squares boosting (\textsc{LS-Boost}$(\varepsilon)$), can be viewed as subgradient descent to minimize the loss function defined as the maximum absolute correlation between the features and residuals.  We also propose a modification
of $\FSe$ that yields an algorithm for the \lassoperiod, and that may be easily extended to an algorithm that computes the \lasso path for different values of the regularization parameter.  Furthermore, we show that these new algorithms for the \lasso may also be interpreted as the same master algorithm (subgradient descent), applied to a regularized version of the maximum absolute correlation loss function. We derive novel, comprehensive computational guarantees for several boosting algorithms in linear regression (including \ebs and $\FSe$) by using techniques of modern first-order methods in convex optimization. Our computational guarantees inform us about the statistical properties of boosting algorithms. In particular they provide, for the first time, a precise theoretical description of the amount of data-fidelity and regularization imparted by running a boosting algorithm with a prespecified learning rate for a fixed but arbitrary number of iterations, for \emph{any} dataset.

\end{abstract}

\section{Introduction}\label{sect_intro}

Boosting~\cite{Schapire90,freund1995boosting,FHT00,schapire2012boosting,ESLBook} is an extremely successful and popular supervised learning method that combines multiple weak\footnote{this term originates in the context of boosting for classification, where a ``weak'' classifier is slightly better than random guessing.}
 learners into a powerful ``committee.''
AdaBoost~\cite{FS96,schapire2012boosting,ESLBook} is one of the earliest boosting algorithms developed in the context of classification.
\cite{Br97,Br98} observed that AdaBoost may be viewed as an optimization algorithm, particularly as a form of gradient descent in a certain function space.
In an influential paper, \cite{FHT00} nicely interpreted boosting methods used in classification problems, and in particular AdaBoost, as instances of stagewise additive modeling~\cite{hastie1990generalized} -- a fundamental modeling tool in statistics.
This connection yielded crucial insight about the statistical model underlying boosting and provided a simple statistical explanation behind the success of boosting methods.
\cite{Friedman99} provided an interesting unified view of stagewise additive modeling and steepest descent minimization methods in function space to explain boosting methods.
This viewpoint was nicely adapted to various loss functions via a greedy function approximation scheme.
For related perspectives from the machine learning community, the interested reader is referred to the works~\cite{Mason00:_boost,ratsch2001soft} and the references therein.
\paragraph{Boosting and Implicit Regularization}
An important instantiation of boosting, and the topic of the present paper, is its application in linear regression. We use the usual notation with model matrix $\M{X} = [ \M{X}_{1}, \ldots, \M{X}_{p}] \in \mathbb{R}^{n \times p}$, response vector $\by \in \mathbb{R}^{n \times 1}$, and regression coefficients $\beta \in \mathbb{R}^p$.  We assume herein
that the features $\M{X}_{i}$ have been centered to have zero mean and unit $\ell_{2}$ norm, i.e., $\|\bX_i\|_2 = 1$ for $i=1, \ldots, p$, and $\M{y}$ is also centered to have zero mean.
For a regression coefficient vector $\beta$, the predicted value of the response is given by $\bX\beta$ and $ r = \by - \bX\beta$ denotes the residuals.

\paragraph{Least Squares Boosting -- \ebs} Boosting, when applied in the context of linear regression leads to models with attractive statistical properties~\cite{Friedman99,ESLBook,buehlmann2006boosting,buehlmann08:_boost_algor}.
We begin our study by describing one of the most popular boosting algorithms for linear regression:~\textsc{LS-Boost}$(\varepsilon)$ proposed in~\cite{Friedman99}:

\begin{center}
{\bf{Algorithm:}} Least Squares Boosting -- \textsc{LS-Boost}$(\varepsilon)$
\end{center}

Fix the learning rate $\varepsilon>0$ and the number of iterations $M$.

Initialize at $\hat r^0 = \by$, $\hat{\beta}^0 = 0$, $k = 0$ .

\bee
\item[\bf 1.] For $0 \leq k \leq M$ do the following:
\item[\bf 2.] Find the covariate $j_{k}$ and $\tilde{u}_{j_k}$ as follows:
$$\tilde{u}_{m}  =  \argmin_{u \in \mathbb{R}} \; \left( \sum_{i=1}^{n} ( \hat r^{k}_{i} - x_{im} u)^2  \right) \text{ for } m = 1,\ldots, p, \;\;\;\;\;  j_{k} \in \argmin_{ 1 \leq m \leq p } \; \sum_{i=1}^{n} ( \hat{r}^k_{i} - x_{im} \tilde{u}_{m})^2 \ . $$
\item[\bf 3.] Update the current residuals and regression coefficients as:
\begin{description}
\item $\hat{r}^{k+1} \gets \hat{r}^{k} - \varepsilon \M{X}_{j_k} \tilde{u}_{j_k}$
\item $\hat{\beta}^{k+1}_{j_k} \gets \hat{\beta}^{k}_{j_k} + \varepsilon \tilde{u}_{j_k}$ and
$\hat{\beta}^{k+1}_j \gets \hat{\beta}^{k}_j \ , j \neq j_k$ .
\end{description}
\eee
A special instance of the \textsc{LS-Boost}$(\varepsilon)$ algorithm with $\varepsilon=1$ is known as \textsc{LS-Boost}\cite{Friedman99} or Forward Stagewise~\cite{ESLBook} --- it is essentially
a method of repeated simple least squares fitting of the residuals~\cite{buehlmann08:_boost_algor}.
 The \textsc{LS-Boost} algorithm starts from the null model with residuals $\hat r^0 = \M{y}$.
At the $k$-th iteration, the algorithm finds a covariate $j_{k}$ which results in the maximal decrease in the univariate regression fit to the current residuals.  Let $\M{X}_{j_{k}}\tilde{u}_{j_k}$ denote the \emph{best} univariate fit for the current residuals, corresponding to the covariate $j_{k}$.
The residuals are then
updated as  $\hat{r}^{k+1} \gets \hat r^k - \M{X}_{j_{k}}\tilde{u}_{j_k}$ and the ${j_k}$-th regression coefficient is updated as
$\hat{\beta}^{k+1}_{j_k} \gets \hat{\beta}^{k}_{j_k} +  \tilde{u}_{j_k}$, with all other regression coefficients unchanged.
We refer the reader to Figure~\ref{fig:fse-ebs-loss-1}, depicting the evolution of the algorithmic properties of the \ebs algorithm as a function of $k$ and $\varepsilon$.
\ebs has old roots --- as noted by~\cite{buehlmann08:_boost_algor}, \textsc{LS-Boost} with $M=2$ is known as ``twicing,'' a method proposed by Tukey~\cite{tukey1976exploratory}.

\textsc{LS-Boost}$(\varepsilon)$ is a slow-learning variant of \textsc{LS-Boost}, where to counterbalance the greedy selection strategy of the \emph{best} univariate fit to the current residuals, the updates are
shrunk by an additional factor of $\varepsilon$, as described in Step~3 in Algorithm \textsc{LS-Boost}$(\varepsilon)$.
This additional  shrinkage factor $\varepsilon$ is also known as the learning rate.  Qualitatively speaking, a small value of $\varepsilon$ (for example, $\varepsilon = 0.001$) slows down the learning rate as compared to the 
choice $\varepsilon = 1$.
As the number of iterations increases, the training error decreases until one eventually attains a least squares fit.  For a small value of $\varepsilon$, the number of iterations required to reach a certain training error increases. However, with a small value of $\varepsilon$ it is possible
to explore a larger class of models, with varying degrees of shrinkage. It has been observed empirically that this often leads to models with better predictive power~\cite{Friedman99}.
In short, both $M$ (the number of boosting iterations) and $\varepsilon$ together control
the training error and the amount of shrinkage.  Up until now, as pointed out by~\cite{ESLBook}, the understanding of this tradeoff has been rather qualitative.
One of the contributions of this paper is a precise quantification of this tradeoff, which we do in Section~\ref{LSBsection}.

The papers \cite{buhlmann2003boosting,buehlmann2006boosting,buehlmann08:_boost_algor} present very interesting perspectives on \textsc{LS-Boost}$(\varepsilon)$,
where they refer to the algorithm as $L2$-\textsc{Boost}.  \cite{buehlmann08:_boost_algor} also obtains approximate expressions for the effective degrees of freedom  of the
$L2$-\textsc{Boost} algorithm.    In the non-stochastic setting, this is known as Matching Pursuit~\cite{mallat1993matching}.  \ebs ~is also closely related to Friedman's MART algorithm~\cite{Friedman00greedyfunction}.

\paragraph{Incremental Forward Stagewise Regression -- $\FSe$}  A close cousin of the \textsc{LS-Boost}$(\varepsilon)$ algorithm is the Incremental Forward Stagewise algorithm~\cite{ESLBook,LARS} presented below, which we refer to as $\FSe$.

\begin{center}
{\bf  Algorithm:} Incremental Forward Stagewise Regression -- $\FSe$
\end{center}

Fix the learning rate $\varepsilon > 0$ and number of iterations $M$.

Initialize at $\hat r^0 = \by$, $\hat{\beta}^0 = 0$, $k = 0$ .

\bee
\item[\bf 1.] For $0 \leq k \leq M$ do the following:

\item[\bf 2.] Compute:
$j_k \in \argmax\limits_{j \in \{1, \ldots, p\}} |(\hat{r}^k)^T\bX_j|$

\item[\bf 3.]
\begin{description}
\item $\hat r^{k+1} \gets \hat{r}^k - \varepsilon \ \sgn((\hat{r}^k)^T\bX_{j_k})\bX_{j_k}$
\item $\hat{\beta}^{k+1}_{j_k} \gets \hat{\beta}^{k}_{j_k} + \varepsilon\ \sgn((\hat{r}^k)^T\bX_{j_k})$ and $\hat{\beta}^{k+1}_j \gets \hat{\beta}^{k}_j \ , j \neq j_k$ .
\end{description}
\eee
In this algorithm, at the $k$-th iteration we choose a covariate $\M{X}_{j_k}$ that is the most correlated (in absolute value) with the current
residual and update the $j_{k}$-th regression coefficient, along with the residuals, with a shrinkage factor $\varepsilon$.
As in the \ebs~algorithm, the choice of $\varepsilon$ plays a crucial role in the statistical behavior of the $\FSe$ algorithm.  A large choice of $\varepsilon$ usually means an aggressive strategy; a smaller value corresponds to a slower learning procedure.
Both the parameters $\varepsilon$ and the number of iterations $M$ control the data fidelity and shrinkage in a fashion qualitatively similar to \ebs.
We refer the reader to Figure~\ref{fig:fse-ebs-loss-1}, depicting the evolution of the algorithmic properties of the $\FSe$ algorithm as a function of $k$ and $\varepsilon$.
In Section \ref{sect_subgrad} herein, we will present for the first time precise descriptions of how the quantities $\varepsilon$ and $M$ control the amount of training error and regularization in $\FSe$, which will consequently inform us about their tradeoffs.

Note that \ebs\  and $\FSe$ have a lot of similarities but contain subtle differences too, as we will characterize in this paper. Firstly, since all of the covariates are standardized to have unit $\ell_{2}$ norm, for same given residual value $\hat r^k$ it is simple to derive that
Step (2.) of \ebs and $\FSe$ lead to the same choice of $j_{k}$. However, they are not the same algorithm and their differences are rather plain to see from their residual updates, i.e., Step (3.).
In particular, the amount of change in the successive residuals differs across the algorithms:
\begin{equation}\label{ebs-fse-together-2}
\begin{myarray}[1.5]{rl}
\text{\ebs}: & \| \hat r^{k+1} - \hat{r}^k\|_{2} =   \varepsilon |(\hat{r}^k)^T\bX_{j_k} | =  \varepsilon \cdot n\cdot \|\nabla L_n(\hat \beta^k)\|_\infty \  \\
\FSe:  &  \| \hat r^{k+1} - \hat{r}^k\|_{2} = \varepsilon |s_{k}| ~~~ \text{where}~~s_{k} = \sgn((\hat{r}^k)^T\bX_{j_k}) \ ,
\end{myarray}
\end{equation}
where $\nabla L_n(\cdot)$ is the gradient of the least squares loss function $L_n(\beta) := \tfrac{1}{2n}\|\by - \bX\beta \|_2^2$.
Note that for both of the algorithms, the quantity $\| \hat r^{k+1} - \hat{r}^k\|_{2}$ involves the shrinkage factor $\varepsilon$. Their difference thus lies in
the multiplicative factor, which is $ n\cdot \|\nabla L_n(\hat \beta^k)\|_\infty$ for \ebs and is
$|\sgn((\hat{r}^k)^T\bX_{j_k})|$
for $\FSe$. The norm of the successive residual differences for \ebs is proportional to the $\ell_\infty$ norm of the gradient of the least squares loss function (see herein equations \eqref{grad1} and \eqref{grad_norm}).
For $\FSe$, the norm of the successive residual differences depends on the absolute value of the sign of the $j_{k}$-th coordinate of the gradient.  Note that $s_{k} \in \{ - 1, 0, 1\}$ depending upon whether $(\hat{r}^k)^T\bX_{j_k}$ is negative, zero, or positive; and $s_{k}=0$ only when $(\hat{r}^k)^T\bX_{j_k}=0$, i.e., only when $\|\nabla L_n(\hat \beta^k)\|_\infty = 0$ and hence $\hat \beta^k$ is a least squares solution.  Thus, for $\FSe$ the $\ell_{2}$ norm of the difference in residuals is almost always $\varepsilon$ during the course of the algorithm.
For the \ebs algorithm, progress is considerably more sensitive to the norm of the gradient --- as the algorithm makes its way to the unregularized least squares
fit, one should expect the norm of the gradient to also shrink to zero, and indeed we will prove this in precise terms in Section \ref{LSBsection}.
Qualitatively speaking, this means that the updates of \ebs are more well-behaved when compared to the
updates of $\FSe$, which are more erratically behaved.  Of course, the additional shrinkage factor $\varepsilon$ further dampens the progress for both algorithms.

Our results in Section~\ref{LSBsection} show that the predicted values $\M{X} \hat{\beta}^k$ obtained from \ebs\ converge (at a globally linear rate) to the least squares fit as $k \rightarrow \infty$, this holding true for any value of $\varepsilon \in (0, 1]$.  On the other hand, for $\FSe$ with $\varepsilon >0$, the iterates $\M{X} \hat{\beta}^k$  need not necessarily converge to the least squares
fit as $k \rightarrow \infty$.  Indeed, the $\FSe$ algorithm, by its operational definition, has a uniform learning rate $\varepsilon$ which remains fixed for all iterations; this makes it impossible to always guarantee convergence to a least squares solution with accuracy less than $O(\varepsilon)$.
While the predicted values of \ebs converge to a least squares solution at a linear rate, we show in Section
\ref{sect_subgrad} that the predictions from the $\FSe$ algorithm
converges to an approximate least squares solution, albeit at a global sublinear rate.\footnote{For the purposes of this paper, linear convergence of a sequence $\{a_i\}$ will mean that $a_i \rightarrow \bar a$ and there exists a scalar $\gamma <1$ for which $(a_i - \bar a )/(a_{i-1} - \bar a) \le \gamma$ for all $i$.  Sublinear convergence will mean that there is no such $\gamma <1$ that satisfies the above property.  For much more general versions of linear and sublinear convergence, see \cite{bertsekas} for example.}

Since the main difference between $\FSe$ and \ebs lies in the choice of the step-size used to update the coefficients, let us therefore consider a non-constant step-size/non-uniform learning rate version of $\FSe$, which we call $\FSek$.
$\FSek$ replaces Step~3 of $\FSe$ by:
\begin{description}
\item[residual update:] $\hat r^{k+1} \gets \hat{r}^k - \varepsilon_{k} \ \sgn((\hat{r}^k)^T\bX_{j_k})\bX_{j_k}$
\item[coefficient update:] $\hat{\beta}^{k+1}_{j_k} \gets \hat{\beta}^{k}_{j_k} + \varepsilon_{k}\ \sgn((\hat{r}^k)^T\bX_{j_k})$ and $\hat{\beta}^{k+1}_j \gets \hat{\beta}^{k}_j \ , j \neq j_k$ ,
\end{description}
where $\{\varepsilon_{k}\}$ is a sequence of learning-rates (or step-sizes) which depend upon the iteration index $k$. \ebs~can thus
be thought of as a version of $\FSek$, where the step-size $\varepsilon_k$ is given by $\varepsilon_{k} :=  \varepsilon  \tilde{u}_{j_k}\sgn((\hat{r}^k)^T\bX_{j_k})$.

In Section~\ref{boostingassubgrad} we provide a unified treatment of \ebs, $\FSe$, and $\FSek$, wherein we show that all these methods can be viewed as special instances of (convex) subgradient optimization.  For another perspective on the similarities and differences between $\FSe$ and \ebs, see~\cite{buehlmann08:_boost_algor}.

\begin{figure}[h]
\centering
\scalebox{.90}[.7]{\begin{tabular}{l c c c}
& \small{$\rho=0$}  &\small{$\rho=0.5$}  & \small{$\rho=0.9$}   \\

\rotatebox{90}{\sf {\scriptsize {~~~~~~~~~~~~~~~~~~~~~~~Training Error}}}&\includegraphics[width=0.31\textwidth,height=0.3\textheight,  trim = 1.1cm 1.2cm 1cm 2cm, clip = true ]{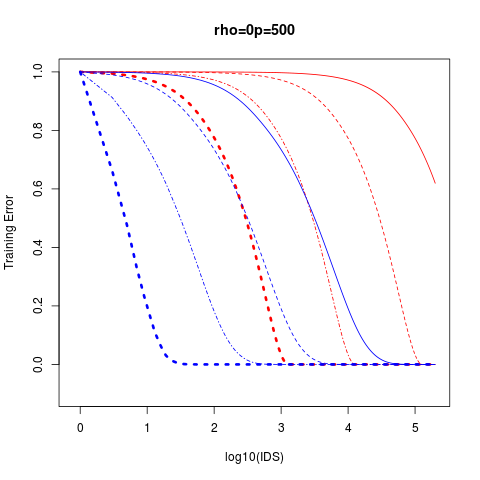}&
\includegraphics[width=0.31\textwidth,height=0.3\textheight,  trim = 2cm 1.2cm 1cm 2cm, clip = true ]{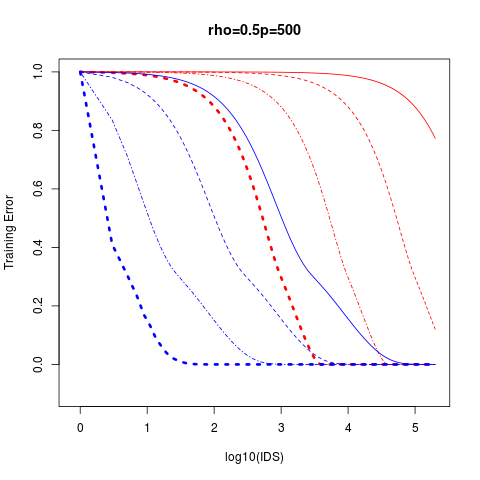}&
\includegraphics[width=0.31\textwidth,height=0.3\textheight,  trim = 2cm 1.2cm 1cm 2cm, clip = true ]{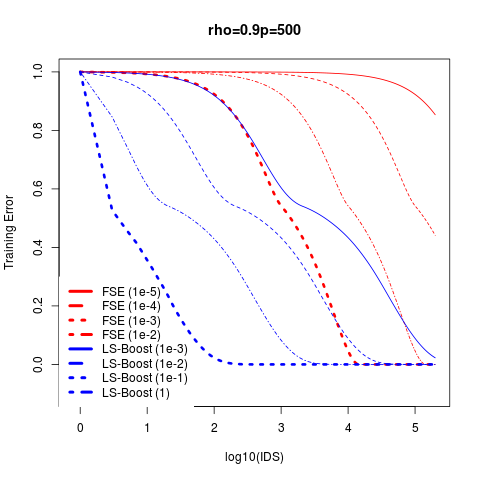} \\

\rotatebox{90}{\sf {\scriptsize{~~~~~~~~~~~~~~~$\ell_{1}$ norm of Coefficients}}}&\includegraphics[width=0.31\textwidth,height=0.3\textheight,  trim = 1.1cm 1.2cm 1cm 2cm, clip = true ]{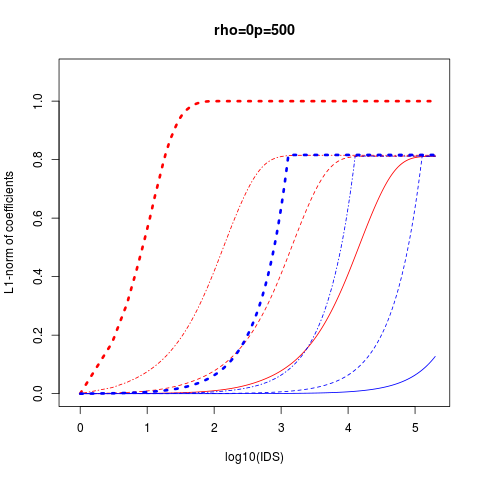}&
\includegraphics[width=0.31\textwidth,height=0.3\textheight,  trim = 2cm 1.2cm 1cm 2cm, clip = true ]{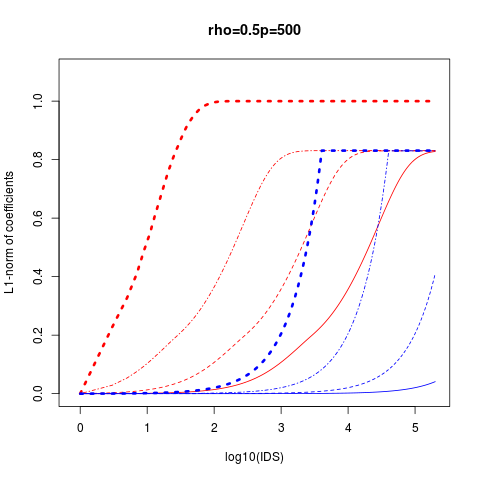}&
\includegraphics[width=0.31\textwidth,height=0.3\textheight,  trim = 2cm 1.2cm 1cm 2cm, clip = true ]{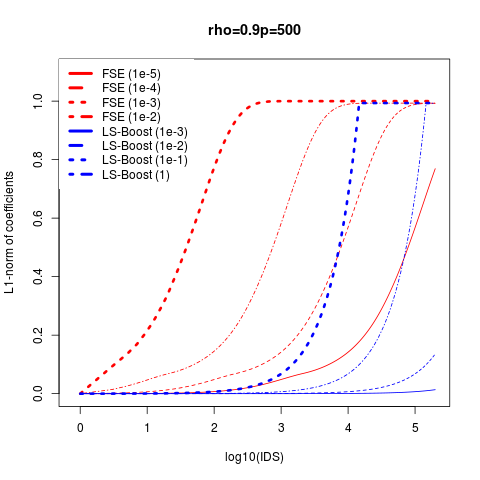} \\

&\multicolumn{3}{c}{\small{\sf{$\log_{10}(\text{Number of Boosting Iterations})$}}} \\
\end{tabular}}
\caption{{\small{Evolution of \ebs and $\FSe$ versus iterations (in the log-scale), run on a synthetic dataset with $n=50$, $p=500$; 
the covariates are drawn from a Gaussian distribution with pairwise correlations $\rho$.  The true $\beta$ has ten non-zeros with $\beta_{i} =1, i \leq 10$ and $\text{SNR}=1$. Several different values of $\rho$ and $\varepsilon$ have been considered. 
[Top Row] Shows the training errors for different learning rates, [Bottom Row] shows the $\ell_{1}$ norm
of the coefficients produced by the different algorithms for different learning rates (here the values have all been re-scaled so that the y-axis lies in $[0,1]$). For detailed discussions about the figure, see the main text.
}} }\label{fig:fse-ebs-loss-1}
\end{figure}

Both \ebs ~and $\FSe$ may be interpreted as  ``cautious'' versions of Forward Selection or Forward Stepwise regression~\cite{miller2002subset,W80}, a classical variable selection tool used widely in applied statistical modeling.
Forward Stepwise regression builds a model sequentially by adding one variable at a time.  At every stage, the algorithm
identifies the variable most correlated (in absolute value) with the current residual, includes it in the model, and updates the \emph{joint least squares} fit based on the
current set of predictors. This aggressive update procedure, where all of the coefficients in the active set are simultaneously updated, is what makes stepwise regression quite different from $\FSe$ and \ebs --- in the latter algorithms only one variable
is updated (with an additional shrinkage factor) at every iteration.

\paragraph{Explicit Regularization Schemes}
While all the methods described above are known to deliver regularized models, the nature of regularization imparted by the algorithms are rather implicit.
To highlight the difference between an implicit and explicit regularization scheme, consider $\ell_{1}$-regularized regression, namely  \lasso \cite{Tibshirani1996LASSO}, which is an extremely
popular method especially for high-dimensional linear regression, i.e., when the number of parameters far exceed the number of samples.
The \lasso performs both variable selection and shrinkage in the regression coefficients, thereby leading to parsimonious models with good predictive performance.
The constraint version of \lasso with regularization parameter $\delta \geq 0$ is given by the following convex quadratic optimization problem:
\begin{equation}\label{poi-lasso}
\begin{array}{rccl}
\mathrm{\lasso} : \ \ \ L_{n, \delta}^\ast := & \min\limits_{\beta}   & \frac{1}{2n}\|\by - \bX\beta\|_2^2\\
&\mathrm{s.t.} & \|\beta\|_1 \leq \delta \ .
\end{array}
\end{equation}
The nature of regularization via the \lasso is explicit --- by its very formulation, it is set up to find the best least squares solution subject to a constraint on the $\ell_{1}$ norm of
the regression coefficients. This is in contrast to boosting algorithms like $\FSe$ and \ebs, wherein regularization is imparted implicitly as a consequence of the structural properties of the algorithm with $\varepsilon$ and $M$ controlling the amount of shrinkage.

\paragraph{Boosting and Lasso} Although \lasso and the above boosting methods originate from different perspectives, there are interesting similarities between the two as
nicely explored in~\cite{ESLBook,LARS,hastie06:_forwar}.

For certain datasets the coefficient profiles\footnote{By a coefficient profile we mean the map $\lambda \mapsto \hat{\beta}_{\lambda}$ where, $\lambda \in \Lambda$ indexes a family of coefficients $\hat{\beta}_{\lambda}$. For example, the family of \lasso solutions~\eqref{poi-lasso} $\{\hat{\beta}_{\delta}, \delta \geq 0\}$ indexed by $\delta$ can also be indexed by the $\ell_{1}$ norm of the coefficients, i.e., $\lambda = \| \hat{\beta}_{\delta}\|_{1}$. This leads to a coefficient profile that depends upon the $\ell_{1}$ norm of the regression coefficients.
Similarly, one may consider the coefficient profile of $\FS_0$ as a function of the $\ell_1$ norm of the regression coefficients delivered by the $\FS_0$ algorithm.} of \lasso and $\FS_0$ are exactly the same~\cite{ESLBook}, where $\FS_0$ denotes the limiting case of the $\FSe$ algorithm as $\varepsilon \rightarrow 0+$. Figure~\ref{fig:lasso-similar} (top panel) shows an example where the \lasso profile is similar to those of $\FSe$ and \ebs (for small values of $\varepsilon$).
However, they are different in general (Figure~\ref{fig:lasso-similar}, bottom panel). Under some conditions on the monotonicity of the coefficient profiles of the \lasso solution,
the \lasso and $\FS_0$ profiles are exactly the same~\cite{LARS,hastie06:_forwar}.
Such equivalences exist for more general loss functions~\cite{rosset02:_boost_regul_path_maxim_margin_class}, albeit under fairly strong assumptions on problem data.

Efforts to understand boosting algorithms in general and in particular the
$\FSe$ algorithm  paved the way for the celebrated Least Angle Regression aka the \LAR algorithm~\cite{LARS} (see also~\cite{ESLBook}).
  The \LAR algorithm is a democratic version of Forward Stepwise.
  Upon identifying the variable most correlated with the current residual in absolute value (as in Forward Stepwise), it moves the coefficient
of the variable towards its least squares value in a continuous fashion.
An appealing aspect of the \LAR algorithm is that it provides a unified algorithmic framework for variable selection and shrinkage -- one instance of \LAR leads to a path algorithm for the \lassoperiod, and a different instance leads to the limiting case of the $\FSe$ algorithm as $\varepsilon \rightarrow 0+$, namely $\FS_0$. In fact, the \emph{Stagewise} version of the \LAR algorithm provides an efficient
way to compute the coefficient profile for $\FS_0$.

\begin{figure}[h!]
\centering
\scalebox{0.9}[.7]{\begin{tabular}{l c cc c c}
\multicolumn{6}{c}{Coefficient Profiles: \ebs, $\FSe$ and Lasso \medskip}\\
&\small{\sf { \lasso}} && \small{\sf{ \ebs, $\varepsilon=0.01$ } } && \small{ \sf{ $\FSe$, $\varepsilon=10^{-5}$} } \\
\rotatebox{90}{\sf {\scriptsize {~~~~~~~~~~~~~~~Regression Coefficients}}}&\includegraphics[width=0.31\textwidth,height=0.25\textheight,  trim = 1.0cm 1.2cm 1cm 1.9cm, clip = true ]{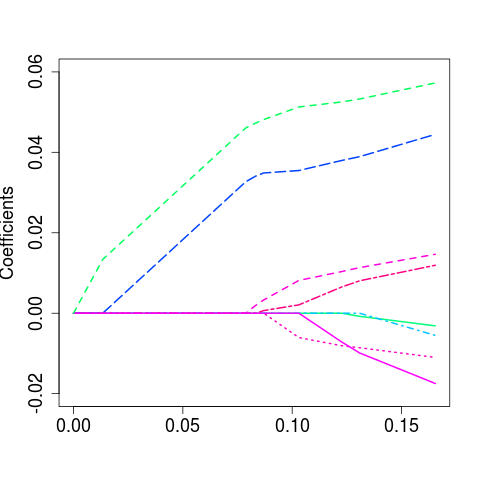}&&
\includegraphics[width=0.31\textwidth,height=0.25\textheight,  trim = 1.8cm 1.2cm 1cm 2cm, clip = true ]{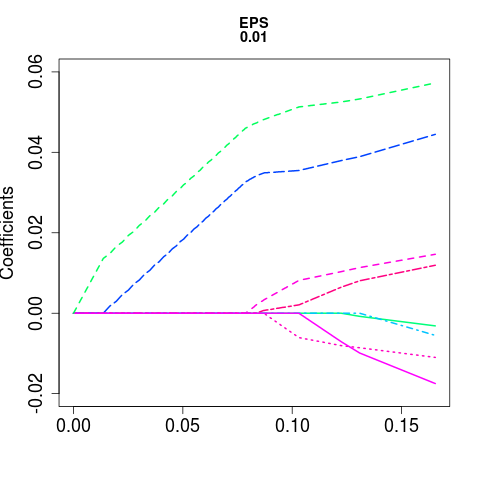}&&
\includegraphics[width=0.31\textwidth,height=0.25\textheight,  trim = 1.8cm 1.2cm 1cm 2cm, clip = true ]{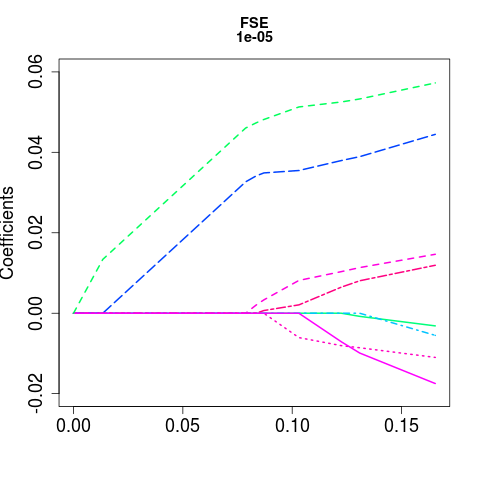} \\

\rotatebox{90}{\sf {\scriptsize {~~~~~~~~~~~~~~~Regression Coefficients}}}&\includegraphics[width=0.31\textwidth,height=0.25\textheight,  trim = 1.0cm 1.2cm 1cm 1.9cm, clip = true ]{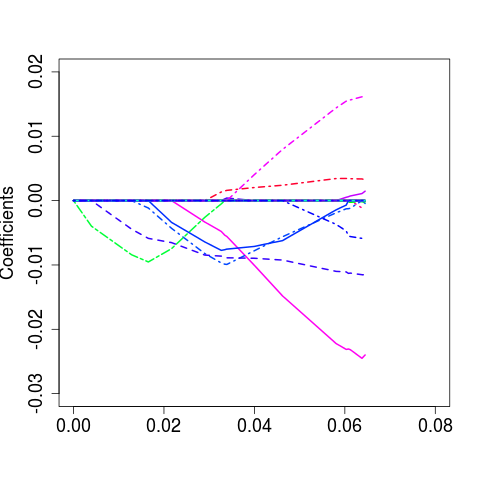}&&
\includegraphics[width=0.31\textwidth,height=0.25\textheight,  trim = 1.8cm 1.2cm 1cm 2cm, clip = true ]{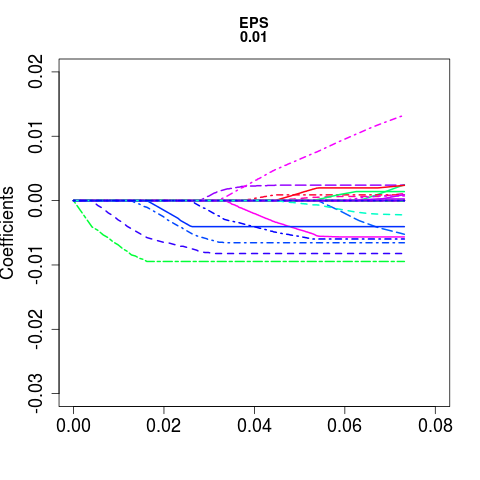}&&
\includegraphics[width=0.31\textwidth,height=0.25\textheight,  trim = 1.8cm 1.2cm 1cm 2cm, clip = true ]{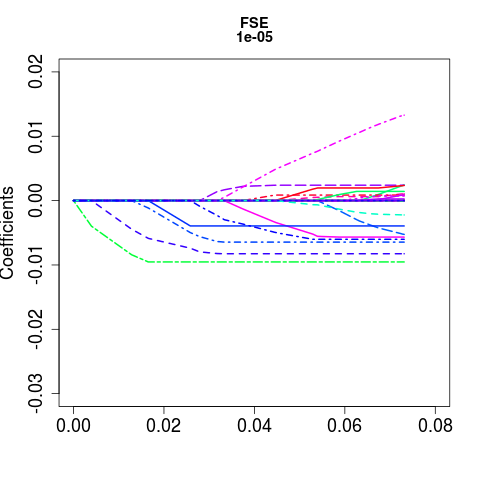} \\
&\scriptsize{\sf {$\ell_{1}$ shrinkage of coefficients}} &&\scriptsize{\sf {$\ell_{1}$ shrinkage of coefficients}}  && \scriptsize{\sf {$\ell_{1}$ shrinkage of coefficients}}  \\
\end{tabular}}
\caption{{\small{Coefficient Profiles for different algorithms as a function of the $\ell_{1}$ norm of the regression coefficients on two different datasets.
 [Top Panel] Corresponds to the full  Prostate Cancer dataset described in Section~\ref{sec:experiments} with $n=98$ and $p=8$.
All the  coefficient profiles look similar.
[Bottom Panel] Corresponds to a subset of samples of the Prostate Cancer dataset  with $n=10$; we also included all second order interactions to get $p=44$.
The coefficient profile of \lasso is seen to be different from $\FSe$ and \ebs.
Figure~\ref{fig:bounds-lasso} shows the training error {\it vis-\`{a}-vis} the $\ell_{1}$-shrinkage of the models, for the same profiles. }} }\label{fig:lasso-similar}
\end{figure}
Due to the close similarities between the \lasso and boosting coefficient profiles, it is natural to investigate probable modifications of boosting that might lead to the \lasso solution path.
This is one of the topics we study in this paper. In a closely related but different line of approach, \cite{zhao2007stagewise} describes \textsc{BLasso}, a modification of  the $\FSe$ algorithm with the inclusion of additional ``backward steps'' so that the resultant coefficient profile mimics the \lasso path.

\paragraph{Subgradient Optimization as a Unifying Viewpoint of Boosting and Lasso}
 In spite of the various nice perspectives on $\FSe$ and its connections to the \lasso as described above,
the present understanding about the relationships between \lassoperiod, $\FSe$, and \ebs for arbitrary datasets and $\varepsilon>0$ is still fairly limited.  One of the aims of this paper is to contribute some substantial further understanding of the relationship between these methods.
Just like the \LAR algorithm can be viewed as a master algorithm with special instances being the \lasso and $\FS_0$, in this paper we establish that
$\FSe$, \ebs and \lasso can be viewed as special instances of one grand algorithm: the subgradient descent method (of convex optimization) applied to the following parametric class of optimization problems:
\begin{equation}\label{dual-lasso-1-0}
P_\delta \ : \ \  \mini_{r} \;\;\; \| \M{X}^T r \|_{\infty} + \frac{1}{2\delta} \| {r} - \M{y} \|_{2}^2 \ \  \;\;\; \mathrm{where} \ {r} = \M{y} - \M{X} \beta \; \  \mathrm{for~some~} \beta \ ,
 \end{equation} and where $\delta \in (0,\infty]$ is a regularization parameter.  Here the first term is the maximum absolute correlation between the features $\M{X}_i$ and the residuals $r$, and the second term is a regularization term that penalizes residuals that are far from the observations $\by$ (which itself can be interpreted as the residuals for the null model $\beta = 0$).  The parameter $\delta$ determines the relative importance assigned to the regularization term, with $\delta=+\infty$ corresponding to no importance whatsoever.
As we describe in Section~\ref{sect_lasso}, Problem~\eqref{dual-lasso-1-0} is in fact a dual of the \lasso Problem~\eqref{poi-lasso}.

 The subgradient descent algorithm applied to Problem~\eqref{dual-lasso-1-0} leads to a new boosting algorithm that is almost identical to $\FSe$.
We denote this algorithm by $\RFSe$ (for Regularized incremental Forward Stagewise regression).
We show the following properties of the new algorithm $\RFSe$:
\begin{itemize}
\item $\RFSe$ is almost identical to $\FSe$, except that it first shrinks all of the coefficients of $\hat \beta^k$ by a scaling factor $1 - \tfrac{\varepsilon}{\delta} < 1$ and then updates the selected coefficient $j_k$ in the same additive fashion as $\FSe$.
\item as the number of iterations become large, $\RFSe$ delivers an approximate \lasso solution.
\item an adaptive version of $\RFSe$, which we call $\PATHe$, is shown to approximate the path of \lasso solutions with precise bounds that quantify the approximation error over the path.

\item $\RFSe$ specializes to $\FSe$, \ebs and the \lasso depending on the parameter value $\delta$ and the learning rates (step-sizes) used therein.

\item the computational guarantees derived herein for $\RFSe$ provide a precise description of the evolution of data-fidelity { \it vis-\`{a}-vis } $\ell_{1}$ shrinkage of the models obtained along the boosting iterations.
\item in our experiments, we observe that $\RFSe$ leads to models with statistical properties that compare favorably with the \lasso and $\FSe$. It also leads to models that are sparser than
$\FSe$.

\end{itemize}
We emphasize that all of these results apply to the finite sample setup
with no assumptions about the dataset nor about the relative sizes of $p$ and $n$.

\paragraph{Contributions}
A summary of the contributions of this paper is as follows:
\begin{enumerate}
\item We analyze several boosting algorithms popularly used in the context of linear regression via the lens of first-order methods in convex optimization.  We show that existing boosting algorithms, namely $\FSe$ and \ebs, can be viewed as instances of the subgradient descent method aimed at minimizing the maximum absolute correlation between the covariates and residuals, namely $\|\M{X}^T r \|_{\infty}$.  This viewpoint provides several insights about the operational characteristics of these boosting algorithms.

\item We derive novel computational guarantees for $\FSe$ and \ebs. These results quantify the rate at which the estimates produced by a boosting algorithm make their way towards an unregularized least squares fit  (as a function of the number of iterations and the learning rate $\varepsilon$). In particular, we demonstrate that for \emph{any} value of $\varepsilon \in (0,1]$ the estimates produced by \ebs converge linearly to their respective least squares values and
 the $\ell_{1}$ norm of the coefficients grows at a rate $O(\sqrt{\varepsilon k})$.
$\FSe$ on the other hand demonstrates a slower sublinear convergence rate to an $O(\varepsilon)$-approximate least squares solution, while
 the $\ell_{1}$ norm of the coefficients grows at a rate $O(\varepsilon k)$.

\item Our computational guarantees yield precise characterizations
of the amount of data-fidelity (training error) and regularization imparted by running a boosting algorithm for $k$ iterations. These results apply to any dataset and do not rely upon
any distributional or structural assumptions on the data generating mechanism.

\item We show that subgradient descent applied to a
regularized version of the loss function $\|\M{X}^Tr \|_{\infty}$, with regularization parameter $\delta$, leads to a new algorithm which we call $\RFSe$, that is a natural and simple generalization of
$\FSe$.  When compared to $\FSe$, the algorithm $\RFSe$ performs a seemingly minor rescaling of the coefficients at every iteration.  As the number of iterations $k$ increases, $\RFSe$ delivers an approximate \lasso solution~\eqref{poi-lasso}.
Moreover, as the algorithm progresses, the $\ell_{1}$ norms of the coefficients evolve as a geometric series towards the regularization parameter value $\delta$.  We derive precise computational guarantees that inform us about the training error and regularization imparted by $\RFSe$.

 \item We present an adaptive extension of $\RFSe$, called $\PATHe$, that delivers a path of approximate \lasso solutions for any prescribed grid sequence of regularization parameters.  We derive guarantees that quantify the average distance from the approximate path traced by $\PATHe$ to the \lasso solution path.
\end{enumerate}

\paragraph{Organization of the Paper}

The paper is organized as follows.  In Section \ref{LSBsection} we analyze the convergence behavior of the \ebs algorithm.
In Section \ref{sect_subgrad} we present a unifying algorithmic framework for $\FSe$, $\FSek$, and \ebs as subgradient descent.
In Section~\ref{sect_lasso} we present the regularized correlation minimization Problem~\eqref{dual-lasso-1-0} and a naturally associated boosting algorithm $\RFSe$, as instantiations of
subgradient descent on the family of Problems~\eqref{dual-lasso-1-0}.
In each of the above cases, we present precise computational guarantees of the algorithms for convergence of residuals, training errors, and shrinkage
and study their statistical implications.
In Section \ref{dynamic_sect}, we further expand $\RFSe$ into a method for computing approximate solutions of the \lasso path.
Section~\ref{sec:experiments} contains computational experiments. To improve readability, most of the technical details have been placed in the Appendix~\ref{sec-appendix}.

\subsection*{Notation}
For a vector $x \in \mathbb{R}^m$, we use $x_i$ to denote the $i$-th coordinate of $x$.  We use superscripts to index vectors in a sequence $\{x^k\}$. Let $e_j$ denote the $j$-th unit vector in $\mathbb{R}^m$, and let $e = (1, \ldots, 1)$ denote the vector of ones.  Let $\|\cdot\|_q$ denote the $\ell_q$ norm for $q \in [1, \infty]$ with unit ball $B_q$, and let $\|v\|_0$ denote the number of non-zero coefficients of the vector $v$. For $A \in \mathbb{R}^{m \times n}$, let $\|A\|_{q_1, q_2} := \max\limits_{x : \|x\|_{q_1} \leq 1}\|Ax\|_{q_2}$ be the operator norm. In particular, $\|A\|_{1,2} = \max(\|A_1\|_2, \ldots, \|A_n\|_2)$ is the maximum $\ell_2$ norm of the columns of $A$. For a scalar $\alpha$, $\sgn(\alpha)$ denotes the sign of $\alpha$.  The notation ``$\tilde v \leftarrow \argmax\limits_{v  \in S} \{f(v)\}$'' denotes assigning $\tilde v$ to be any optimal solution of the problem $\max\limits_{v  \in S} \{f(v)\}$.  For a convex set $P$ let $\Pi_P(\cdot)$ denote the Euclidean projection operator onto $P$, namely $\Pi_P(\bar x):=\argmin_{x \in P} \|x-\bar x\|_2$. Let $\partial f(\cdot)$ denote the subdifferential operator of a convex function $f(\cdot)$. If $Q \ne 0$ is a symmetric positive semidefinite matrix, let $\lambda_{\max}(Q)$, $\lambda_{\min}(Q)$, and $\lambda_{\pmin}(Q)$ denote the largest, smallest, and smallest nonzero (and hence positive) eigenvalues of $Q$, respectively.

\section{\ebs: Computational Guarantees and Statistical Implications}\label{LSBsection}

\paragraph{Roadmap} We begin our formal study by examining the \ebs algorithm. We study the rate at which the coefficients generated by \ebs converge to the set of unregularized least square solutions.
This characterizes the amount of data-fidelity as a function of the
number of iterations and $\varepsilon$.  In particular, we show (global) linear convergence of the regression coefficients to the set of least squares coefficients, with similar convergence rates derived for the prediction estimates and the boosting training errors delivered by \ebs.
We also present bounds on the shrinkage of the regression coefficients $\hat{\beta}^k$ as a function of $k$ and $\varepsilon$, thereby describing how
the amount of shrinkage of the regression coefficients changes as a function of the number of iterations $k$.

\subsection{Computational Guarantees and Intuition}\label{sec:ebs-bounds}
We first review some useful properties associated with the familiar least squares regression problem:
\begin{equation}\label{poi-ls2}
\begin{array}{rccl}
\mathrm{LS} : \ \ \ L_n^* := & \min\limits_{\beta}   & L_n(\beta) := \frac{1}{2n}\|\by - \bX\beta\|_2^2\\
&\mathrm{s.t.} & \beta\ \in \mathbb{R}^p \ ,
\end{array}
\end{equation}
where $L_n(\cdot)$ is the least squares loss, whose gradient is:
\begin{equation}\label{grad1}
\nabla L_n(\beta) = -\tfrac{1}{n}\bX^T(\by - \bX\beta) = -\tfrac{1}{n}\bX^Tr \
\end{equation}
where $r = \by - \bX \beta$ is the vector of residuals corresponding to the regression coefficients $\beta$.  It follows that $\beta$ is a least-squares solution of $\mathrm{LS}$ if and only if $\nabla L_n(\beta) = 0$, which leads to the well known normal equations:
\begin{equation}\label{grad2}
0 =  -\bX^T(\by - \bX\beta) = -\bX^Tr \ .
\end{equation}
It also holds that:
\begin{equation}\label{grad_norm}
n \cdot \|\nabla L_n(\beta)\|_\infty = \|\bX^Tr\|_\infty = \max\limits_{j \in \{1, \ldots, p \}}\{|r^T\bX_j | \} \ .
\end{equation}

The following theorem describes precise computational guarantees for \textsc{LS-Boost}$(\varepsilon)$: linear convergence of \ebs with respect to \eqref{poi-ls2}, and bounds on the $\ell_1$ shrinkage of the coefficients produced.  Note that the theorem uses the quantity $\lambda_{\pmin}(\bX^T\bX)$ which denotes the smallest nonzero (and hence positive) eigenvalue of $\bX^T\bX$.\medskip

\begin{theorem}\label{august24-2} {\bf (Linear Convergence of \ebs for Least Squares)}  Consider the \ebs algorithm with learning rate $\varepsilon \in (0,1]$, and define the linear convergence rate coefficient $\gamma$:
\begin{equation}\label{gamma_def}
\gamma := \left(1-\frac{\varepsilon(2-\varepsilon)\lambda_{\pmin}(\bX^T\bX)}{4p} \right) < 1 \ .
\end{equation}
For all $k \geq 0$ the following bounds hold:
\begin{itemize}
\item[(i)] (training error): $L_n(\hat{\beta}^k) - L_n^* \leq \frac{1}{2n}\|\bX\hat{\beta}_{\text{LS}}\|_2^2 \cdot \gamma^k$ \\
\item[(ii)] (regression coefficients): there exists a least squares solution $\hat{\beta}^k_{LS}$ such that:
 $$\|\hat{\beta}^k-\hat{\beta}^k_{LS}\|_2  \le \frac{\|\bX\hat{\beta}_{\text{LS}}\|_2}{\sqrt{\lambda_{\pmin}(\bX^T\bX)}}\cdot \gamma^{k/2}$$
\item[(iii)] (predictions):  for every least-squares solution $\hat{\beta}_{\text{LS}}$ it holds that $$\|\bX\hat{\beta}^k - \bX \hat{\beta}_{\text{LS}}\|_2 \le \|\bX\hat{\beta}_{\text{LS}}\|_2 \cdot \gamma^{k/2}$$
\item[(iv)] (gradient norm/correlation values): $\|\nabla L_n(\hat \beta^k)\|_\infty = \tfrac{1}{n}\|\bX^T\hat{r}^k\|_\infty \leq \tfrac{1}{n}\|\bX\hat{\beta}_{\text{LS}}\|_2 \cdot \gamma^{k/2}$ \\
\item[(v)] ($\ell_{1}$-shrinkage of coefficients):
$$\|\hat{\beta}^k\|_1 \le  \min\left\{\sqrt{k}\sqrt{\tfrac{\varepsilon}{2-\varepsilon}}\sqrt{\|\bX\hat{\beta}_{\text{LS}}\|^2_2 - \|\bX\hat{\beta}_{\text{LS}}-\bX\hat{\beta}^k\|^2_2}\  \ , \ \frac{\varepsilon\|\bX\hat\beta_{LS}\|_2}{1 - \sqrt{\gamma}}\left(1 - \gamma^{k/2}\right)\right\} $$
\item[(vi)] (sparsity of coefficients): $\|\hat{\beta}^k\|_0 \le k$. \qed
\end{itemize}
\end{theorem}
Before remarking on the various parts of Theorem \ref{august24-2}, we first discuss the quantity $\gamma$ defined in \eqref{gamma_def}, which is called the linear convergence rate coefficient.  We can write $\gamma = 1-\tfrac{\varepsilon(2-\varepsilon)}{4\kappa(\bX^T\bX)}$ where $\kappa(\bX^T\bX)$ is defined to be the ratio $\kappa(\bX^T\bX):=\tfrac{p}{\lambda_{\pmin}(\bX^T\bX)}$.  Note that $\kappa(\bX^T\bX) \in [1, \infty)$.  To see this, let $\tilde{\beta}$ be an eigenvector associated with the largest eigenvalue of $\bX^T\bX$, then:
\begin{equation}\label{eigen_bound}
0 < \lambda_{\pmin}(\bX^T\bX) \leq \lambda_{\max}(\bX^T\bX) = \frac{\|\bX\tilde{\beta}\|_2^2}{\|\tilde{\beta}\|_2^2} \leq \frac{\|\bX\|_{1,2}^2\|\tilde{\beta}\|_1^2}{\|\tilde{\beta}\|_2^2} \leq p \ ,
\end{equation}
where the last inequality uses our assumption that the columns of $\bX$ have been normalized (whereby $\|\bX\|_{1,2} = 1$), and the fact that $\|\tilde{\beta}\|_1 \leq \sqrt{p}\|\tilde{\beta}\|_2$. This then implies that $\gamma \in [0.75,1.0)$ -- independent of any assumption on the dataset -- and most importantly it holds that $\gamma <1$.\medskip

Let us now make the following immediate remarks on Theorem \ref{august24-2}:
\begin{itemize}
\item The bounds in parts {\em (i)-(iv)} state that the training errors, regression coefficients, predictions, and correlation values produced by \ebs converge linearly (also known as geometric or exponential convergence) to their least squares counterparts: they decrease by at least the constant multiplicative factor $\gamma < 1$ for part {\em (i)}, and by $\sqrt{\gamma}$ for parts {\em (ii)-(iv)}, at every iteration. The bounds go to zero at this linear rate as $k \rightarrow \infty$.

\item The computational guarantees in parts {\em (i) - (vi)} provide characterizations of the data-fidelity and shrinkage of the \ebs\ algorithm for any given specifications of the learning rate $\varepsilon$ and the number of boosting iterations $k$. Moreover, the quantities appearing in the bounds can be computed from simple characteristics of the data that can be obtained {\it a priori} without even running the boosting algorithm.  (And indeed, one can even substitute $\|\by\|_2$ in place of $\|\bX\hat{\beta}_{\text{LS}} \|_2$ throughout the bounds if desired since $\|\bX\hat\beta_{\text{LS}} \|_2 \le \|\by\|_2$.)
\end{itemize}

\paragraph{Some Intuition Behind Theorem~\ref{august24-2}} Let us now study the \ebs algorithm and build intuition regarding its progress with respect to solving the unconstrained least squares problem \eqref{poi-ls2}, which will inform the results in Theorem \ref{august24-2}. Since the predictors are all standardized to have unit $\ell_2$ norm, it follows that the coefficient index $j_k$ and corresponding step-size $\tilde{u}_{j_k}$ selected in Step (2.) of \ebs satisfy:
\begin{equation}\label{choose-jk-1}
j_k \in \argmax\limits_{j \in \{1, \ldots, p\}} |(\hat{r}^k)^T\bX_j| \ \ \ \ \ \ \ \mbox{and} \ \ \ \ \ \ \ \tilde{u}_{j_k} = (\hat{r}^k)^T\bX_{j_k} \ .
\end{equation}
Combining \eqref{grad_norm} and \eqref{choose-jk-1}, we see that
\begin{equation}\label{drew2}
|\tilde{u}_{j_k}| = |(\hat{r}^k)^T\bX_{j_k}| = n \cdot \|\nabla L_n(\hat{\beta}^k)\|_\infty \ .
\end{equation}
Using the formula for $\tilde{u}_{j_k}$ in \eqref{choose-jk-1}, we have the following convenient way to express the change in residuals at each iteration of \textsc{LS-Boost}$(\varepsilon)$:
\begin{equation}\label{ebs_res_update}
\hat{r}^{k+1} = \hat{r}^k -\varepsilon \left((\hat{r}^k)^T\bX_{j_k} \right)\bX_{j_k} \ .
\end{equation}
Intuitively, since \eqref{ebs_res_update} expresses $\hat{r}^{k+1}$ as the difference of two correlated variables, $\hat r^k$ and $\sgn((\hat{r}^k)^T\bX_{j_k})\bX_{j_k}$, we expect the squared $\ell_2$ norm of $\hat r^{k+1}$ (i.e. its sample variance) to be smaller than that of $\hat r^k$. On the other hand, as we see from \eqref{ebs-fse-together-2}, convergence of the residuals is ensured by the dependence of the change in residuals on $|(\hat{r}^k)^T\bX_{j_k}|$, which goes to 0 as we approach a least squares solution. In the proof of Theorem \ref{august24-2} in Appendix \ref{ebs_main_proof} we make this intuition precise by using \eqref{ebs_res_update} to quantify the amount of decrease in the least squares objective function at each iteration of \ebs. The final ingredient of the proof uses properties of convex quadratic functions (Appendix \ref{appendix_quad}) to relate the exact amount of the decrease from iteration $k$ to $k+1$ to the current optimality gap $L_n(\hat{\beta}^{k}) -L_n^*$, which yields the following strong linear convergence property:
\begin{equation}\label{drew41}
L_n(\hat{\beta}^{k+1}) - L_n^* ~\le~  \gamma \cdot (L_n(\hat{\beta}^{k}) -L_n^*) \ .
\end{equation}
The above states that the training error gap decreases at each iteration by at least the multiplicative factor of $\gamma$, and clearly implies item {\em (i)} of Theorem \ref{august24-2}.\medskip

\paragraph{Comments on the global linear convergence rate in Theorem~\ref{august24-2}}  The global linear convergence of \ebs proved in Theorem \ref{august24-2}, while novel, is not at odds with the present understanding of such convergence for optimization problems.  One can view \ebs as performing steepest descent optimization steps with respect to the $\ell_1$ norm unit ball (rather than the $\ell_2$ norm unit ball which is the canonical version of the steepest descent method, see \cite{polyak}).  It is known \cite{polyak} that canonical steepest decent exhibits global linear convergence for convex quadratic optimization so long as the Hessian matrix $Q$ of the quadratic objective function is positive definite, i.e., $\lambda_{\min}(Q) > 0$.  And for the least squares loss function $Q=\tfrac{1}{n}\bX^T\bX$, which yields the condition that $\lambda_{\min}(\bX^T\bX) > 0$.  As discussed in \cite{BV2004}, this result extends to other norms defining steepest descent as well.  Hence what is modestly surprising herein is not the linear convergence {\it per se}, but rather that \ebs exhibits global linear convergence even when $\lambda_{\min}(\bX^T\bX) = 0$, i.e., even when $\bX$ does not have full column rank (essentially replacing $\lambda_{\text{min}}(\bX^T\bX)$ with $\lambda_{\text{pmin}}(\bX^T\bX)$ in our analysis).  This derives specifically from the structure of the least squares loss function, whose function values (and whose gradient) are invariant in the null space of $\bX$, i.e., $L_n(\beta + d) = L_n(\beta)$ for all $d$ satisfying $\bX d = 0$, and is thus rendered ``immune'' to changes in $\beta$ in the null space of $\bX^T\bX$.

\subsection{Statistical Insights from the Computational Guarantees}\label{subsec-inter-bounds-ebs-1}

Note that in most noisy problems, the limiting least squares solution is statistically less interesting than an estimate obtained in the interior of the boosting profile, since the latter typically corresponds to a model with better bias-variance tradeoff.
We thus caution the reader that the bounds in Theorem~\ref{august24-2} should \emph{not} be merely interpreted as statements about how rapidly the boosting iterations reach the least squares fit.  We  rather intend for these bounds to inform us about the \emph{evolution} of the training errors and the amount of shrinkage of the coefficients as the \ebs algorithm progresses and when $k$ is at most moderately large.  When the training errors are paired with the profile of
the $\ell_{1}$-shrinkage values of the regression coefficients, they lead to the ordered pairs:
\begin{equation}\label{pair-1}
\left ( \frac{1}{2n} \| \by - \bX\hat{\beta}^k \|^2_{2} \ , \ \| \hat{\beta}^k \|_{1} \right), \;\;\;\; k \geq 1 \ ,
\end{equation}
which describes the data-fidelity and $\ell_1$-shrinkage tradeoff as a function of $k$, for the given learning rate $\varepsilon>0$.  This profile is described in
Figure~\ref{fig:bounds-lasso} in Appendix \ref{dance} for several data instances. The bounds in Theorem~\ref{august24-2} provide estimates for the two components of the
ordered pair~\eqref{pair-1}, and they can be computed prior to running the boosting algorithm.
For simplicity, let us use the following crude estimate:
$$ \ell_{k}:= \min \left\{ \|\bX\hat{\beta}_{\text{LS}}\|_{2} \sqrt{\frac{k\varepsilon}{2-\varepsilon}} \ \ , \ \ \frac{\varepsilon \|\bX\hat{\beta}_{\text{LS}}\|_{2}}{1 - \sqrt{\gamma} } \left( 1 - \gamma^{\frac{k}{2}} \right) \right\},$$
which is an upper bound of the bound in part {\em (v)} of the theorem, to provide an upper approximation of $\| \hat{\beta}_{k}\|_{1}$.  Combining the above estimate
with the guarantee in part {\em (i)} of Theorem \ref{august24-2} in~\eqref{pair-1}, we obtain the following ordered pairs:
\begin{equation}\label{pair-11}
\left( \frac{1}{2n}\|\bX\hat{\beta}_{\text{LS}}\|^2_2 \cdot \gamma^{k} + L_{n}^* \ \ , \ \ \ell_{k} \right), \;\;\;\; k \geq 1 \ ,
\end{equation}
which describe the \emph{entire} profile of the training error bounds and the $\ell_{1}$-shrinkage bounds as a function of $k$ as suggested by Theorem~\ref{august24-2}.  These profiles, as described above in \eqref{pair-11}, are illustrated in Figure~\ref{fig:ls}.

\begin{figure}[h!]
\centering
\scalebox{.95}[0.8]{\begin{tabular}{l c c c}
&&&\\
\multicolumn{4}{c}{\sf \ebs algorithm: $\ell_{1}$-shrinkage versus data-fidelity tradeoffs (theoretical bounds) \medskip}\\
& \small{\sf {Synthetic dataset $(\kappa=1)$ }} &\small{\sf {Synthetic dataset $(\kappa=25)$ }} & \small{\sf {Leukemia dataset} } \\
\rotatebox{90}{\sf {\scriptsize {~~~~~~~~~~~~~~~~~~~~~~~~~~~Training Error}}}&\includegraphics[width=0.31\textwidth,height=0.3\textheight,  trim = 1.1cm 1.2cm 1cm 2cm, clip = true ]{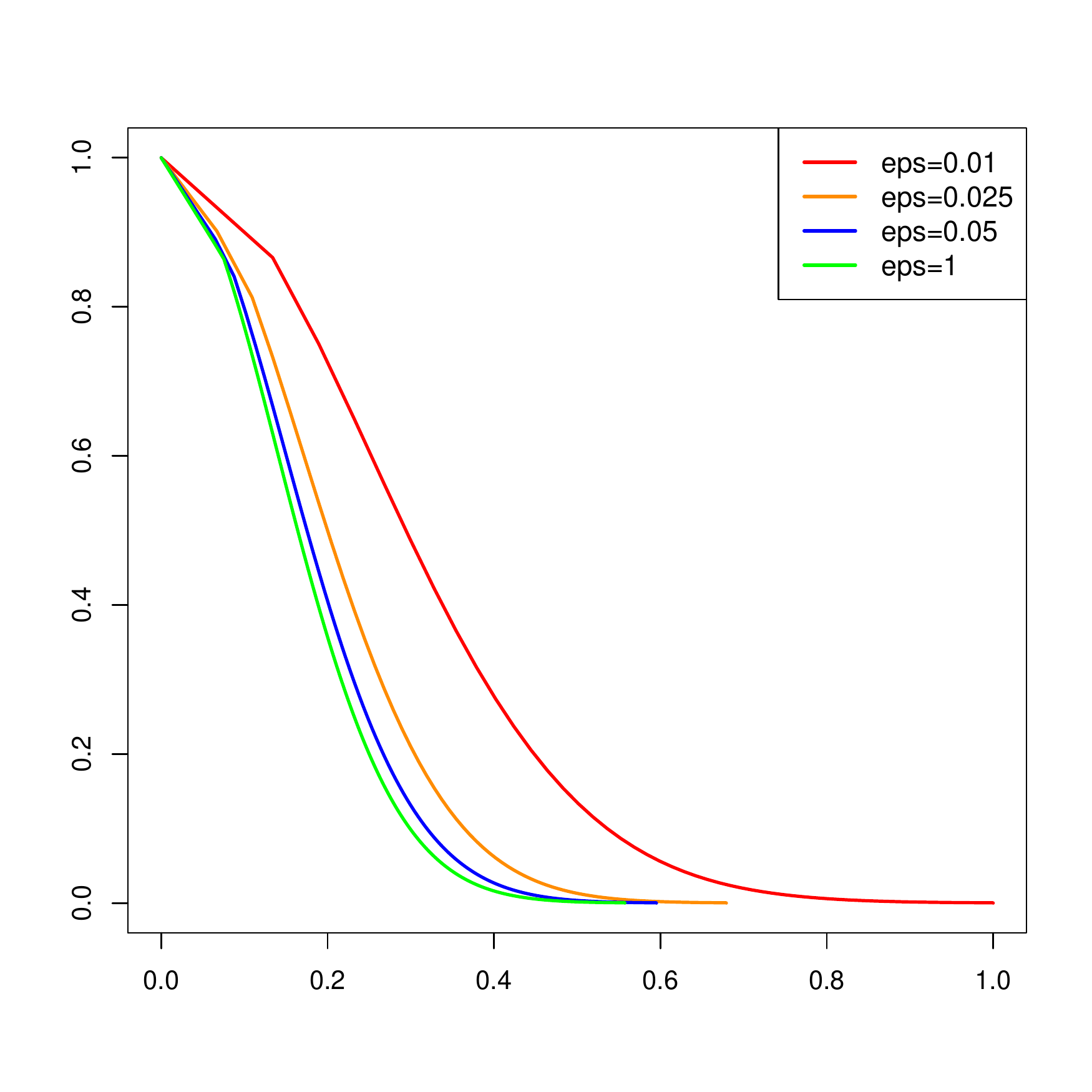}&
\includegraphics[width=0.31\textwidth,height=0.3\textheight,  trim = 2cm 1.2cm 1cm 2cm, clip = true ]{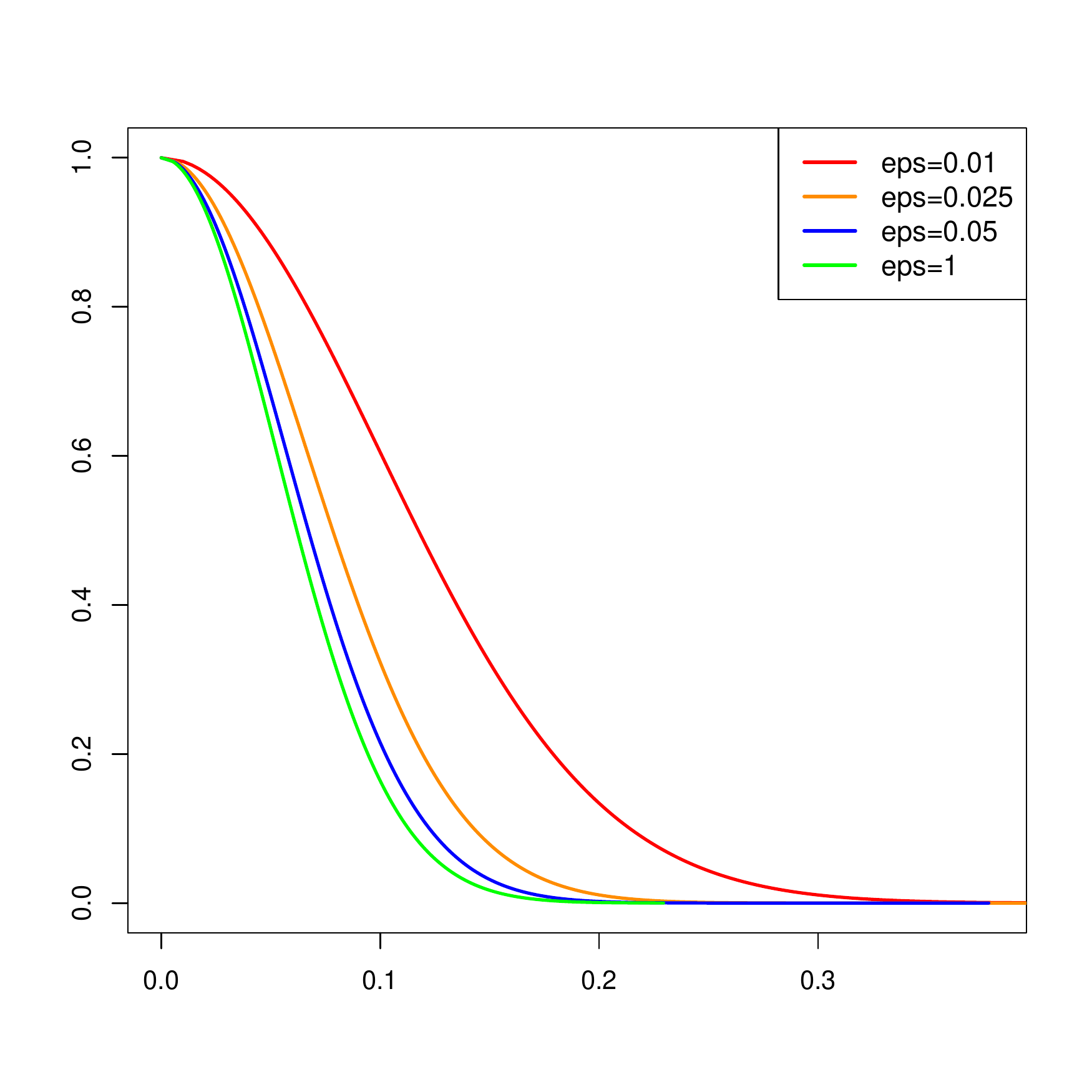}&
\includegraphics[width=0.31\textwidth,height=0.3\textheight,  trim = 2cm 1.2cm 1cm 2cm, clip = true ]{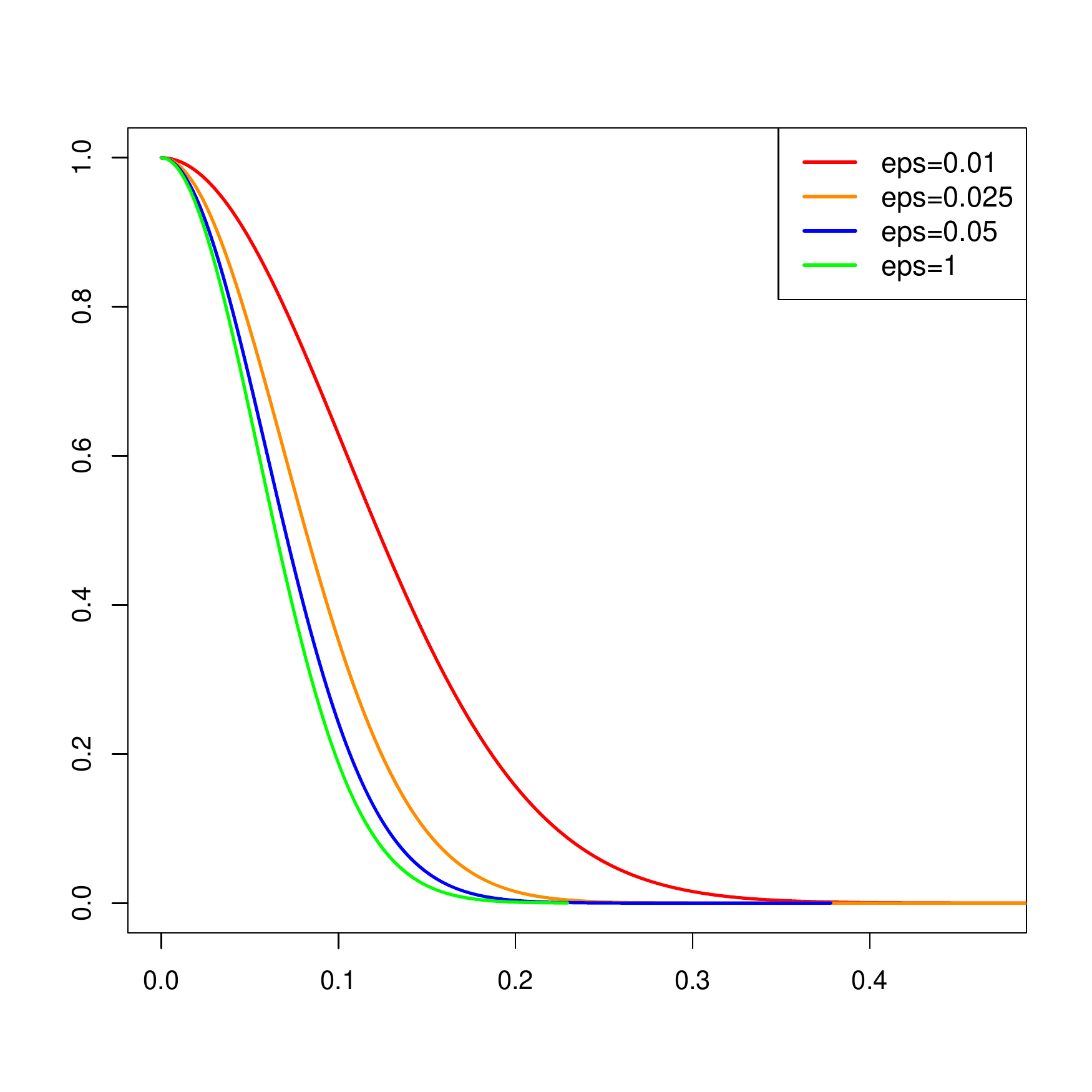} \\
&\scriptsize{\sf {$\ell_{1}$ shrinkage of coefficients}} &\scriptsize{\sf {$\ell_{1}$ shrinkage of coefficients}}  & \scriptsize{\sf {$\ell_{1}$ shrinkage of coefficients}}  \\
\end{tabular}}
\caption{{\small{Figure showing profiles of $\ell_{1}$ shrinkage of the regression coefficients versus training error for the \ebs algorithm, for different values of  the learning rate $\varepsilon$ (denoted by the moniker ``eps'' in the legend). The profiles have been obtained from the computational bounds in Theorem~\ref{august24-2}.  The left and middle panels correspond to synthetic values of the ratio $\kappa = \tfrac{p}{\lambda_{\text{pmin}}}$, and for the right panel profiles the value of $\kappa$ (here, $\kappa = 270.05$) is extracted from the
Leukemia dataset, described in Section~\ref{sec:experiments}. The vertical axes have been normalized so that the training error at $k=0$ is one, and the horizontal axes have been scaled to the unit interval.}} }\label{fig:ls}
\end{figure}

It is interesting to consider the profiles of Figure \ref{fig:ls} alongside the \emph{explicit} regularization framework of the \textsc{Lasso}~\eqref{poi-lasso} which also traces out a profile of the form~\eqref{pair-1}:
\begin{equation}\label{pair-1-lasso}
\left ( \frac{1}{2n}\| \by - \bX\hat{\beta}^*_{\delta} \|^2_{2} \ \ , \ \  \| \hat{\beta}^*_{\delta} \|_{1} \right),  \;\;\;\; \delta \geq 0 \ ,
\end{equation}
as a function of $\delta$, where, $\hat{\beta}^*_{\delta}$ is a solution to the \lasso problem~\eqref{poi-lasso}.
For a value of $\delta:= \ell_{k}$ the optimal objective value of the \lasso problem will serve as a lower bound of the corresponding \ebs loss function value at iteration $k$. Thus the training error of $\hat{\beta}^{k}$ delivered by the \ebs algorithm will be sandwiched between the following lower and upper bounds:
$$  L_{i,k}:= \frac{1}{2n}\| \by - \bX\hat{\beta}^*_{\ell_{k}} \|^2_{2}   \leq  \frac{1}{2n}\| \by - \bX\hat{\beta}^k \|^2_{2}  \leq \frac{1}{2n}\|\bX\hat{\beta}_{\text{LS}}\|^2_2 \cdot \gamma^{k} + L_{n}^*=: U_{i,k} $$
for every $k$. Note that the difference between the upper and lower bounds above, given by: $U_{i,k} - L_{i,k}$ converges to zero as $k \rightarrow \infty$.
Figure~\ref{fig:bounds-lasso} in Appendix~\ref{dance} shows the training error versus shrinkage profiles for \ebs and \lasso for different datasets.

\begin{figure}[h!]
\begin{center}
\scalebox{.9}[.8]{\begin{tabular}{ccccc}
\rotatebox{90}{\sf { {~~~~~~~~~~~~~~~~~~~~~~~~~~~~~~~\large{$\gamma$}}}}&\includegraphics[width=0.4\textwidth,height=0.3\textheight,  trim = 1.1cm 1.5cm 0cm 2cm, clip = true ]{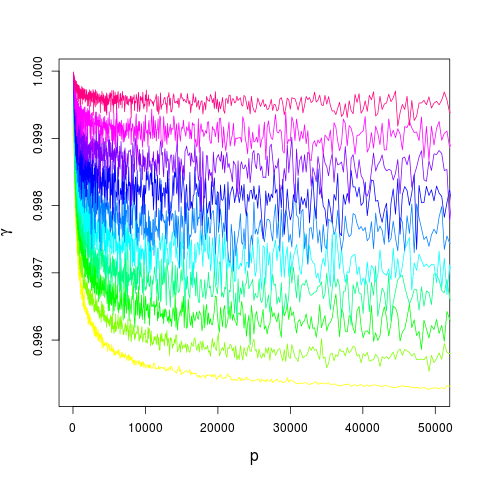}&&
\rotatebox{90}{\sf { {~~~~~~~~~~~~~~~~~~~~~~~~$\lambda_{\pmin}(\bX^T\bX)$}}}&\includegraphics[width=0.4\textwidth,height=0.3\textheight,   trim = 1.1cm 1.5cm 0cm 2cm, clip = true ]{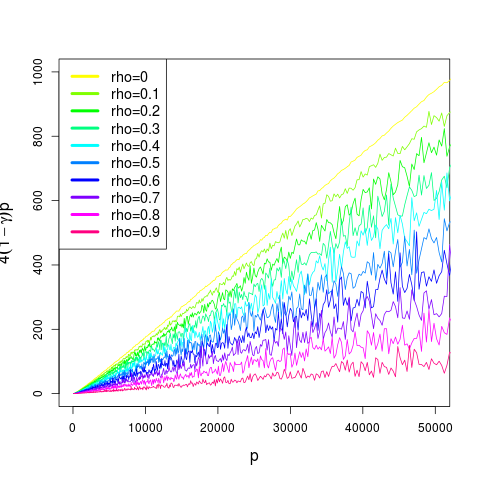}\\
&\sf $p$& &&\sf $p$ \\
\end{tabular}}
\end{center}
\caption{{\small{Figure showing the behavior of $\gamma$ [left panel] and $\lambda_{\text{pmin}}(\bX^T\bX)$ [right panel] for different values of $\rho$ (denoted by the moniker ``rho" in the legend) and $p$, with $\varepsilon=1$.
There are ten profiles in each panel corresponding to different values of $\rho$ for $\rho = 0, \ 0.1, \ \ldots, \ 0.9$.  Each profile documents the change
in $\gamma$ as a function of $p$.
 Here, the data matrix $\bX$ is comprised of
$n=50$ samples from a $p$-dimensional multivariate Gaussian distribution with mean zero, and all pairwise correlations equal to $\rho$, and the features are then standardized to have unit $\ell_{2}$ norm.
The left panel shows that $\gamma$ exhibits a phase of rapid decay (as a function of $p$) after which it stabilizes into the regime of \emph{fastest} convergence.
Interestingly, the behavior shows a monotone trend in $\rho$: the rate of progress of \ebs becomes slower for larger values of $\rho$ and faster for smaller values of $\rho$.}} }\label{fig:one-ls-boost-rate}
\end{figure}

For the bounds in parts {\em (i)} and {\em (iii)} of Theorem \ref{august24-2}, the asymptotic limits (as $k \rightarrow \infty$) are the unregularized least squares training error and predictions --- which are quantities that are uniquely defined even in the underdetermined case.

The bound in part {\em(ii)} of Theorem \ref{august24-2} is a statement concerning the regression coefficients. In this case, the notion of convergence needs to be appropriately modified from parts {\em (i)} and {\em (iii)}, since
the \emph{natural} limiting object $\hat{\beta}_{\text{LS}}$ is not necessarily unique.  In this case, perhaps not surprisingly, the regression coefficients $\hat{\beta}^{k}$ need not converge.
The result in part {\em(ii)} of the theorem states that $\hat{\beta}^k$ converges at a linear rate to the \emph{set} of least squares solutions.
In other words, at every \ebs boosting iteration, there exists a least squares solution $\hat{\beta}_{\text{LS}}^k$ for which the presented bound holds.  Here $\hat{\beta}_{\text{LS}}^k$ is in fact the closest least squares solution to $\hat{\beta}^k$ in the $\ell_{2}$ norm --- and the particular candidate least squares solution $\hat{\beta}_{\text{LS}}^k$ may be different for each iteration.

\paragraph{Interpreting the parameters and algorithm dynamics} There are several determinants of the quality of the bounds in the different parts of Theorem \ref{august24-2} which can be grouped into:
\begin{itemize}
\item algorithmic parameters: this includes the learning rate $\varepsilon$ and the number of iterations $k$, and
\item data dependent quantities: $\|\bX\hat{\beta}_{\text{LS}}\|_2$, $\lambda_{\pmin}(\bX^T\bX)$, and $p$.
\end{itemize}
The coefficient of linear convergence is given by the quantity $\gamma:=1-\frac{\varepsilon(2-\varepsilon)}{4\kappa(\bX^T\bX)}$, where $\kappa(\bX^T\bX):= \tfrac{p}{\lambda_{\text{pmin}}(\bX^T\bX)}$.
Note that $\gamma$ is monotone decreasing in $\varepsilon$ for $\varepsilon \in (0,1]$, and is minimized at $\varepsilon = 1$. This simple observation confirms
the general intuition about \ebs: $\varepsilon=1$ corresponds to the most aggressive model fitting behavior in the \ebs family, with smaller values of $\varepsilon$ corresponding to a slower model
fitting process.  The ratio $\kappa(\bX^T\bX)$ is a close cousin of the condition number associated with the data matrix $\bX$ --- and smaller values of $\kappa(\bX^T\bX)$ imply a faster rate of convergence.

In the overdetermined case with $n \ge p$ and $\rank(\bX) = p$, the condition number $\bar\kappa(\bX^T\bX) := \frac{\lambda_{\max}(\bX^T\bX)}{\lambda_{\min}(\bX^T\bX)}$ plays a key role in determining the stability of the least-squares solution $\hat{\beta}_{\text{LS}}$ and in measuring the degree of multicollinearity present. Note that $\bar\kappa(\bX^T\bX) \in [1, \infty)$, and that the problem is better conditioned for smaller values of this ratio. Furthermore, since $\rank(\bX) = p$ it holds that $\lambda_{\pmin}(\bX^T\bX) = \lambda_{\min}(\bX^T\bX)$, and thus $\bar\kappa(\bX^T\bX) \leq \kappa(\bX^T\bX)$ by \eqref{eigen_bound}. Thus the condition number $\kappa(\bX^T\bX)$ always upper bounds the classical condition number $\bar\kappa(\bX^T\bX)$, and if $\lambda_{\max}(\bX^T\bX)$ is close to $p$, then $\bar\kappa(\bX^T\bX) \approx \kappa(\bX^T\bX)$ and the two measures essentially coincide.
Finally, since in this setup $\hat{\beta}_{\text{LS}}$ is unique, part {\em (ii)} of Theorem \ref{august24-2} implies that the sequence $\{\hat{\beta}^k\}$ converges linearly to the unique least squares solution $\hat{\beta}_{\text{LS}}$.

In the underdetermined case with $p > n$, $\lambda_{\min}(\bX^T\bX) = 0$ and thus $\bar\kappa(\bX^T\bX) = \infty$. On the other hand, $\kappa(\bX^T\bX) < \infty$ since $\lambda_{\pmin}(\bX^T\bX)$ is the smallest \emph{nonzero} (hence positive) eigenvalue of $\bX^T\bX$.  Therefore the condition number $\kappa(\bX^T\bX)$ is similar to the classical condition number $\bar\kappa(\cdot)$ restricted to the subspace $\cal S$ spanned by the columns of $\bX$ (whose dimension is $\rank(\bX))$.
Interestingly, the linear rate of convergence enjoyed by \ebs is in a sense adaptive --- the algorithm automatically adjusts itself to the convergence rate dictated by the parameter $\gamma$ ``as if'' it knows that the null space of $\bX$ is not relevant.

\begin{figure}[h!]
\begin{center}
\scalebox{0.99}{\begin{tabular}{l c c c}
\multicolumn{4}{c}{ Dynamics of the \ebs algorithm versus number of boosting iterations \medskip}\\
& \small{$\rho=0$}  &\small{$\rho=0.5$}  & \small{$\rho=0.9$}   \\
\rotatebox{90}{\sf {\scriptsize {~~~~~~~~~~~~~~~Sorted Coefficient Indices}}}&\includegraphics[width=0.31\textwidth,height=0.25\textheight,  trim = 1.1cm 1.2cm 1cm 2cm, clip = true ]{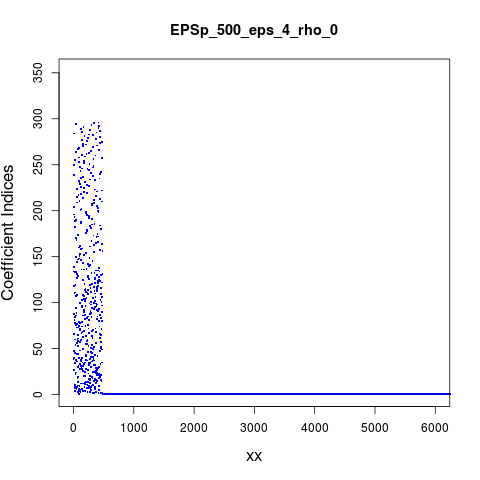}&
\includegraphics[width=0.31\textwidth,height=0.25\textheight,  trim = 2.1cm 1.2cm 1cm 2cm, clip = true ]{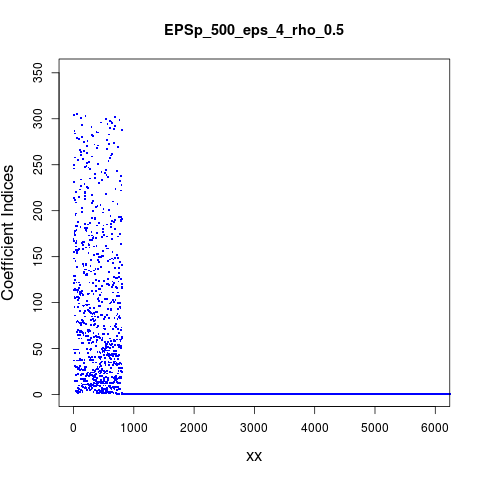}&
\includegraphics[width=0.31\textwidth,height=0.25\textheight,  trim = 2.1cm 1.2cm 1cm 2cm, clip = true ]{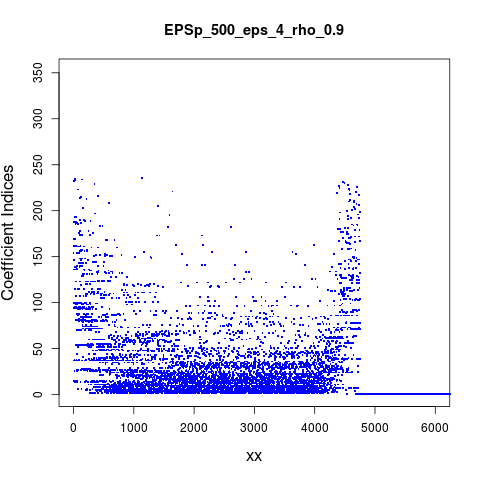}  \\
&\scriptsize{\sf Number of Boosting Iterations} &\scriptsize{\sf Number of Boosting Iterations}&\scriptsize{\sf Number of Boosting Iterations}
\end{tabular}}
\caption{{\small{Showing the \ebs algorithm run on the same synthetic dataset as was used in Figure~\ref{fig:bounds-lasso}, with $p=500$ and $\varepsilon = 1$, for three different values of the pairwise correlation $\rho$. A point is ``on'' if the corresponding regression coefficient is updated at iteration $k$. Here the vertical axes have been reoriented so that
the coefficients that are updated the maximum number of times appear lower on the axes.  For larger values of $\rho$, we see that the \ebs algorithm aggressively updates the coefficients for a large number of iterations, whereas the dynamics of the algorithm for smaller values of $\rho$ are less pronounced.  For larger values of $\rho$ the \ebs algorithm takes longer to reach the least squares fit and this is reflected in the above figure from the update patterns in the regression coefficients.  The dynamics of the algorithm evident in this figure nicely complements the insights gained from Figure~\ref{fig:fse-ebs-loss-1}.  }} }\label{fig:change-feat}
\end{center}
\end{figure}

As the dataset is varied, the value of $\gamma$ can change substantially from one dataset to another, thereby leading to differences in the convergence behavior bounds in parts {\em (i)-(v)} of Theorem \ref{august24-2}.  To settle all of these ideas, we can derive some simple bounds on $\gamma$ using tools from random matrix theory.
Towards this end, let us suppose that the entries of $\bX$ are drawn from a standard Gaussian ensemble, which are subsequently standardized such that every column of $\bX$ has unit $\ell_{2}$ norm.
Then it follows from random matrix theory \cite{vershynin2010introduction} that $\lambda_{\pmin}(\bX^T\bX) \gtrapprox \frac{1}{n}(    \sqrt{p} - \sqrt{n}  )^2$ with high probability. (See Appendix \ref{random_matrix} for a more detailed discussion of this fact.)  To gain better insights into the behavior of $\gamma$ and how it depends on the values of pairwise correlations of the features, we performed some computational experiments, the results of which are shown in Figure~\ref{fig:one-ls-boost-rate}.
Figure~\ref{fig:one-ls-boost-rate} shows the behavior of $\gamma$ as a function of $p$ for a fixed $n=50$ and $\varepsilon=1$, for different datasets $\bX$ simulated as follows.
We first generated a multivariate data matrix from a Gaussian distribution with mean zero and covariance $\Sigma_{p \times p}=(\sigma_{ij})$, where, $\sigma_{ij}  = \rho$ for all $i \neq j$; and then all of the columns of the data matrix were
standardized to have unit $\ell_{2}$ norm. The resulting matrix was taken as $\bX$.  We considered different cases by varying the magnitude of pairwise correlations of the features $\rho$ --- when $\rho$ is small, the rate of convergence is typically faster (smaller $\gamma$) and the rate becomes slower (higher $\gamma$) for higher values of $\rho$.
Figure~\ref{fig:one-ls-boost-rate} shows that the coefficient of linear convergence $\gamma$ is quite close to $1.0$ --- which suggests a slowly converging algorithm and confirms our intuition about the algorithmic behavior of \ebs.  Indeed, \ebs, like any other boosting algorithm, should indeed converge slowly to the unregularized least squares solution.  The slowly converging nature of the \ebs algorithm provides, for the first time, a precise theoretical justification of the empirical observation made in~\cite{ESLBook} that
stagewise regression is widely considered ineffective as a tool to obtain the unregularized least squares fit, as compared to other stepwise model fitting procedures
like Forward Stepwise regression (discussed in Section~\ref{sect_intro}).

The above discussion sheds some interesting insight into the behavior of the \ebs\ algorithm. For larger values of $\rho$, the observed covariates tend to be even more highly correlated (since $p \gg n$). Whenever  a pair of features are highly correlated,
the \ebs algorithm finds it \emph{difficult} to prefer one over the other and thus takes turns in updating both coefficients, thereby distributing the effects of a covariate to all of its correlated cousins.  Since a group of correlated covariates are all competing
to be updated by the \ebs algorithm, the progress made by the algorithm in decreasing the loss function is naturally slowed down.
In contrast,  when $\rho$ is small, the \ebs algorithm brings in a covariate and in a sense completes the process by doing the exact line-search on that feature.  This heuristic explanation attempts to explain the slower rate of convergence of the \ebs algorithm for large values of $\rho$ --- a phenomenon that we observe in practice and which is also substantiated by the computational guarantees in Theorem \ref{august24-2}.  We refer the reader to Figures~\ref{fig:fse-ebs-loss-1} and \ref{fig:change-feat} which further illustrate the above justification.
Statement {\em (v)} of Theorem \ref{august24-2} provides upper bounds on the $\ell_{1}$ shrinkage of the coefficients. Figure \ref{fig:ls} illustrates the evolution of the data-fidelity versus $\ell_{1}$-shrinkage as obtained from the computational bounds in Theorem~\ref{august24-2}.
Some additional discussion and properties of \ebs are presented in Appendix~\ref{sec-add-results-ebs}.

\section{Boosting Algorithms as Subgradient Descent}\label{sect_subgrad}
\paragraph{Roadmap} In this section we present a new unifying framework for interpreting the three boosting algorithms that were discussed in Section \ref{sect_intro}, namely $\FSe$, its non-uniform learning rate extension $\FSek$, and \textsc{LS-Boost}$(\varepsilon)$.  We show herein that all three algorithmic families can be interpreted as instances of the subgradient descent method of convex optimization, applied to the problem of minimizing the largest correlation between residuals and predictors.  Interestingly, this unifying lens will also result in a natural generalization of $\FSe$  with very strong ties to the \lasso solutions, as we will present in Sections~\ref{sect_lasso} and~\ref{dynamic_sect}.  The framework presented in this section leads to convergence guarantees for $\FSe$ and $\FSek$.  In Theorem \ref{December17} herein, we present a theoretical description of the
evolution of the $\FSe$ algorithm, in terms of its data-fidelity and shrinkage guarantees as a function of the number of boosting iterations. These results are a consequence of
the computational guarantees for $\FSe$ that inform us about the rate at which the $\FSe$ training error, regression coefficients, and predictions make their way to their least squares counterparts.  In order to develop these results, we first motivate and briefly review the subgradient descent method of convex optimization.

\subsection{Brief Review of Subgradient Descent}
We briefly motivate and review the subgradient descent method for non-differentiable convex optimization problems. Consider the following optimization problem:
\begin{equation}\label{poi1}
\begin{array}{rccl}
f^* := & \min\limits_{\x} & f(\x)\\
&\mathrm{s.t.} & \x \in P \ ,
\end{array}
\end{equation}\medskip
where $P \subseteq \mathbb{R}^n$ is a closed convex set and $f(\cdot) : P \to \mathbb{R}$ is a convex function. If $f(\cdot)$ is differentiable, then $f(\cdot)$ will satisfy the following gradient inequality:
$$ f(\y) \ge f(\x) + \nabla f(\x)^T(\y - \x) \ \ \ \mathrm{for~any~} \x,\y \in P \ , $$
which states that $f(\cdot)$ lies above its first-order (linear) approximation at $\x$. One of the most intuitive optimization schemes for solving \eqref{poi1} is the method of gradient descent.  This method is initiated at a given point $\x^0 \in P$.  If $\x^k$ is the current iterate, then the next iterate is given by the update formula:  $\x^{k+1} \gets \Pi_P(\x^k - \alpha_k \nabla f(\x^k))$.  In this method the potential new point is $\x^k - \alpha_k \nabla f(\x^k)$, where $\alpha_k >0$ is called the step-size at iteration $k$, and the step is taken in the direction of the negative of the gradient.  If this potential new point lies outside of the feasible region $P$, it is then projected back onto $P$.  Here recall that $\Pi_P(\cdot)$ is the Euclidean projection operator, namely $\Pi_P(\x) := \arg\min_{\y \in P} \|\x-\y\|_2$.\medskip

Now suppose that $f(\cdot)$ is not differentiable.  By virtue of the fact that $f(\cdot)$ is convex, then $f(\cdot)$ will have a {\em subgradient} at each point $\x$.  Recall that $g$ is a subgradient of $f(\cdot)$ at $\x$ if the following subgradient inequality holds:  \begin{equation}\label{jeff}f(\y) \ge f(\x) + g^T(\y - \x)  \ \ \ \mathrm{for~all~} \y \in P \ , \end{equation} which generalizes the gradient inequality above and states that $f(\cdot)$ lies above the linear function on the right side of \eqref{jeff}.  Because there may exist more than one subgradient of $f(\cdot)$ at $\x$, let $\partial f(\x)$ denote the set of subgradients of $f(\cdot)$ at $\x$.  Then ``$g \in \partial f(\x)$'' denotes that $g$ is a subgradient of $f(\cdot)$ at the point $\x$, and so $g$ satisfies \eqref{jeff} for all $\y$.  The subgradient descent method (see \cite{ShorBook}, for example) is a simple generalization of the method of gradient descent to the case when $f(\cdot)$ is not differentiable.  One simply replaces the gradient by the subgradient, yielding the following update scheme:

\begin{equation}\label{subgrad_update2}
\begin{array}{lcl}\mathrm{Compute~a~subgradient~of~} f(\cdot) \ \mathrm{at~}\x^k \ & : & \  g^k \in \partial f(\x^k) \\
\mathrm{Peform~update~at~}\x^k \ & : & \  \x^{k+1} \leftarrow \Pi_P (\x^k -\alpha_k g^k) \ .\end{array}
\end{equation}

The following proposition summarizes a well-known computational guarantee associated with the subgradient descent method.\medskip

\begin{proposition}{\bf{(Convergence Bound for Subgradient Descent}\cite{polyak, nesterovBook})}\label{subgrad}
Consider the subgradient descent method \eqref{subgrad_update2}, using a constant step-size $\alpha_{i} = \alpha$ for all $i$.  Let $\x^*$ be an optimal solution of \eqref{poi1} and suppose that the subgradients are uniformly bounded, namely $\|g^i\|_2 \leq G$ for all $i \geq 0$. Then for each $k \geq 0$, the following inequality holds:
\begin{equation}\label{sd_bound1}
\min_{i \in \{0,\ldots,k\}} f(\x^i)  \ \ \leq \ \ f^* \ + \ \frac{\|\x^0-\x^*\|_2^2}{2(k+1)\alpha} + \frac{\alpha G^2}{2} \ . \ \qed
\end{equation}
\end{proposition}

The left side of \eqref{sd_bound1} is simply the best objective function value obtained among the first $k$ iterations.  The right side of \eqref{sd_bound1} bounds the best objective function value from above, namely the optimal value $f^*$ plus a nonnegative quantity that is a function of the number of iterations $k$, the constant step-size $\{ \alpha_i\}$, the bound $G$ on the norms of subgradients, and the distance from the initial point to an optimal solution $\x^*$ of \eqref{poi1}.  Note that for a fixed step-size $\alpha>0$, the right side of \eqref{sd_bound1} goes to $\frac{\alpha G^2}{2}$ as $k \rightarrow \infty$.  In the interest of completeness, we include a proof of Proposition \ref{subgrad} in Appendix \ref{simple}.

\subsection{A Subgradient Descent Framework for Boosting}\label{boostingassubgrad}

We now show that the boosting algorithms discussed in Section~\ref{sect_intro}, namely $\FSe$ and its relatives $\FSek$ and \textsc{LS-Boost}$(\varepsilon)$,
can all be interpreted as instantiations of the subgradient descent method to minimize the largest absolute correlation between the residuals and predictors.\medskip

Let $P_{\mathrm{res}}:= \{r \in \mathbb{R}^n : r=\by-\bX\beta \ \mathrm{for~some~} \beta \in \mathbb{R}^p\}$ denote the affine space of residuals and consider the following convex optimization problem:
\begin{equation}\label{FS-problem}
\begin{array}{lccll}
\mathrm{Correlation~Minimization~(CM)}: \ \ \ f^\ast := & \min\limits_{r} & f(r) \ \ := & \|\bX^Tr\|_\infty\\
&\mathrm{s.t.} & r \in P_{\mathrm{res}} \ , \  & \ \ \ \ \ \ \ \ \ \ \ \ \ \ \ \ \ \ \ \ \ \ \
\end{array}
\end{equation}
which we dub the ``Correlation Minimization'' problem, or CM for short.  Note an important subtlety in the CM problem, namely that the optimization variable in CM is the {\em residual} $r$ and \emph{not} the regression coefficient vector $\beta$.\medskip

Since the columns of $\bX$ have unit $\ell_2$ norm by assumption, $f(r)$ is the largest absolute correlation between the residual vector $r$ and the predictors.  Therefore \eqref{FS-problem} is the convex optimization problem of minimizing the largest correlation between the residuals and the predictors, over all possible values of the residuals.
From~\eqref{grad2} with $r=\by-\bX\beta$ we observe that $\bX^Tr = 0$ if and only if $\beta$ is a least squares solution, whereby $f(r)=\|\bX^Tr\|_\infty = 0$ for the
 least squares residual vector $r = \hat r_{\text{LS}} = \by -\bX \hat{\beta}_{\text{LS}}$.  Since the objective function in~\eqref{FS-problem} is nonnegative, we conclude that $f^*=0$ and the least squares residual vector $\hat r_{\text{LS}}$ is also the unique optimal solution of the CM problem \eqref{FS-problem}. Thus CM can be viewed as an optimization problem which also produces the least squares solution.

The following proposition states that the three boosting algorithms
$\FSe$, $\FSek$ and \textsc{LS-Boost}$(\varepsilon)$ can all be viewed as instantiations of the subgradient descent method to solve the CM problem \eqref{FS-problem}.\medskip

\begin{proposition}\label{FSequiv}
Consider the subgradient descent method \eqref{subgrad_update2} with step-size sequence $\{\alpha_k\}$ to solve the correlation minimization (CM) problem \eqref{FS-problem}, initialized at $\hat r^0 = \by$. Then:
\begin{itemize}
\item[{(i)}] the $\FSe$ algorithm is an instance of subgradient descent, with a constant step-size $\alpha_k := \varepsilon$ at each iteration,

\item[{(ii)}] the $\FSek$ algorithm is an instance of subgradient descent, with non-uniform step-sizes $\alpha_k := \varepsilon_{k}$ at iteration $k$, and

\item[{(iii)}] the \textsc{LS-Boost}$(\varepsilon)$ algorithm is an instance of subgradient descent, with non-uniform step-sizes $\alpha_{k}:= \varepsilon |\tilde{u}_{j_k}|$ at iteration $k$, where $\tilde{u}_{j_k} := \argmin_{u} \|  \hat r^{k} - \M{X}_{j_k}u \|_{2}^2$.
\end{itemize}
\end{proposition}

\begin{proof}  We first prove  {\em (i)}.  Recall the update of the residuals in $\FSe$:
\begin{equation*}
\hat r^{k+1} = \hat{r}^k - \varepsilon \cdot \sgn((\hat{r}^k)^T\bX_{j_k})\bX_{j_k} \ .
\end{equation*}
We first show that $g^k := \sgn((\hat{r}^k)^T\bX_{j_k})\bX_{j_k}$ is a subgradient of the objective function $f(r) = \|\bX^Tr\|_\infty$ of the correlation minimization problem CM  \eqref{FS-problem} at $r = \hat{r}^k$.  At iteration $k$,  $\FSe$ chooses the coefficient to update by selecting $j_k \in \argmax\limits_{j \in \{1, \ldots, p\}}|(\hat{r}^k)^T\bX_j|$, whereby \\ $\sgn((\hat{r}^k)^T\bX_{j_k})\left((\hat{r}^k)^T\bX_{j_k}\right)= \|\bX^T(\hat{r}^k)\|_\infty$, and therefore for any $r$ it holds that:
$$\begin{array}{rcl} f(r) = \| \bX^Tr\|_\infty &\ge& \sgn((\hat{r}^k)^T\bX_{j_k})\left((\bX_{j_k})^Tr \right) \\ \\ &=& \sgn((\hat{r}^k)^T\bX_{j_k})\left((\bX_{j_k})^T(\hat{r}^k + r - \hat{r}^k) \right) \\ \\ &=&  \|\bX^T(\hat{r}^k)\|_\infty + \sgn((\hat{r}^k)^T\bX_{j_k})\left((\bX_{j_k})^T(r - \hat{r}^k) \right) \\ \\ &=& f(\hat{r}^k)  + \sgn((\hat{r}^k)^T\bX_{j_k})\left((\bX_{j_k})^T(r - \hat{r}^k)  \right) \ . \end{array}$$
Therefore using the definition of a subgradient in \eqref{jeff}, it follows that $g^k := \sgn((\hat{r}^k)^T\bX_{j_k})\bX_{j_k}$ is a subgradient of  $f(r) = \|\bX^Tr\|_\infty$ at $r = \hat{r}^k$.  Therefore the update $\hat r^{k+1} = \hat{r}^k - \varepsilon \cdot \sgn((\hat{r}^k)^T\bX_{j_k})\bX_{j_k} $ is of the form $\hat r^{k+1} = \hat{r}^k - \varepsilon g^k$ where $g^k \in \partial f(\hat{r}^k)$.  Last of all notice that the update can also be written as  $\hat{r}^k - \varepsilon g^k = \hat r^{k+1} = \by - \bX\hat{\beta}^{k+1} \in P_{\mathrm{res}}$, hence $\Pi_{P_{\mathrm{res}}}(\hat{r}^k - \varepsilon g^k) = \hat{r}^k - \varepsilon g^k$, i.e., the projection step is superfluous here, and therefore $\hat r^{k+1} = \Pi_{P_{\mathrm{res}}}(\hat{r}^k - \varepsilon g^k)$, which is precisely the update for the subgradient descent method with step-size $\alpha_k := \varepsilon$.

The proof of {\em (ii)} is the same as {\em (i)} with a step-size choice of $\alpha_k = \varepsilon_{k}$ at iteration $k$.
Furthermore, as discussed in Section \ref{sect_intro}, \textsc{LS-Boost}$(\varepsilon)$ may be thought of as a specific instance of $\FSek$, whereby the proof of {\em (iii)} follows as a special case of {\em (ii)}.\end{proof}

Proposition~\ref{FSequiv} presents a new interpretation of the boosting algorithms $\FSe$ and its cousins as subgradient descent. This is interesting especially since $\FSe$ and \ebs have been traditionally
interpreted as greedy coordinate descent or steepest descent type procedures \cite{ESLBook,Friedman00greedyfunction}. This has the following consequences of note:
\begin{itemize}
\item We take recourse to existing tools and results about subgradient descent optimization to inform us about the computational guarantees of these methods.  When translated to the setting of linear regression, these results will shed light on the data fidelity {\it vis-\`{a}-vis}  shrinkage characteristics of $\FSe$ and its cousins --- all using quantities that can be easily obtained prior to running the boosting algorithm.  We will show the details of this in Theorem \ref{December17} below.

\item The subgradient optimization viewpoint provides a unifying algorithmic theme which we will also apply to a regularized version of problem CM~\eqref{FS-problem}, and that we will show is very strongly connected to the \lassoperiod.  This will be developed in Section \ref{sect_lasso}. Indeed, the regularized version of the CM problem that we will develop in Section~\ref{sect_lasso} will lead to a new family of boosting algorithms which are a seemingly minor variant of the basic $\FSe$ algorithm but deliver ($O(\varepsilon)$-approximate) solutions to the \lassoperiod.
\end{itemize}

\subsection{Deriving and Interpreting Computational Guarantees for $\FSe$}
The following theorem presents the convergence properties of $\FSe$, which are a consequence of the interpretation of $\FSe$ as an instance of the subgradient descent method.\medskip

\begin{theorem}\label{December17} {\bf (Convergence Properties of $\FSe$)}  Consider the $\FSe$ algorithm with learning rate $\varepsilon$. Let $k \geq 0$ be the total number of iterations. Then there exists an index $i \in \{0, \ldots, k\}$ for which the following bounds hold:
\begin{itemize}
\item[(i)] (training error): $L_n(\hat{\beta}^i) - L_n^* \leq \frac{p}{2n\lambda_{\pmin}(\bX^T\bX)}\left[\frac{\|\bX\hat{\beta}_{\text{LS}}\|_2^2}{\varepsilon(k+1)}+ \varepsilon \right]^2 $
\item[(ii)] (regression coefficients): there exists a least squares solution $\hat{\beta}^i_{LS}$ such that:
 $$\|\hat{\beta}^i-\hat{\beta}^i_{LS}\|_2 \le \frac{\sqrt{p}}{\lambda_{\pmin}(\bX^T\bX)}\left[\frac{\|\bX\hat{\beta}_{\text{LS}}\|_2^2}{\varepsilon(k+1)}+ \varepsilon \right]$$
\item[(iii)] (predictions): for every least-squares solution $\hat{\beta}_{\text{LS}}$ it holds that $$\|\bX\hat{\beta}^i - \bX \hat{\beta}_{\text{LS}}\|_2 \le \frac{\sqrt{p}}{\sqrt{\lambda_{\pmin}(\bX^T\bX)}}\left[\frac{\|\bX\hat{\beta}_{\text{LS}}\|_2^2}{\varepsilon(k+1)}+ \varepsilon \right]$$
\item[(iv)] (correlation values) $\|\bX^T\hat r^i\|_\infty \le \displaystyle\frac{\|\bX\hat{\beta}_{\text{LS}}\|_2^2}{2\varepsilon(k+1)}+ \frac{\varepsilon}{2}$
\item[(v)] ($\ell_{1}$-shrinkage of coefficients): $\|\hat{\beta}^i\|_1 \le k \varepsilon$
\item[(vi)] (sparsity of coefficients): $\|\hat{\beta}^i\|_0 \le k$ . \ \qed
\end{itemize}
\end{theorem}

The proof of Theorem \ref{December17} is presented in Appendix \ref{migraine}.\medskip

\begin{figure}[htpb!]
\centering
\scalebox{0.99}[.85]{\begin{tabular}{l c c c}
\multicolumn{4}{c}{\sf $\FSe$ algorithm: $\ell_{1}$ shrinkage versus data-fidelity tradeoffs (theoretical bounds) \medskip}\\
&\small{\sf {   Synthetic dataset $(\kappa=1)$ }} & \small{\sf {Leukemia dataset} } & \small{\sf {Leukemia dataset (zoomed)} } \\
\rotatebox{90}{\sf {\scriptsize {~~~~~~~~~~~~~~~~~~~~~~~~~~~~~Training Error}}}&\includegraphics[width=0.3\textwidth,height=0.3\textheight,  trim = 1.1cm 1.2cm 1cm 2cm, clip = true ]{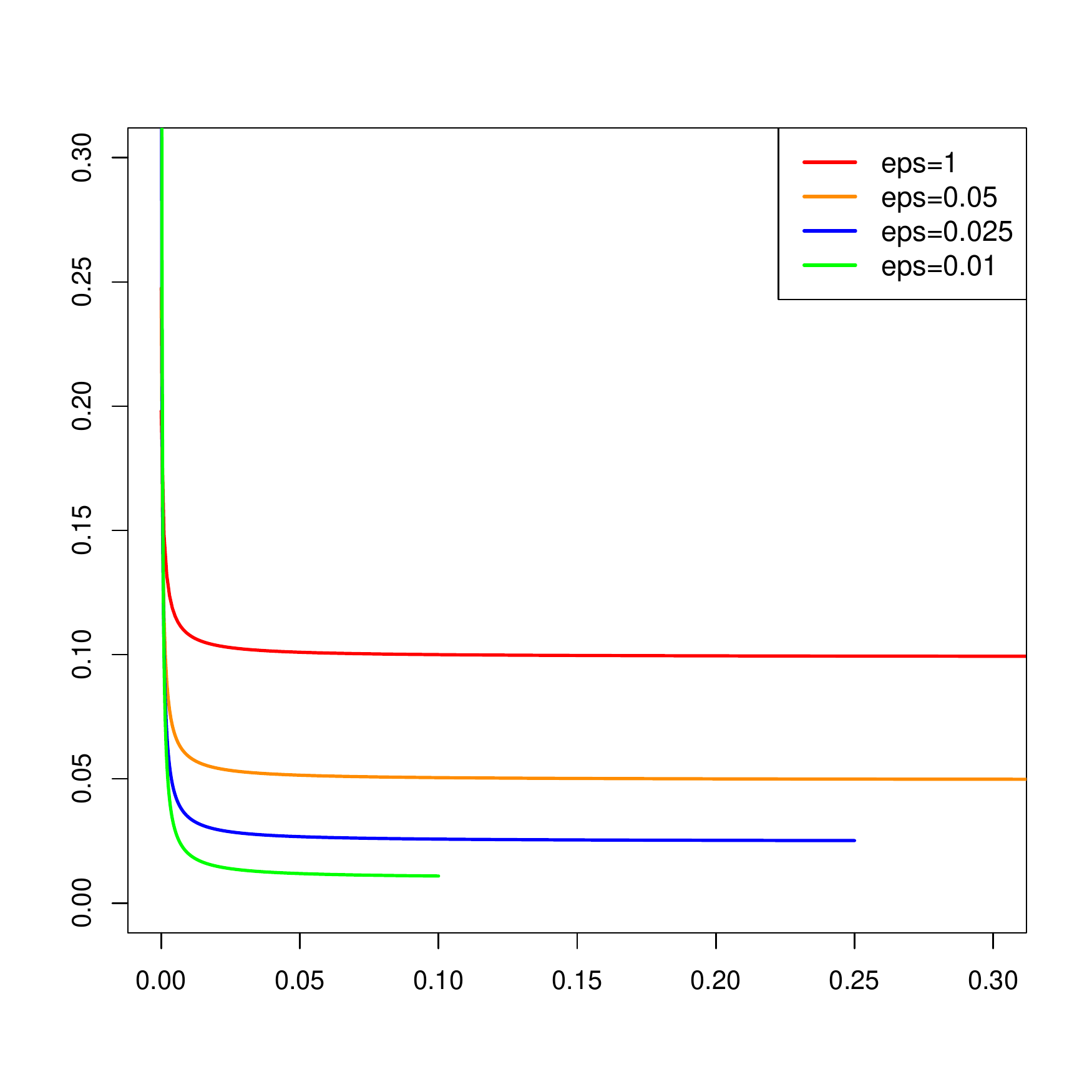}&
\includegraphics[width=0.3\textwidth,height=0.3\textheight,  trim = 1.1cm 1.2cm 1cm 2cm, clip = true ]{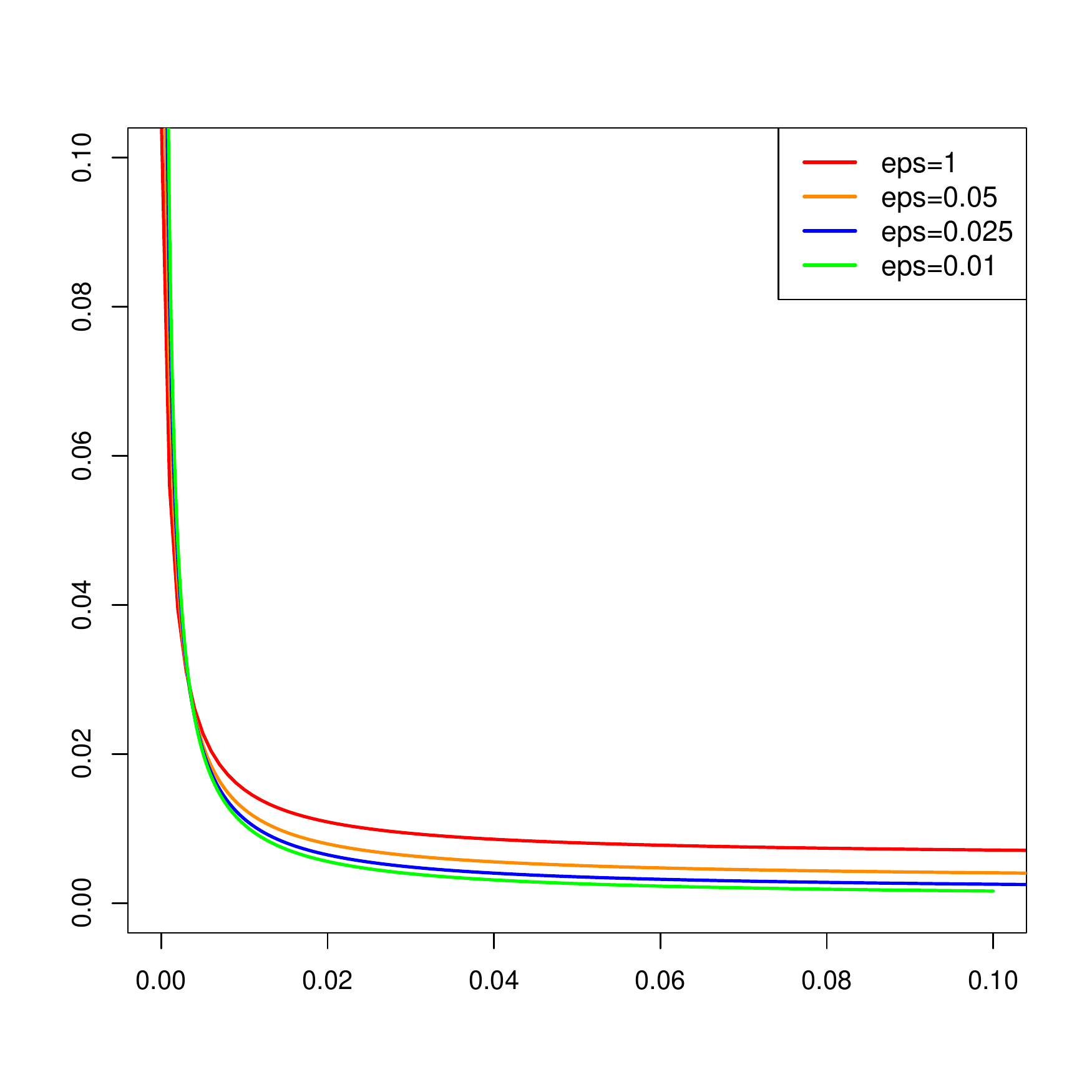}&
\includegraphics[width=0.3\textwidth,height=0.3\textheight,  trim = 1.1cm 1.2cm 1cm 2cm, clip = true ]{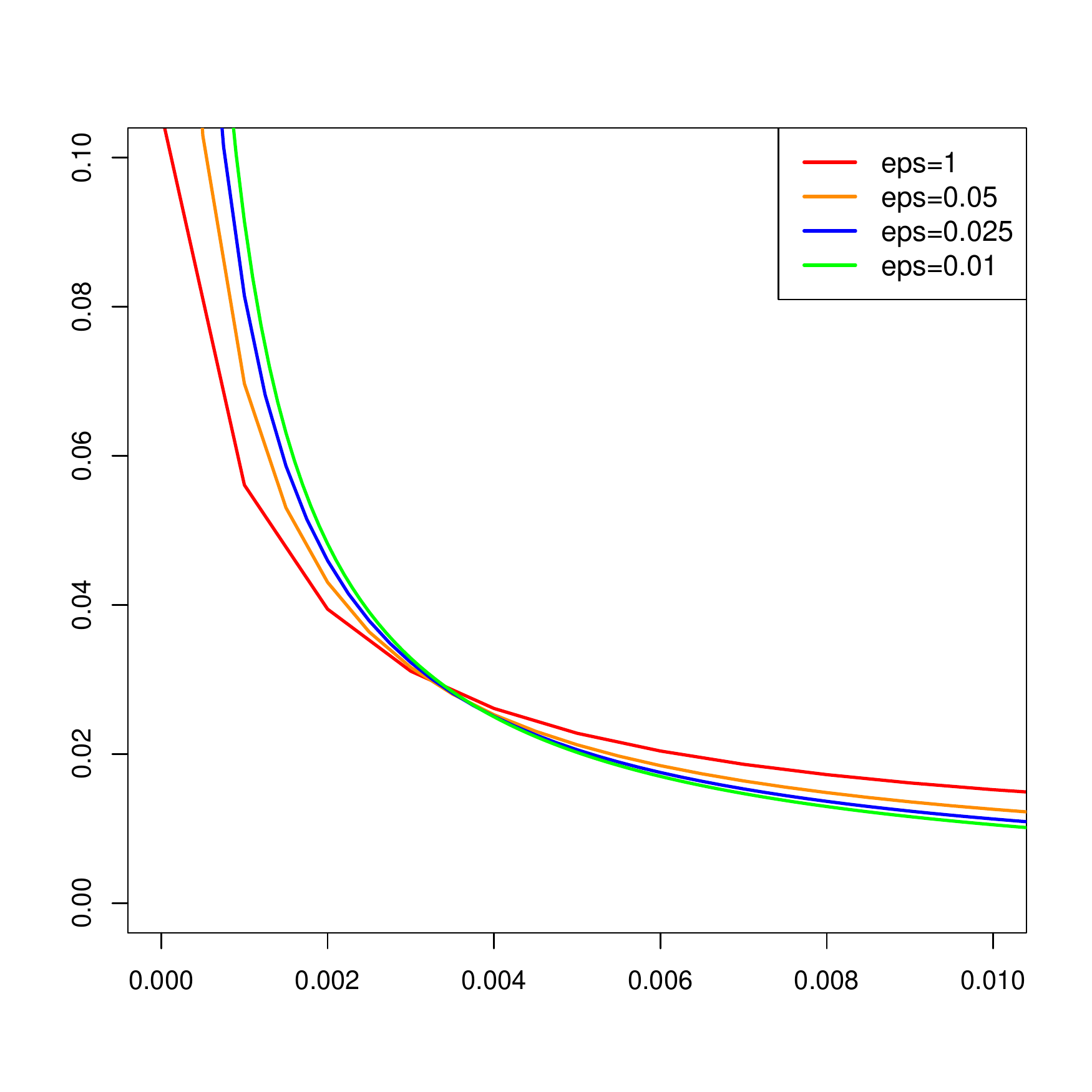} \\
& \scriptsize{\sf {$\ell_{1}$ shrinkage of coefficients}}  & \scriptsize{\sf {$\ell_{1}$ shrinkage of coefficients}} & \scriptsize{\sf {$\ell_{1}$ shrinkage of coefficients}}  \\
\end{tabular}}
\caption{{\small{Figure showing profiles of $\ell_{1}$ shrinkage bounds of the regression coefficients versus training error bounds for the $\FSe$ algorithm, for different values of  the learning rate $\varepsilon$. The profiles have been obtained from the bounds in parts {\em (i)} and {\em (v)} of Theorem \ref{December17}.  The left panel corresponds to a hypothetical dataset using $\kappa=\tfrac{p}{\lambda_{\pmin}} = 1$, and the middle and right panels use the parameters of the Leukemia dataset.}} }\label{fig:fse}
\end{figure}

\paragraph{Interpreting the Computational Guarantees} Theorem~\ref{December17}  accomplishes for $\FSe$ what Theorem~\ref{august24-2} did for \ebs --- parts {\em (i)} -- {\em (iv)} of the theorem describe the rate in which the training error, regression coefficients, and related quantities
make their way towards their ($O(\varepsilon)$-approximate) unregularized least squares counterparts. Part {\em (v)} of the theorem also describes the
rate at which the shrinkage of the regression coefficients evolve as a function of the number of boosting iterations.
The rate of convergence of $\FSe$ is sublinear, unlike the linear rate of convergence for \ebs.
Note that this type of sublinear convergence implies that the rate of decrease of the training error (for instance)
is dramatically faster in the very early iterations as compared to later iterations.  Taken together, Theorems \ref{December17}  and \ref{august24-2} highlight an important difference between the behavior of algorithms \ebs and $\FSe$:
\begin{itemize}
\item the limiting solution of the \ebs algorithm (as $k \rightarrow \infty$) corresponds to the unregularized least squares solution, but
\item the limiting solution of the $\FSe$ algorithm (as $k \rightarrow \infty$) corresponds to an $O(\varepsilon)$ approximate least squares solution.
\end{itemize}
As demonstrated in Theorems \ref{august24-2} and \ref{December17}, both \ebs and $\FSe$ have nice convergence properties with respect to the unconstrained least squares problem \eqref{poi-ls2}.  However, unlike the convergence results for \ebs in Theorem \ref{august24-2}, $\FSe$ exhibits a \emph{sublinear} rate of convergence towards a \emph{suboptimal} least squares solution.
For example, part {\em (i)} of Theorem \ref{December17} implies in the limit as $k \to \infty$ that $\FSe$ identifies a model with training error at most:
\begin{equation}\label{training_bound}
L_n^* + \frac{p\varepsilon^2}{2n(\lambda_{\pmin}(\bX^T\bX))} \ .
\end{equation}
In addition, part {\it (ii)} of Theorem \ref{December17} implies that as $k \to \infty$, $\FSe$ identifies a model whose distance to the set of least squares solutions $\{\hat{\beta}_{\text{LS}} : \bX^T\bX\hat{\beta}_{\text{LS}} = \bX^T\by\}$ is at most:
$\frac{\varepsilon\sqrt{p}}{\lambda_{\pmin}(\bX^T\bX)}. $

Note that the computational guarantees in Theorem~\ref{December17} involve the quantities $\lambda_{\pmin}(\bX^T\bX)$ and $\|\bX\hat{\beta}_{\text{LS}}\|_2$, assuming $n$ and $p$ are fixed.  To settle ideas, let us consider the synthetic datasets used in Figures~\ref{fig:one-ls-boost-rate} and~\ref{fig:fse-ebs-loss-1},
 where the covariates were generated from a multivariate Gaussian distribution with pairwise correlation $\rho$.
Figure~\ref{fig:one-ls-boost-rate} suggests that $\lambda_{\pmin}(\bX^T\bX)$ decreases with increasing $\rho$ values.
Thus, controlling for other factors
appearing in the computational bounds\footnote{To control for other factors, for example, we may assume that $p > n$  and for
different values of $\rho$ we have $\|\bX \hat\beta_{LS}\|_2 = \|\by\|_{2} = 1$ with $\varepsilon$ fixed across the different examples.} , it follows from the statements of Theorem~\ref{December17} that the training error decreases much more rapidly
for smaller $\rho$ values, as a function of $k$.  This is nicely validated by the computational results in Figure~\ref{fig:fse-ebs-loss-1} (the three top panel figures), which
show that the training errors decay at a faster rate for smaller values of $\rho$.\medskip

Let us examine more carefully the properties of the sequence of models explored by $\FSe$ and the corresponding tradeoffs between data fidelity and model complexity. Let \textsc{TBound} and \textsc{SBound} denote the training error bound and shrinkage bound in parts {\it (i)} and {\it (v)} of Theorem \ref{December17}, respectively. Then simple manipulation of the arithmetic in these two bounds yields the following tradeoff equation:
\begin{equation*}
\textsc{TBound} = \frac{p}{2n\lambda_{\pmin}(\bX^T\bX)}\left[\frac{\|\bX\hat{\beta}_{\text{LS}}\|_2^2}{\textsc{SBound} + \varepsilon}+ \varepsilon \right]^2 \ .
\end{equation*} The above tradeoff between the training error bound and the shrinkage bound is illustrated in Figure \ref{fig:fse}, which shows this tradeoff curve for four different values of the learning rate $\varepsilon$. Except for very small shrinkage levels, lower values of $\varepsilon$ produce smaller training errors.  But unlike the corresponding tradeoff curves for \ebs, there is a range of values of the shrinkage for which smaller values of $\varepsilon$ actually produce larger training errors, though admittedly this range is for very small shrinkage values.  For more reasonable shrinkage values, smaller values of $\varepsilon$ will correspond to smaller values of the training error.\medskip

Part {\em (v)} of Theorems \ref{august24-2} and~\ref{December17} presents shrinkage bounds for $\FSe$ and \ebs, respectively.  Let us briefly compare these bounds.  Examining the shrinkage bound for \ebs, we can bound the left term from above by $\sqrt{k} \sqrt{\varepsilon}\|\bX\hat\beta_{\text{LS}}\|_2$.  We can also bound the right term from above by $\varepsilon \|\bX\hat\beta_{\text{LS}}\|_2/(1-\sqrt{\gamma})$ where recall from Section \ref{LSBsection} that $\gamma$ is the linear convergence rate coefficient $\gamma := 1-\frac{\varepsilon(2-\varepsilon)\lambda_{\pmin}(\bX^T\bX)}{4p}$.  We may therefore alternatively write the following shrinkage bound for \ebs:

\begin{equation}\label{2miles}\|\hat\beta^k\|_1 \le \|\bX\hat\beta_{\text{LS}}\|_2 \min\left\{\sqrt{k} \sqrt{\varepsilon} \ , \ \varepsilon /(1-\sqrt{\gamma}) \right\} \ . \end{equation}

The shrinkage bound for $\FSe$ is simply $k\varepsilon$.  Comparing these two bounds, we observe that not only does the shrinkage bound for $\FSe$ grow at a faster rate as a function of $k$ for large enough $k$, but also the shrinkage bound for $\FSe$ grows unbounded in $k$, unlike the right term above for the shrinkage bound of \ebs.\medskip

One can also compare $\FSe$ and \ebs in terms of the efficiency with which these two methods achieve a certain pre-specified data-fidelity.  In Appendix \ref{retread} we show, at least in theory, that \ebs is much more efficient than $\FSe$ at achieving such data-fidelity, and furthermore it does so with much better shrinkage.

\section{Regularized Correlation Minimization, Boosting, and \lassoperiod}\label{sect_lasso}

\paragraph{Roadmap} In this section we introduce a new boosting algorithm, parameterized by a scalar $\delta \geq 0$, which we denote by $\RFSe$ (for Regularized incremental Forward Stagewise regression), that is obtained by incorporating a simple rescaling step to the coefficient updates in $\FSe$.  We then introduce a regularized version of the Correlation Minimization (CM) problem~\eqref{FS-problem} which we refer to as RCM.  We show that the adaptation of the subgradient descent algorithmic framework to the Regularized Correlation Minimization problem RCM exactly yields the algorithm $\RFSe$. The new algorithm $\RFSe$ may be interpreted as a natural extension of popular boosting algorithms like $\FSe$, and has the following notable properties:
\begin{itemize}

\item Whereas $\FSe$ updates the coefficients in an additive fashion by adding a small amount $\varepsilon$ to the coefficient most correlated with the current residuals, $\RFSe$ first shrinks \emph{all} of the coefficients by a scaling factor $1 - \tfrac{\varepsilon}{\delta} < 1$ and then updates the selected coefficient in the same additive fashion as $\FSe$.

\item $\RFSe$ delivers $O(\varepsilon)$-accurate solutions to the \lasso in the limit as $k \rightarrow \infty$, unlike $\FSe$ which delivers $O(\varepsilon)$-accurate
solutions to the unregularized least squares problem.

\item $\RFSe$ has computational guarantees similar in spirit to the ones described in the context of $\FSe$ -- these quantities directly inform us about
the data-fidelity {\it vis-\`{a}-vis} shrinkage tradeoffs as a function of the number of boosting iterations and the learning rate $\varepsilon$.

\end{itemize}

The notion of using additional regularization along with the implicit shrinkage imparted by boosting is not new in the literature. Various interesting notions
have been proposed in~\cite{friedman2003importance,buhlmann2006sparse,zhao2007stagewise,duchi2009boosting,FJ08}, see also the discussion in Appendix~\ref{related-work-lasso-boost} herein.  However, the framework we present here is new.  We present a unified subgradient descent framework for a class of regularized CM problems that results in algorithms that have appealing structural similarities with forward stagewise regression type algorithms, while also being very strongly connected to the \lassoperiod.

\paragraph{Boosting with additional shrinkage -- $\RFSe$} Here we give a formal description of the $\RFSe$ algorithm.  $\RFSe$ is controlled by two parameters: the learning rate $\varepsilon$, which plays the same role as the learning rate in $\FSe$, and the ``regularization parameter" $\delta \geq \varepsilon$. Our reason for referring to $\delta$ as a regularization parameter is due to the connection between $\RFSe$ and the \lassoperiod, which will be made clear later. The shrinkage factor, i.e., the amount by which we shrink the coefficients before updating the selected coefficient, is determined as $1 - \tfrac{\varepsilon}{\delta}$. Supposing that we choose to update the coefficient indexed by $j_k$ at iteration $k$, then the coefficient update may be  written as: $$\hat{\beta}^{k+1} \gets \left(1 - \tfrac{\varepsilon}{\delta}\right)\hat{\beta}^k + \varepsilon\cdot\sgn((\hat{r}^k)^T\bX_{j_k})e_{j_k} \ . $$
Below we give a concise description of $\RFSe$, including the update for the residuals that corresponds to the update for the coefficients stated above.

\begin{center}
 { {\bf  Algorithm:} $\RFSe$ }
\end{center}

Fix the learning rate $\varepsilon > 0$, regularization parameter $\delta>0$ such that $\varepsilon \leq \delta$, and number of iterations $M$.

Initialize at $\hat r^0 = \by$, $\hat{\beta}^0 = 0$, $k = 0$.

\bee
\item[\bf 1.] For $0 \leq k \leq M$ do the following:

\item[\bf 2.] Compute:
$j_k \in \argmax\limits_{j \in \{1, \ldots, p\}} |(\hat{r}^k)^T\bX_j|$

\item[\bf 3.]
\begin{description}
\item $\hat r^{k+1} \gets \hat{r}^k - \varepsilon\left[\sgn((\hat{r}^k)^T\bX_{j_k})\bX_{j_k} + \tfrac{1}{\delta}(\hat{r}^k - \by)\right]$
\item $\hat{\beta}^{k+1}_{j_k} \gets \left(1 - \tfrac{\varepsilon}{\delta}\right) \hat{\beta}^{k}_{j_k} + \varepsilon\ \sgn((\hat{r}^k)^T\bX_{j_k})$ and
$\hat{\beta}^{k+1}_j \gets \left(1 - \tfrac{\varepsilon}{\delta}\right) \hat{\beta}^{k}_j \ , j \neq j_k$
\end{description}
\eee
\medskip

Note that $\RFSe$ and $\FSe$ are structurally very similar -- and indeed when $\delta = \infty$ then $\RFSe$ is exactly $\FSe$.  Note also that $\RFSe$ shares the same upper bound on the sparsity of the regression coefficients as $\FSe$, namely for all $k$ it holds that: $\|\hat{\beta}^k\|_0 \le k$. When $\delta < \infty$ then, as previously mentioned, the main structural difference between $\RFSe$ and $\FSe$ is the additional rescaling of the coefficients by the factor $1 - \tfrac{\varepsilon}{\delta}$. This rescaling better controls the growth of the coefficients and, as will be demonstrated next, plays a key role in connecting $\RFSe$ to the \lassoperiod.

\paragraph{Regularized Correlation Minimization (RCM) and \lassoperiod} The starting point of our formal analysis of $\RFSe$ is the Correlation Minimization (CM) problem \eqref{FS-problem}, which we now modify by introducing a regularization term that penalizes residuals that are far from the vector of observations $\by$. This modification leads to the following parametric family of optimization problems indexed by $\delta \in (0, \infty]$:
\begin{equation}\label{FS-dual-prob}
\begin{array}{rccrll}
\mathrm{RCM}_\delta : \ \ \ \ \ \ \ \ \ f_\delta^\ast := & \min\limits_{r} & f_\delta(r) \ \ & := & \|\bX^Tr\|_\infty + \tfrac{1}{2\delta}\|r-\by\|_2^2\\ \\
&\mathrm{s.t.} & r \in P_{\mathrm{res}} & := & \{r \in \mathbb{R}^n : r=\by-\bX\beta \ \mathrm{for~some~} \beta \in \mathbb{R}^p\}  \ , \end{array}
\end{equation}
where ``RCM'' connotes Regularlized Correlation Minimization.
Note that RCM reduces to the correlation minimization problem CM \eqref{FS-problem} when $\delta = \infty$.  RCM may be interpreted as the problem of minimizing, over the space of residuals, the largest correlation between the residuals and the predictors plus a regularization term that penalizes residuals that are far from the response $\by$ (which itself can be interpreted as the residuals associated with the model $\beta = 0$).

Interestingly, as we show in Appendix \ref{lasso_dual_appendix}, RCM \eqref{FS-dual-prob} is equivalent to the \lasso \eqref{poi-lasso} via duality. This equivalence provides further insight about the regularization used to obtain $\mathrm{RCM}_\delta$. Comparing the \lasso and RCM, notice that the space of the
variables of the \lasso is the space of regression coefficients $\beta$, namely $\mathbb{R}^p$, whereas the space of the variables of RCM is the space of model residuals, namely $\mathbb{R}^n$, or more precisely $P_{\mathrm{res}}$. The duality relationship shows that $\mathrm{RCM}_\delta$ \eqref{FS-dual-prob} is an equivalent characterization of the \lasso problem, just like the correlation minimization (CM) problem \eqref{FS-problem} is an equivalent characterization of the (unregularized) least squares problem. Recall that Proposition~\ref{FSequiv} showed that subgradient descent applied to the CM problem~\eqref{FS-dual-prob} (which is $\mathrm{RCM}_\delta$ with $\delta = \infty$) leads to the well-known boosting algorithm $\FSe$. We now extend this theme with the following Proposition, which demonstrates $\RFSe$ is equivalent to subgradient descent applied to $\mathrm{RCM}_\delta $.\medskip

\begin{proposition}\label{RFSequiv}
The $\RFSe$ algorithm is an instance of subgradient descent to solve the regularized correlation minimization ($\mathrm{RCM}_\delta$) problem \eqref{FS-dual-prob}, initialized at $\hat r^0 = \by$, with a constant step-size $\alpha_k := \varepsilon$ at each iteration.
\end{proposition}

The proof of Proposition \ref{RFSequiv} is presented in Appendix \ref{may3}.

\subsection{$\RFSe$: Computational Guarantees and their Implications}

In this subsection we present computational guarantees and convergence properties of the boosting algorithm $\RFSe$.  Due to the structural equivalence between $\RFSe$ and subgradient descent applied to the $\mathrm{RCM}_\delta$ problem \eqref{FS-dual-prob} (Proposition \ref{RFSequiv}) and the close connection between $\mathrm{RCM}_\delta$ and the \lasso (Appendix \ref{lasso_dual_appendix}), the convergence properties of $\RFSe$ are naturally stated with respect to the \lasso problem \eqref{poi-lasso}.  Similar to Theorem \ref{December17} which described such properties for $\FSe$ (with respect to the unregularized least squares problem), we have the following properties for $\RFSe$.\medskip

\begin{theorem}\label{RFSe-guarantees} {\bf (Convergence Properties of $\RFSe$ for the \lasso)}  Consider the $\RFSe$ algorithm with learning rate $\varepsilon$ and regularization parameter $\delta \in (0, \infty)$, where $\varepsilon \leq \delta$.  Then the regression 
coefficient $\hat{\beta}^k$ is feasible for the \lasso problem \eqref{poi-lasso} for all $k\ge0$.  Let $k \ge 0$ denote a specific iteration counter.  Then there exists an index $i \in \{0, \ldots, k\}$ for which the following bounds hold:
\begin{itemize}
\item[(i)] (training error): $L_n(\hat{\beta}^i) - L_{n,\delta}^* ~\leq \frac{\delta}{n}\left[\frac{\|\bX\hat{\beta}_{\text{LS}}\|_2^2}{2\varepsilon(k+1)} + 2\varepsilon \right]$
\item[(ii)] (predictions): for every \lasso solution $\hat{\beta}^\ast_\delta$ it holds that $$\|\bX\hat{\beta}^i - \bX \hat{\beta}^\ast_\delta\|_2 ~\leq \sqrt{\frac{\delta\|\bX\hat{\beta}_{\text{LS}}\|_2^2}{\varepsilon(k+1)} + 4\delta\varepsilon }$$
\item[(iii)] ($\ell_{1}$-shrinkage of coefficients): $\|\hat{\beta}^i\|_1 \le~ \delta\left[1 - \left(1 - \tfrac{\varepsilon}{\delta}\right)^k\right] \leq~ \delta $
\item[(iv)] (sparsity of coefficients): $\|\hat{\beta}^i\|_0 \le k$ . \qed
\end{itemize}
\end{theorem}

The proof of Theorem \ref{RFSe-guarantees} is presented in Appendix \ref{migraine2}.\medskip

\begin{figure}[h!]
\centering
\scalebox{0.94}[0.85]{\begin{tabular}{c c c   c c}
& \multicolumn{4}{c}{\sf $\RFSe$ algorithm, Prostate cancer dataset (computational bounds) } \medskip\\
\rotatebox{90}{\sf {\scriptsize {~~~~~~~$\ell_{1}$-norm of coefficients (relative scale)}}}&\includegraphics[width=0.3\textwidth,height=0.3\textheight,  trim = 1.0cm 1.7cm .1cm 2cm, clip = true ]{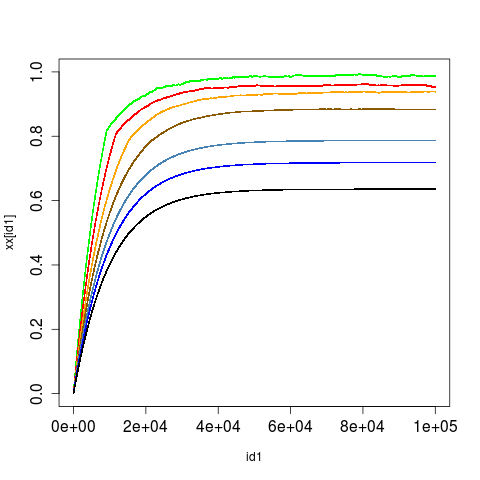}&
\rotatebox{90}{\sf {\scriptsize {~~~~~~~~~~~~Training Error (relative scale)}}}&\includegraphics[width=0.3\textwidth,height=0.3\textheight,  trim = 1.0cm 1.7cm .5cm 2cm, clip = true ]{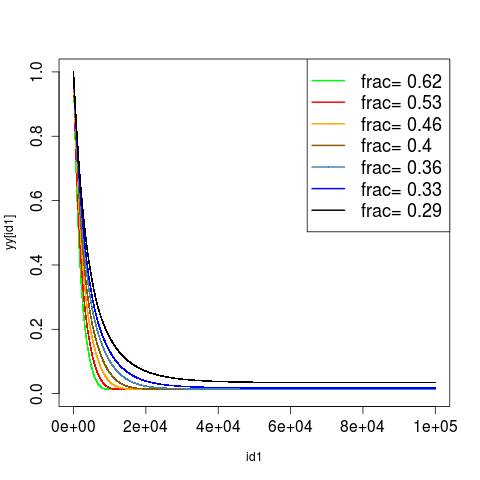}&
\includegraphics[width=0.3\textwidth,height=0.3\textheight,  trim = 1.0cm 1.7cm .3cm 2cm, clip = true ]{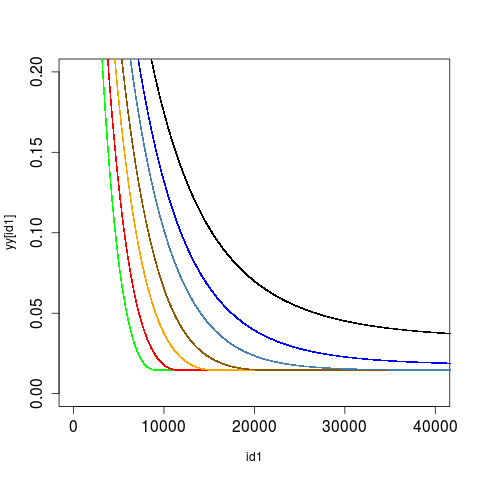} \\
& \scriptsize{\sf {Iterations}}  &&  \scriptsize{\sf {Iterations}}   & \scriptsize{\sf {Iterations}}    \\
\end{tabular}}
\caption{{\small{Figure showing the evolution of the $\RFSe$ algorithm (with $\varepsilon= 10^{-4}$) for different values of $\delta$, as a function of the number of boosting iterations for the Prostate cancer dataset, with
$n=10,p=44$, appearing in the bottom panel of Figure~\ref{fig:FSe-lasso-RFSe}.
 [Left panel] shows the change of the $\ell_{1}$-norm of the regression coefficients.
[Middle panel] shows the evolution of the training errors, and [Right panel] is a zoomed-in version of the middle panel. Here we took different values of  $\delta$ given by $\delta = \text{frac} \times \delta_{\max}$, where,
$\delta_{\max}$
 denotes the $\ell_{1}$-norm of the minimum $\ell_{1}$-norm least squares solution, for $7$ different values of \text{frac}.}}}\label{fig:FW-fse-loss-L1}
\end{figure}

\paragraph{Interpreting the Computational Guarantees} The statistical interpretations implied by the computational guarantees presented in Theorem \ref{RFSe-guarantees} are analogous to those previously discussed for \ebs (Theorem \ref{august24-2}) and $\FSe$ (Theorem \ref{December17}).  These guarantees inform us about the data-fidelity {\it vis-\`{a}-vis} shrinkage tradeoffs as a function of the number of boosting iterations, as nicely demonstrated in Figure~\ref{fig:FW-fse-loss-L1}.  There is, however, an important differentiation between the properties of $\RFSe$ and the properties of \ebs and $\FSe$, namely:
\begin{itemize}
\item For \ebs and $\FSe$, the computational guarantees (Theorems \ref{august24-2} and \ref{December17}) describe how the estimates make their way to a unregularized ($O(\varepsilon)$-approximate) least squares solution as a function of the number of boosting iterations.
\item For $\RFSe$, our results (Theorem \ref{RFSe-guarantees}) characterize how the estimates approach a ($O(\varepsilon)$-approximate) \lasso solution.
\end{itemize}

Notice that like $\FSe$, $\RFSe$ traces out a profile of regression coefficients. This is reflected in item {\em (iii)} of Theorem \ref{RFSe-guarantees} which bounds the $\ell_{1}$-shrinkage of the coefficients as a function of the number of boosting iterations $k$. Due to the rescaling of the coefficients, the $\ell_{1}$-shrinkage may be bounded by a geometric series that approaches $\delta$ as $k$ grows. Thus, there are two important aspects of the bound in item {\em (iii)}: {\em (a)} the dependence on the number of boosting iterations $k$ which characterizes model complexity during early iterations, and {\em (b)} the uniform bound of $\delta$ which applies even in the limit as $k \to \infty$ and implies that all regression coefficient iterates $\hat{\beta}^k$ are feasible for the \lasso problem \eqref{poi-lasso}.

On the other hand, item {\em (i)} characterizes the quality of the coefficients with respect to the \lasso solution, as opposed to the unregularized least squares problem as in $\FSe$. In the limit as $k \to \infty$, item {\em (i)} implies that $\RFSe$ identifies a model with training error at most
$L_{n, \delta}^* + \frac{2\delta\varepsilon}{n} \ .$
This upper bound on the training error may be set to any prescribed error level by appropriately tuning $\varepsilon$; in particular, for $\varepsilon \approx 0$ and fixed $\delta>0$ this limit is essentially $L_{n, \delta}^*$. Thus, combined with the uniform bound of $\delta$ on the $\ell_{1}$-shrinkage, we see that the $\RFSe$ algorithm delivers the \lasso solution in the limit as $k \rightarrow \infty$.

It is important to emphasize that $\RFSe$ should not just be interpreted as an algorithm to solve the \lassoperiod. Indeed, like $\FSe$, the trajectory of the algorithm is important and $\RFSe$ may identify a more statistically interesting model in the interior of its profile. Thus, even if the \lasso solution for $\delta$ leads to overfitting, the $\RFSe$ updates may visit a model with better predictive performance by trading off bias and variance in a more desirable fashion suitable for the particular problem at hand.

Figure~\ref{fig:FSe-lasso-RFSe} shows the profiles of $\RFSe$ for different values of $\delta \leq \delta_{\max}$, where $\delta_{\max}$ is the $\ell_{1}$-norm of  the minimum $\ell_{1}$-norm least squares solution. Curiously enough, Figure~\ref{fig:FSe-lasso-RFSe} shows that in some cases, the profile of $\RFSe$ bears a lot of similarities with that of the \lasso (as presented in Figure~\ref{fig:lasso-similar}).
However, the profiles are in general different.
Indeed, $\RFSe$ imposes a uniform bound of $\delta$ on the $\ell_{1}$-shrinkage, and so for values larger than $\delta$ we cannot possibly expect $\RFSe$ to approximate the \lasso path. However, even if $\delta$ is taken to be sufficiently large (but finite) the profiles may be different. In this connection it is helpful to draw the analogy between the curious similarities between the
$\FSe$ (i.e., $\RFSe$ with $\delta = \infty$) and \lasso coefficient profiles, even though the profiles are different in general.

\begin{figure}[ht!]
\centering
\begin{tabular}{l c c c c}
 & \sf {\scriptsize {$\FSe$ }} & \sf {\scriptsize { $\RFSe$, $\delta=0.99\delta_{\max}$ }} & \sf {\scriptsize {$\RFSe$,$\delta=0.91\delta_{\max}$ }} & \sf {\scriptsize {$\RFSe$,$\delta=0.81\delta_{\max}$}} \\
\rotatebox{90}{\sf {\scriptsize {~~~~~~~~Regression Coefficients}}}&
\includegraphics[width=0.23\textwidth,height=0.2\textheight,  trim = 1.cm 1.5cm .2cm 1.7cm, clip = true ]{example_case3fse_1.png}&
\includegraphics[width=0.23\textwidth,height=0.2\textheight,  trim = 2cm 1.5cm .2cm 1.7cm, clip = true ]{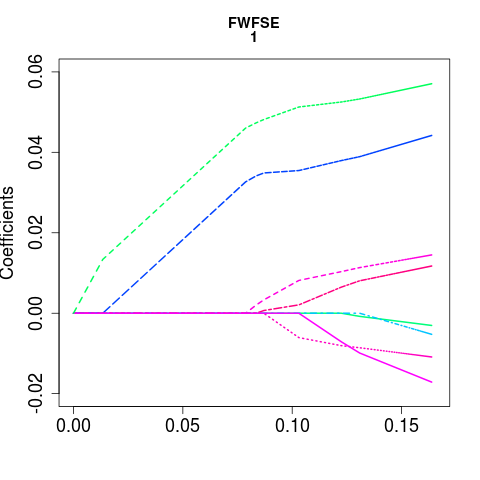}&
\includegraphics[width=0.23\textwidth,height=0.2\textheight,  trim = 2cm 1.5cm .2cm 1.7cm, clip = true ]{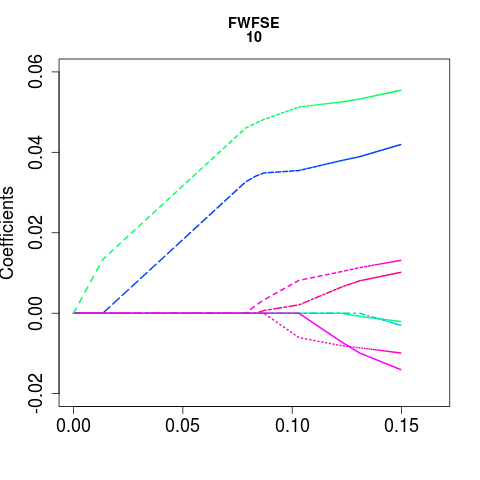}&
\includegraphics[width=0.23\textwidth,height=0.2\textheight,  trim = 2cm 1.5cm .2cm 1.7cm, clip = true ]{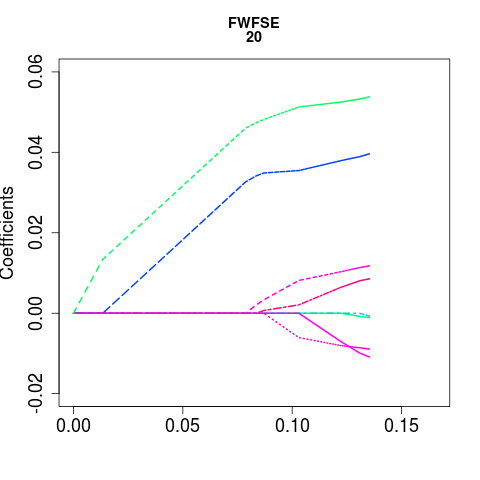}\\
&&&& \\
& \sf {\scriptsize {$\FSe$ }} & \sf {\scriptsize { $\RFSe$, $\delta=0.44\delta_{\max}$ }} & \sf {\scriptsize {$\RFSe$,$\delta=0.37\delta_{\max}$ }} & \sf {\scriptsize {$\RFSe$,$\delta=0.33\delta_{\max}$}} \\
\rotatebox{90}{\sf {\scriptsize {~~~~~~~~~~~Regression Coefficients}}}&
\includegraphics[width=0.23\textwidth,height=0.2\textheight,  trim = 1.cm 1.5cm .2cm 1.7cm, clip = true ]{example_case2fse_1.png}&
\includegraphics[width=0.23\textwidth,height=0.2\textheight,  trim = 2cm 1.5cm .2cm 1.7cm, clip = true ]{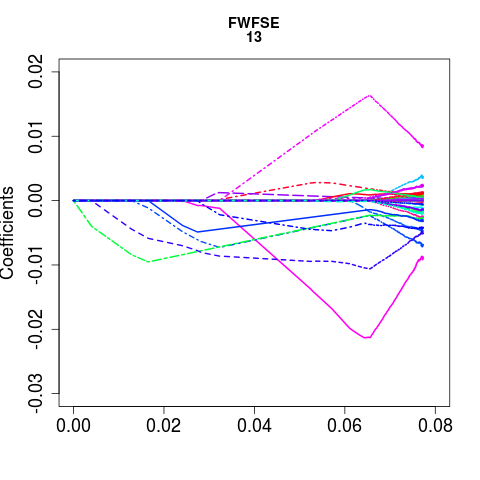}&
\includegraphics[width=0.23\textwidth,height=0.2\textheight,  trim = 2cm 1.5cm 0.2cm 1.7cm, clip = true ]{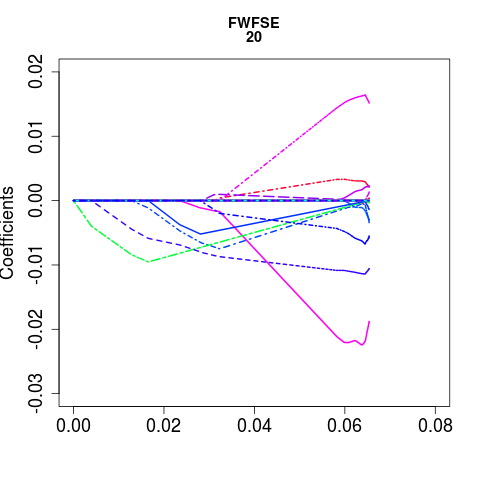}&
\includegraphics[width=0.23\textwidth,height=0.2\textheight,  trim = 2cm 1.5cm 0.2cm 1.7cm, clip = true ]{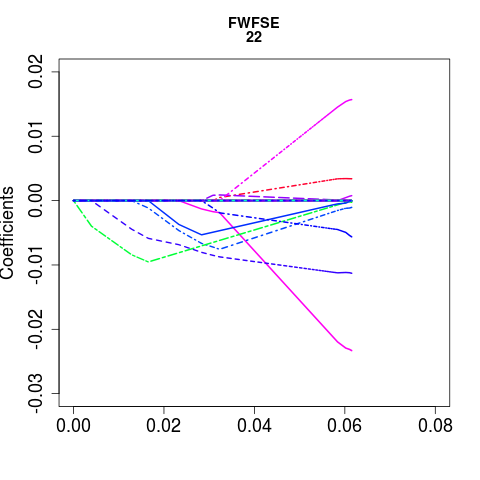}\\

 & \sf {\scriptsize { $\ell_{1}$-norm of coefficients }} & \sf {\scriptsize { $\ell_{1}$-norm of coefficients }}& \sf {\scriptsize { $\ell_{1}$-norm of coefficients }} & \sf {\scriptsize { $\ell_{1}$-norm of coefficients }}  \\
\end{tabular}
\caption{{\small{Coefficient profiles for $\RFSe$ as a function of the $\ell_{1}$-norm of the regression coefficients, for the same datasets appearing in Figure~\ref{fig:lasso-similar}.
 For each example, different values of $\delta$ have been considered.
The left panel corresponds to the choice $\delta= \infty$, i.e., $\FSe$. In all the above cases, the algorithms were run for a maximum of 100,000 boosting iterations with $\varepsilon=10^{-4}$.
 [Top Panel] Corresponds to the Prostate cancer dataset with $n=98$ and $p=8$.
All the  coefficient profiles look similar, and they all seem to coincide with the \lasso profile (see also Figure~\ref{fig:lasso-similar}).
[Bottom Panel]  Shows the Prostate cancer dataset with a subset of samples $n=10$ with all interactions included with $p=44$. The coefficient profiles in this example are sensitive to the choice of $\delta$ and are seen to be
more constrained towards the end of the path, for decreasing $\delta$ values. The profiles are different than the \lasso profiles, as seen in Figure~\ref{fig:lasso-similar}. The regression coefficients at the end of the path correspond to
approximate \lasso solutions, for the respective values of $\delta$.}} }\label{fig:FSe-lasso-RFSe}
\end{figure}

As a final note, we point out that one can also interpret $\RFSe$ as the Frank-Wolfe algorithm in convex optimization applied to the \lasso \eqref{poi-lasso} in line with \cite{bach2012duality}.  We refer the reader to Appendix \ref{FW_appendix} for discussion of this point.

\section{A Modified Forward Stagewise Algorithm for Computing the \lasso Path}\label{dynamic_sect}
In Section~\ref{sect_lasso} we introduced the boosting algorithm $\RFSe$ (which is a very close cousin of $\FSe$) that delivers solutions to the \lasso problem \eqref{poi-lasso} for a fixed but arbitrary $\delta$, in the limit as $k \rightarrow \infty$ with $\varepsilon \approx 0$.  Furthermore, our experiments in Section~\ref{sec:experiments} suggest that $\RFSe$ may lead to estimators with good statistical properties for a wide range of values of $\delta$, provided that the value of $\delta$ is not too small. While $\RFSe$ by itself may be considered as a regularization scheme with excellent statistical properties, the boosting profile delivered by $\RFSe$ might in some cases be different from the \lasso coefficient profile, as we saw in Figure~\ref{fig:FSe-lasso-RFSe}.  Therefore in this section we investigate the following question:  is it possible to modify the $\RFSe$ algorithm, while still retaining its basic algorithmic characteristics, so that it delivers an approximate \lasso coefficient profile for any dataset?  We answer this question in the affirmative herein.\medskip

To fix ideas, let us consider producing the (approximate) \lasso path by producing a sequence of (approximate) \lasso solutions on a predefined grid of regularization parameter values $\delta$ in the interval $(0,\bar\delta]$ given by  $0 < \bar\delta_0 <  \bar\delta_1 <  \ldots < \bar\delta_K = \bar\delta$.  (A standard method for generating the grid points is to use a geometric sequence such as $\bar\delta_i = \eta^{-i} \cdot\bar\delta_0$ for $i=0, \ldots, K$, for some $\eta  \in (0,1)$.)  Motivated by the notion of warm-starts popularly used in the statistical computing literature in the context of computing a path of \lasso solutions \eqref{lasso-lag-1} via coordinate descent methods \cite{FHT2007}, we propose here a slight modification of the $\RFSe$ algorithm that sequentially updates the value of $\delta$ according to the predefined grid values $ \bar\delta_0,\bar\delta_1,\cdots,\bar\delta_K = \bar\delta$, and does so prior to each update of $\hat r^i$ and $\hat \beta^i$.  We call this method $\PATHe$, whose complete description is as follows:

\begin{center}
{\bf  Algorithm:} $\PATHe$
\end{center}

Fix the learning rate $\varepsilon > 0$,  choose values $\bar\delta_i$, $i= 0, \ldots, K$, satisfying $0 < \bar\delta_0 \le \bar\delta_1 \le \cdots \le \bar\delta_K \le \bar\delta$ such that $\varepsilon \leq \bar\delta_0$.

Initialize at $\hat r^0 = \by$, $\hat{\beta}^0 = 0$, $k = 0$ .

\bee
\item[\bf 1.] For $0 \leq k \leq K$ do the following:

\item[\bf 2.] Compute:
$j_k \in \argmax\limits_{j \in \{1, \ldots, p\}} |(\hat{r}^k)^T\bX_j|$

\item[\bf 3.] Set:
\begin{description}
\item $\hat r^{k+1} \gets \hat{r}^k - \varepsilon\left[\sgn((\hat{r}^k)^T\bX_{j_k})\bX_{j_k} + (\hat{r}^k - \by)/\bar\delta_k \right]$
\item $\hat{\beta}^{k+1}_{j_k} \gets \left(1 - \varepsilon/\bar\delta_{k}\right) \hat{\beta}^{k}_{j_k} + \varepsilon\ \sgn((\hat{r}^k)^T\bX_{j_k})$ and
$\hat{\beta}^{k+1}_j \gets \left(1 - \varepsilon/\bar\delta_{k}\right) \hat{\beta}^{k}_j \ , j \neq j_k$
\end{description}
\eee

\medskip

Notice that $\PATHe$ retains the identical structure of a forward stagewise regression type of method, and uses the same essential update structure of Step (3.) of $\RFSe$.  Indeed, the updates of $\hat r^{k+1}$ and $\hat{\beta}^{k+1}$ in $\PATHe$ are identical to those in Step (3.) of $\RFSe$ except that they use the regularization value  $\bar\delta_k$ at iteration $k$ instead of the constant value of $\delta$ as in $\RFSe$.

\paragraph{Theoretical Guarantees for $\PATHe$}
Analogous to Theorem \ref{RFSe-guarantees} for $\RFSe$, the following theorem describes properties of the $\PATHe$ algorithm.  In particular, the theorem provides rigorous guarantees about the distance between the $\PATHe$ algorithm and the \lasso coefficient profiles -- which apply to any general dataset. \medskip

\begin{theorem}\label{PATHe-guarantees} {\bf (Computational Guarantees of $\PATHe$)}  Consider the $\PATHe$ algorithm with the given learning rate $\varepsilon$ and regularization parameter sequence $\{\bar\delta_k\}$.  Let $k \ge 0$ denote the total number of iterations. Then the following holds:
\begin{itemize}
\item[(i)] (\lasso feasibility and average training error): for each $i=0, \ldots, k$, $\hat{\beta}^i$ provides an approximate solution to the \lasso problem for $\delta = \bar\delta_{i}$.  More specifically, $\hat{\beta}^i$ is feasible for the \lasso problem for $\delta = \bar\delta_{i}$, and satisfies the following suboptimality bound with respect to the entire boosting profile:
$$\displaystyle\frac{1}{k+1}\sum_{i=0}^k\left(L_n(\hat{\beta}^i) - L_{n, \bar\delta_i}^\ast\right) ~\leq~ \frac{\bar\delta\|\bX\hat{\beta}_{\text{LS}}\|_2^2}{2n\varepsilon(k+1)} + \frac{2\bar\delta\varepsilon}{n}$$
\item[(ii)] ($\ell_{1}$-shrinkage of coefficients): $\|\hat{\beta}^i\|_1 \le \bar\delta_i$ for $i=0, \ldots, k$.
\item[(iii)] (sparsity of coefficients): $\|\hat{\beta}^i\|_0 \le i$ for $i=0, \ldots, k$. \ \qed
\end{itemize}
\end{theorem}

\medskip

\begin{corollary} {\bf ($\PATHe$ approximates the \lasso path)}
For every fixed $\varepsilon > 0$ and $k \rightarrow \infty$  it holds that:
$$\limsup_{k \rightarrow \infty} \;\; \displaystyle\frac{1}{k+1}\sum_{i=0}^k\left(L_n(\hat{\beta}^i) - L_{n, \bar\delta_i}^\ast\right)  \leq  \frac{2\bar\delta\varepsilon}{n} \ , $$
(and the quantity on the right side of the above bound goes to zero as $\varepsilon \rightarrow 0$). \qed
\end{corollary}

The proof of Theorem \ref{PATHe-guarantees} is presented in Appendix \ref{redfolder}.

\paragraph{Interpreting the computational guarantees} Let us now provide some interpretation of the results stated in Theorem \ref{PATHe-guarantees}.
Recall that Theorem \ref{RFSe-guarantees} presented bounds on the distance between the training errors achieved by the boosting algorithm $\RFSe$ and
\lasso training errors for a \emph{fixed} but arbitrary $\delta$ that is specified {\it a priori}. The message in Theorem~\ref{PATHe-guarantees} generalizes this notion to a {\it family} of \lasso solutions
corresponding to a {\it grid} of $\delta$ values. The theorem thus quantifies how the boosting algorithm $\PATHe$ \emph{simultaneously} approximates a path  of \lasso solutions.

Part {\em (i)} of Theorem \ref{PATHe-guarantees} first implies that the sequence of regression coefficient vectors $\{\hat{\beta}^i\}$ is feasible along the \lasso path, for the \lasso problem \eqref{poi-lasso} for the sequence of regularization parameter values $\{\bar\delta_i\}$.  In considering guarantees with respect to the training error, we would ideally like guarantees that hold across the entire spectrum of $\{\bar\delta_i\}$ values.  While part {\em (i)} does not provide such strong guarantees, part {\em (i)} states that these quantities will be sufficiently small \emph{on average}.  Indeed, for a fixed $\varepsilon$ and as $k \to \infty$, part {\em (i)} states that the average of the differences between the training errors produced by the algorithm and the optimal training errors is at most $\frac{2\bar\delta\varepsilon}{n}$. This non-vanishing bound (for $\varepsilon>0$) is a consequence of the fixed learning rate $\varepsilon$ used in $\PATHe$ -- such bounds were also observed for $\RFSe$ and $\FSe$.\medskip

Thus on average, the training error of the model $\hat{\beta}^i$ will be sufficiently close (as controlled by the learning rate $\varepsilon$) to the optimal training error for the corresponding regularization parameter
 value $\bar\delta_i$.
 In summary, while $\PATHe$ provides the most amount of flexibility in terms of controlling for model complexity since it allows for \emph{any} (monotone) sequence of regularization parameter values in the range $(0, \bar\delta]$, this freedom comes at the cost of weaker training error guarantees with respect to any particular $\bar\delta_i$ value (as opposed to $\RFSe$ which provides strong guarantees with respect to the fixed value $\delta$). Nevertheless, part {\em (i)} guarantees that the training errors will be sufficiently small on average across the entire path of regularization parameter values explored by the algorithm.

\section{Some Computational Experiments}\label{sec:experiments}\label{1999}

We consider an array of examples exploring statistical properties of the different boosting algorithms studied herein.
We consider different types of synthetic and real datasets, which are briefly described here.

\paragraph{Synthetic datasets}
We considered synthetically generated datasets of the following types:
\begin{itemize}
\item {\bf{Eg-A.}} Here the data matrix $\bX$ is generated from a multivariate normal distribution, i.e., for each $i = 1, \ldots, n$, $\mathbf{x}_i \sim \text{MVN}(0 , \Sigma)$. Here $\mathbf{x}_i$ denotes the $i^\text{th}$ row of $\bX$ and $\Sigma = (\sigma_{ij}) \in \bbR^{p \times p}$ has all off-diagonal entries equal to $\rho$ and all diagonal entries equal to one. The response
$\by \in \bbR^{n}$ is generated as $\by = \bX \beta^\text{pop} + \epsilon$, where $\epsilon_{i} \stackrel{\text{iid}}{\sim} N(0, \sigma^2)$. The underlying regression coefficient was taken to be sparse with
$\beta^\text{pop}_{i} = 1$ for all $i \leq 5$ and $\beta^\text{pop}_{i}=0$ otherwise. $\sigma^2$ is chosen so as to control the signal to noise ratio $\text{SNR} := \text{Var}( \mathbf{x}'\beta)/\sigma^2.$

Different values of SNR, $n,p$ and $\rho$ were taken and they have been specified in our results when and where appropriate.

\item {\bf{Eg-B.}}  Here the datasets are generated similar to above, with $\beta^\text{pop}_{i} = 1$ for $i \leq 10$ and $\beta^\text{pop}_{i}=0$ otherwise.
We took the value of SNR=1in this example.
\end{itemize}

\paragraph{Real datasets}
We considered four different publicly available microarray datasets as described below.
\begin{itemize}
\item {\bf{Leukemia dataset.}} This dataset, taken from~\cite{dettling2003boosting}, was processed to have $n=72$ and $p=500$.
$\by$ was created as $\by = \bX\beta^\text{pop} + \epsilon$; with $\beta^\text{pop}_{i} =1$ for all $i \leq 10$ and zero otherwise.

\item {\bf{Golub dataset.}} This dataset, taken from the R package {\texttt{mpm}}, was processed to have $n=73$ and $p=500$, with artificial responses generated as above.

\item {\bf{Khan dataset.}} This dataset, taken from the website of~\cite{ESLBook}, was processed to have $n=73$ and $p=500$, with artificial responses generated as above.

\item {\bf {Prostate dataset.}}
This dataset, analyzed in~\cite{LARS}, was processed to create three types of different datasets:
(a) the original dataset with $n=97$ and $p=8$, (b) a dataset with $n=97$ and $p=44$, formed by extending the covariate space to include second order interactions, and
(c) a third dataset with $n=10$ and $p=44$, formed by subsampling the previous dataset.
\end{itemize}
For more detail on the above datasets, we refer the reader to the Appendix \ref{sec:append:comp-details}.

Note that in all the examples we standardized $\bX$ such that the columns have unit $\ell_{2}$ norm, before running the different algorithms studied herein.

\subsection{Statistical properties of boosting algorithms: an empirical study}
We performed some experiments to better understand the statistical behavior of the different boosting methods described in this paper. We summarize our findings here; for details (including tables, figures and discussions) we refer the reader to
Appendix, Section~\ref{sec:append:comp-details}.

\paragraph{Sensitivity of the Learning Rate in \ebs and $\FSe$}
We explored how the training and test errors for \ebs and $\FSe$ change as a function of
the number of boosting iterations and the learning rate.  We observed that the best predictive models were sensitive to the choice of $\varepsilon$---the best models were obtained at values larger than zero and
smaller than one. When compared to \lasso, stepwise regression~\cite{LARS} and $\FS_0$~\cite{LARS}; $\FSe$ and \ebs were found to be as good as the others, in some cases the better than the rest.

\paragraph{Statistical properties of  $\RFSe$, \lasso and $\FSe$: an empirical study} We performed some experiments to evaluate the performance of $\RFSe$, in terms of predictive accuracy and
sparsity of the optimal model, versus the more widely known methods $\FSe$ and \lassoperiod.
We found that when $\delta$ was larger than the best $\delta$ for the \lasso (in terms of obtaining a model with the best predictive performance), $\RFSe$ delivered a model
with excellent statistical properties -- $\RFSe$ led to sparse solutions and the predictive performance was as good as, and in some cases better than, the \lasso solution.
We observed that the choice of $\delta$ does not play a very crucial role in the $\RFSe$ algorithm, once it is chosen to be reasonably large; indeed the number of boosting iterations play a more important role.
The best models delivered by $\RFSe$ were more sparse than $\FSe$.

\subsection*{Acknowledgements}
The authors will like to thank Alexandre Belloni, Jerome Friedman, Trevor Hastie, Arian Maleki and Tomaso Poggio for helpful discussions and encouragement. A preliminary unpublished version of some of the results herein was posted on the ArXiv~\cite{BoostFreundGM13}.

\bibliographystyle{abbrv}

{\small{\bibliography{../../GFM-papers}}}

\clearpage

\appendix

\section{Technical Details and Supplementary Material}\label{sec-appendix}

\subsection{Additional Details for Section~1}

\medskip

\subsubsection{Figure showing Training error versus $\ell_{1}$-shrinkage bounds}\label{dance}
Figure~\ref{fig:bounds-lasso} showing profiles of $\ell_{1}$ norm of the regression coefficients  versus training error for \ebs, $\FSe$ and \lassoperiod.
\begin{figure}[h!]
\begin{center}
\scalebox{0.95}[0.85]{\begin{tabular}{l c c c}
\multicolumn{4}{c}{$\ell_{1}$ shrinkage versus data-fidelity tradeoffs: \ebs, $\FSe$, and \lasso}\\
&&&\\
\rotatebox{90}{\sf {\scriptsize {~~~~~~~~~~~~~~~~~~~~~~~Training Error}}}&\includegraphics[width=0.31\textwidth,height=0.25\textheight,  trim = 1.1cm 1.2cm 1cm 2cm, clip = true ]{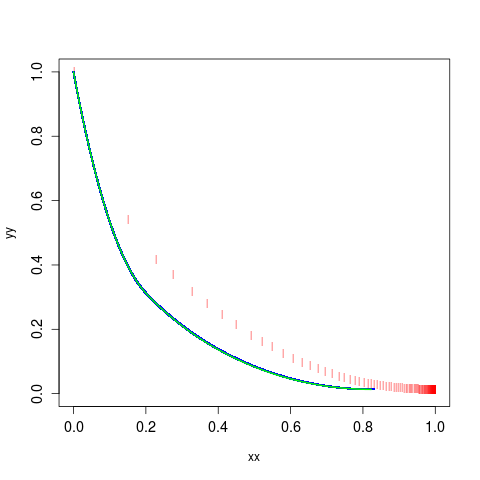}&
\includegraphics[width=0.31\textwidth,height=0.25\textheight,  trim = 1.1cm 1.2cm 1cm 2cm, clip = true ]{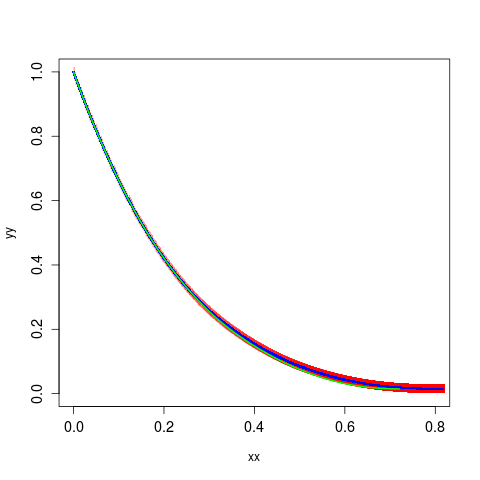}&
\includegraphics[width=0.31\textwidth,height=0.25\textheight,  trim = 1.1cm 1.2cm 1cm 2cm, clip = true ]{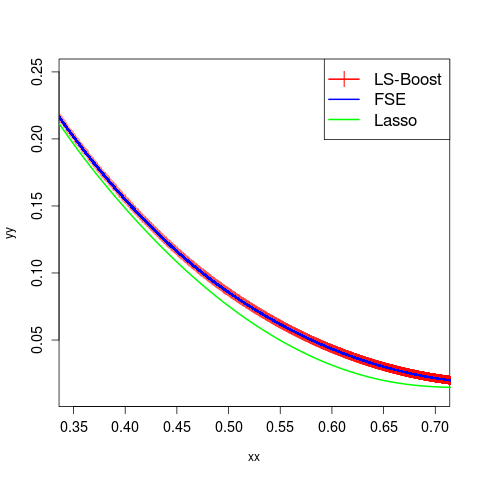} \\
&\scriptsize{\sf {$\ell_{1}$ shrinkage of coefficients}} &\scriptsize{\sf {$\ell_{1}$ shrinkage of coefficients}}  & \scriptsize{\sf {$\ell_{1}$ shrinkage of coefficients}}  \\
\end{tabular}}
\caption{{\small{Figure showing profiles of $\ell_{1}$ norm of the regression coefficients  versus training error for \ebs, $\FSe$ and \lassoperiod.  [Left panel] Shows profiles for a synthetic dataset where the covariates are drawn from a Gaussian distribution with pairwise correlations $\rho=0.5$.  The true $\beta$ has ten non-zeros with $\beta_{i} =1$ for $i=1, \ldots, 10$, and $\text{SNR}=1$.  Here we ran \ebs with $\varepsilon=1$ and ran $\FSe$ with $\varepsilon=10^{-2}$.  The middle (and right) panel profiles corresponds to the Prostate cancer dataset (described in Section~\ref{sec:experiments}). Here  we ran \ebs with $\varepsilon=0.01$ and we ran $\FSe$ with $\varepsilon=10^{-5}$.  The right panel figure is a zoomed-in version of the middle panel in order to highlight the difference in profiles between \ebs, $\FSe$ and \lassoperiod.
The vertical axes have been normalized so that the training error at $k=0$ is one, and the horizontal axes have been scaled to the unit interval (to express the $\ell_{1}$-norm of $\hat{\beta}^k$ as a fraction of the maximum).}} }\label{fig:bounds-lasso}
\end{center}
\end{figure}

\subsection{Additional Details for Section~2}

\subsubsection{Properties of Convex Quadratic Functions}\label{appendix_quad}

Consider the following quadratic optimization problem (QP) defined as:
$$h^* := \min\limits_{x \in \mathbb{R}^n} h(x) :=  \tfrac{1}{2}x^TQx + q^Tx + q^o \ , $$
where $Q$ is a symmetric positive semi-definite matrix, whereby $h(\cdot)$ is a convex function.  We assume that $Q \ne 0$, and recall that $\lambda_{\pmin}(Q)$ denotes the smallest nonzero (and hence positive) eigenvalue of $Q$.\medskip

\begin{proposition}\label{qp}
If $h^* > -\infty$, then for any given $x$, there exists an optimal solution $x^*$ of (QP) for which
$$\|x-x^*\|_2 \le \sqrt{\frac{2(h(x) - h^*)}{\lambda_{\pmin}(Q)}} \ . $$
Also, it holds that $$\|\nabla h(x)\|_2 \ge \sqrt{\frac{\lambda_{\pmin}(Q)\cdot(h(x)-h^*)}{2}} \ . $$
\end{proposition}

\noindent {\bf Proof:} The result is simply manipulation of linear algebra.  Let us assume without loss of generality that $q^o = 0$.  If $h^* > -\infty$, then (QP) has an optimal solution $x^*$, and the set of optimal solutions are characterized by the gradient condition $$0 = \nabla h(x) = Qx +q \ . $$  Now let us write the sparse eigendecomposition of $Q$ as $Q=PDP^T$ where $D$ is a diagonal matrix of non-zero eigenvalues of $Q$ and the columns of $P$ are orthonormal, namely $P^TP=I$.  Because (QP) has an optimal solution, the system of equations $Qx = -q$ has a solution, and let $\tilde x$ denote any such solution.  Direct manipulation establishes:
$$PP^Tq = -PP^TQ\tilde x = -PP^TPDP^T\tilde x =  -PDP^T\tilde x = - Q\tilde x = q \ . $$
Furthermore, let $\hat x:=-PD^{-1}P^Tq$.  It is then straightforward to see that $\hat x$ is an optimal solution of (QP) since in particular:

$$Q\hat x = -PDP^TPD^{-1}P^Tq = -PP^Tq = -q \ , $$
and hence
$$h^* = \tfrac{1}{2}\hat x^TQ\hat x + q^T\hat x = -\tfrac{1}{2}\hat x^TQ\hat x = -\tfrac{1}{2}q^TPD^{-1}P^TPDP^TPD^{-1}P^Tq = -\tfrac{1}{2}q^TPD^{-1}P^Tq \ . $$

Now let $x$ be given, and define $x^*:=[I-PP^T]x -PD^{-1}P^Tq$.  Then just as above it is straightforward to establish that $Qx^*=-q$ whereby $x^*$ is an optimal solution.  Furthermore, it holds that:
$$\begin{array}{rcl}
\|x - x^*\|_2^2 &=& (q^TPD^{-1}+x^TP)P^TP(D^{-1}P^Tq + P^Tx) \\ \\
&=& (q^TPD^{-\tfrac{1}{2}} + x^TPD^{\tfrac{1}{2}})D^{-1}(D^{-\tfrac{1}{2}}P^Tq + D^{\tfrac{1}{2}}P^Tx) \\ \\
&\le& \tfrac{1}{\lambda_{\pmin}(Q)}(q^TPD^{-\tfrac{1}{2}} + x^TPD^{\tfrac{1}{2}})(D^{-\tfrac{1}{2}}P^Tq + D^{\tfrac{1}{2}}P^Tx)\\ \\
&=& \tfrac{1}{\lambda_{\pmin}(Q)}(q^TPD^{-1}P^Tq + x^TPDP^Tx + 2 x^TPP^Tq) \\ \\
&=& \tfrac{1}{\lambda_{\pmin}(Q)}(-2h^* + x^TQx + 2 x^Tq) \\ \\
&=& \tfrac{2}{\lambda_{\pmin}(Q)}(h(x) - h^*) \ ,
\end{array}$$ and taking square roots establishes the first inequality of the proposition.

Using the gradient inequality for convex functions, it holds that:
$$\begin{array}{rcl}
h^* = h(x^*) & \ge & h(x) + \nabla h(x)^T(x^*-x) \\ \\
&\ge & h(x) - \|\nabla h(x)\|_2 \|x^*- x\|_2 \\ \\
&\ge & h(x) - \|\nabla h(x)\|_2 \sqrt{\frac{2(h(x) - h^*)}{\lambda_{\pmin}(Q)}} \ ,
\end{array}$$ and rearranging the above proves the second inequality of the proposition.\qed

\subsubsection{Proof of Theorem \ref{august24-2}}\label{ebs_main_proof}
We first prove part {\em (i)}. Utilizing \eqref{ebs_res_update}, which states that $\hat{r}^{k+1} = \hat{r}^k -\varepsilon \left((\hat{r}^k)^T\bX_{j_k} \right)\bX_{j_k}$, we have:
\begin{equation}\label{amy}\begin{array}{rcl}
L_n(\hat{\beta}^{k+1}) &=& \tfrac{1}{2n}\|\hat{r}^{k+1}\|_2^2 \\ \\
&=& \tfrac{1}{2n}\|\hat{r}^k -\varepsilon \left((\hat{r}^k)^T\bX_{j_k}\right)\bX_{j_k}  \|_2^2 \\ \\
&=& \tfrac{1}{2n}\|\hat{r}^k\|_2^2 -\tfrac{1}{n}\varepsilon \left((\hat{r}^k)^T\bX_{j_k} \right)^2 + \tfrac{1}{2n}\varepsilon^2 \left((\hat{r}^k)^T\bX_{j_k} \right)^2\\ \\
&=& L_n(\hat{\beta}^{k}) -\tfrac{1}{2n}\varepsilon(2-\varepsilon) \left((\hat{r}^k)^T\bX_{j_k} \right)^2  \\ \\
&=& L_n(\hat{\beta}^{k}) -\tfrac{1}{2n}\varepsilon(2-\varepsilon) n^2 \|\nabla L_n(\hat{\beta}^k)\|_\infty^2  \ , \\
\end{array}\end{equation}
(where the last equality above uses \eqref{drew2}), which yields:
\begin{equation}\label{drew3}
L_n(\hat{\beta}^{k+1}) - L_n^* =  L_n(\hat{\beta}^{k}) -L_n^* -\tfrac{n}{2}\varepsilon(2-\varepsilon) \|\nabla L_n(\hat{\beta}^k)\|_\infty^2 \ .
\end{equation}
We next seek to bound the right-most term above.  We will do this by invoking Proposition \ref{qp}, which presents two important properties of convex quadratic functions.  Because $L_n(\cdot)$ is a convex quadratic function of the same format as Proposition \ref{qp} with $h(\cdot) \gets L_n(\cdot)$, $Q \gets \tfrac{1}{n}\bX^T\bX$, and $h^* \gets L_n^*$, it follows from the second property of Proposition \ref{qp} that
$$\|\nabla L_n(\beta)\|_2 \ge \sqrt{\frac{\lambda_\pmin (\tfrac{1}{n}\bX^T\bX) (L_n(\beta) -L_n^*)}{2}} = \sqrt{\frac{\lambda_\pmin (\bX^T\bX) (L_n(\beta) -L_n^*)}{2n}} \ . $$
Therefore
$$\|\nabla L_n(\beta)\|_\infty^2 \ge \tfrac{1}{p}\|\nabla L_n(\beta)\|_2^2 \ge \frac{\lambda_\pmin (\bX^T\bX) (L_n(\beta) -L_n^*)}{2np} \ . $$
Substituting this inequality into \eqref{drew3} yields after rearranging:
\begin{equation}\label{drew4}
L_n(\hat{\beta}^{k+1}) - L_n^* \le  (L_n(\hat{\beta}^{k}) -L_n^*)\left(1-\frac{\varepsilon(2-\varepsilon)\lambda_{\pmin}(\bX^T\bX)}{4p} \right) = (L_n(\hat{\beta}^{k}) -L_n^*) \cdot \gamma \ .
\end{equation}
Now note that $L_n(\hat{\beta}^0) = L_n(0)= \frac{1}{2n}\|\by\|_2^2$ and
\begin{equation*}
L_n(\hat{\beta}^0) - L_n^\ast = \tfrac{1}{2n}\|\by\|_2^2 - \tfrac{1}{2n}\|\by - \bX\hat{\beta}_{\text{LS}}\|_2^2 = \tfrac{1}{2n}\|\by\|_2^2 - \tfrac{1}{2n}(\|\by\|_2^2 - 2\by^T\bX\hat{\beta}_{\text{LS}} + \|\bX\hat{\beta}_{\text{LS}}\|_2^2) = \tfrac{1}{2n}\|\bX\hat{\beta}_{\text{LS}}\|_2^2 \ ,
\end{equation*}
where the last equality uses the normal equations \eqref{grad2}. Then {\em(i)} follows by using elementary induction and combining the above with \eqref{drew4}:
$$L_n(\hat{\beta}^k) - L_n^\ast \le (L_n(\hat{\beta}^0) - L_n^\ast) \cdot \gamma^k = \tfrac{1}{2n}\|\bX\hat{\beta}_{\text{LS}}\|_2^2 \cdot \gamma^k \ . $$

To prove {\em (ii)}, we invoke the first inequality of Proposition \ref{qp}, which in this context states that
$$\|\hat{\beta}^k -\hat{\beta}_{\text{LS}}\|_2 \le \frac{\sqrt{2(L_n(\hat{\beta}^k) - L_n^*)}}{\sqrt{\lambda_{\pmin}(\tfrac{1}{n}\bX^T\bX)}} = \frac{\sqrt{2n(L_n(\hat{\beta}^k) - L_n^*)}}{\sqrt{\lambda_{\pmin}(\bX^T\bX)}} \ . $$ Part {\em (ii)} then follows by substituting the bound on $(L_n(\hat{\beta}^k) - L_n^*)$ from {\em (i)} and simplifying terms.  Similarly, the proof of {\em (iii)} follows from the observation that
$\|\bX\hat{\beta}^k- \bX\hat{\beta}_{\text{LS}} \|_2 = \sqrt{2n(L_n(\hat{\beta}^k) - L_n^*)}$ and then substituting the bound on $(L_n(\hat{\beta}^k) - L_n^*)$ from {\em (i)} and simplifying terms.

To prove {\em (iv)}, define the point $\tilde{\beta}^k := \hat{\beta}^k + \tilde{u}_{j_k}e_{j_k}$.  Then using similar arithmetic as in \eqref{amy} one obtains:
$$L_n^* \le L_n(\tilde{\beta}^k) = L_n(\hat{\beta}^k) - \tfrac{1}{2n}\tilde{u}_{j_k}^2 \ , $$
where we recall that $\tilde{u}_{j_k} = (\hat{r}^k)^T\bX_{j_k}$. This inequality then rearranges to
\begin{equation}\label{step_size_bound}
|\tilde{u}_{j_k}| ~\leq~ \sqrt{2n(L_n(\hat{\beta}^k)-L_n^*)} ~\leq~ \|\bX\hat{\beta}_{\text{LS}}\|_2 \cdot \gamma^{k/2} \ ,
\end{equation}
where the second inequality follows by substituting the bound on $(L_n(\hat{\beta}^i) - L_n^*)$ from {\em (i)}. Recalling \eqref{grad_norm} and \eqref{drew2}, the above is exactly part {\em (iv)}.

Part {\em (v)} presents two distinct bounds on $\|\hat{\beta}^k\|_1$, which we prove independently. To prove the first bound, let $\hat{\beta}_{\text{LS}}$ be any least-squares solution, which therefore satisfies \eqref{grad2}.  It is then elementary to derive using similar manipulation as in \eqref{amy} that for all $i$ the following holds:
\begin{equation}\label{there}\|\bX(\hat{\beta}^{i+1} - \hat{\beta}_{\text{LS}})\|_2^2 = \|\bX(\hat{\beta}^{i} - \hat{\beta}_{\text{LS}})\|_2^2 - (2\varepsilon - \varepsilon^2) \tilde{u}_{j_i}^2 \end{equation}
which implies that
\begin{equation}\label{subgrad_sum} (2\varepsilon - \varepsilon^2)\sum_{i=0}^{k-1} \tilde{u}_{j_i}^2 = \|\bX(\hat{\beta}^{0} - \hat{\beta}_{\text{LS}})\|_2^2 - \|\bX(\hat{\beta}^{k} - \hat{\beta}_{\text{LS}})\|_2^2 = \|\bX\hat{\beta}_{\text{LS}}\|_2^2 - \|\bX(\hat{\beta}^{k} - \hat{\beta}_{\text{LS}})\|_2^2 \ . \end{equation}
Then note that $$\|\hat{\beta}^k\|_1 \le \|(\varepsilon\tilde{u}_{j_0}, \ldots, \varepsilon\tilde{u}_{j_{k-1}})\|_1 \le \sqrt{k}\varepsilon\|(\tilde{u}_{j_0}, \ldots, \tilde{u}_{j_{k-1}})\|_2 = \sqrt{k}\sqrt{\tfrac{\varepsilon}{2-\varepsilon}}\sqrt{\|\bX\hat{\beta}_{\text{LS}}\|^2_2 - \|\bX\hat{\beta}_{\text{LS}}-\bX\hat{\beta}^k\|^2_2} \ , $$ where the last equality is from \eqref{subgrad_sum}.

To prove the second bound in {\em (v)}, noting that $\hat{\beta}^k = \sum_{i = 0}^{k-1}\varepsilon\tilde{u}_{j_i}e_{j_i}$, we bound $\|\hat{\beta}^k\|_1$ as follows:
\begin{align*}
\|\hat{\beta}^k\|_1 \leq \varepsilon\sum_{i = 0}^{k-1}|\tilde{u}_{j_i}| ~&\leq~ \varepsilon\|\bX\hat{\beta}_{\text{LS}}\|_2\sum_{i = 0}^{k-1}\gamma^{i/2}\\
&= \frac{\varepsilon\|\bX\hat\beta_{LS}\|_2}{1 - \sqrt{\gamma}}\left(1 - \gamma^{k/2}\right) \ ,
\end{align*}
where the second inequality uses \eqref{step_size_bound} for each $i \in \{0,\ldots,k-1\}$ and the final equality is a geometric series, which completes the proof of {\em (v)}.  Part {\em (vi)} is simply the property of \ebs~that derives from the fact that $\hat{\beta}^0:=0$ and at every iteration at most one coordinate of $\beta$ changes status from a zero to a non-zero value.  \qed

\subsubsection{Additional properties of \textsc{LS-Boost}$(\varepsilon)$}\label{sec-add-results-ebs}
We present two other interesting properties of the \ebs~algorithm, namely an additional bound on the correlation between residuals and predictors, and a bound on the $\ell_{2}$-shrinkage of the regression coefficients.  Both are presented in the following proposition.\medskip

\begin{proposition}{\bf (Two additional properties of \textsc{LS-Boost}$(\varepsilon)$)}\label{prop-ebs-grad-shr-1}
Consider the iterates of the \ebs algorithm after $k$ iterations and consider the linear convergence rate coefficient $\gamma$:
$$ \gamma := \left(1-\frac{\varepsilon(2-\varepsilon)\lambda_{\pmin}(\bX^T\bX)}{4p} \right).$$
\item[(i)] There exists an index $i \in \{0, \ldots, k\}$ for which the $\ell_{\infty}$ norm of the gradient vector of the least squares loss function evaluated at $\hat{\beta}^i$ satisfies:
\begin{equation}\label{four}
\small{\|\nabla L_n(\hat \beta^i)\|_\infty \ \ = \ \ \tfrac{1}{n}\|\bX^T\hat{r}^i\|_\infty \ \  \le \ \ \min\left\{\frac{\sqrt{\|\bX\hat{\beta}_{\text{LS}}\|^2_2 - \|\bX\hat{\beta}_{\text{LS}}-\bX\hat{\beta}^{k+1}\|^2_2}}{n \sqrt{\varepsilon(2 - \varepsilon)(k+1)}} \ , \tfrac{1}{n}\|\bX\hat{\beta}_{\text{LS}}\|_2 \cdot \gamma^{k/2}\right\} \ .}
\end{equation}
\item[(ii)] Let $J_\ell$ denote the number of iterations of \ebs, among the first $k$ iterations, where the algorithm takes a step in
coordinate $\ell$, for $\ell=1, \ldots, p$, and let $J_{\max}:=\max\{J_1, \ldots, J_p\}$.  Then the following bound on the shrinkage of $\hat{\beta}^k$ holds:
\begin{equation}\label{bound-l2-norm-ebs}
\|\hat{\beta}^k\|_2 \le  \sqrt{J_{\max}}\sqrt{\frac{\varepsilon}{2-\varepsilon}}\sqrt{\|\bX\hat{\beta}_{\text{LS}}\|^2_2 - \|\bX\hat{\beta}_{\text{LS}}-\bX\hat{\beta}^k\|^2_2} \ .
\end{equation} \qed

\end{proposition}
\begin{proof}
We first prove part {\em (i)}.  The first equality of \eqref{four} is a restatement of \eqref{drew2}.  For each $i \in \{0, \ldots, k\}$, recall that $\tilde{u}_{j_i} = (\hat{r}^i)^T\bX_{j_i}$ and that $|\tilde{u}_{j_i}| = |(\hat{r}^i)^T\bX_{j_i}| = \|\bX^T\hat{r}^i\|_\infty$, from \eqref{drew2}. Therefore:
\begin{equation}\label{desk}
\left(\min_{i \in \{0, \ldots, k\}}|\tilde{u}_{j_i}|\right)^2 = \min_{i \in \{0, \ldots, k\}}\tilde{u}_{j_i}^2 \leq \frac{1}{k+1}\sum_{i = 0}^k \tilde{u}_{j_i}^2 \leq \frac{\|\bX\hat{\beta}_{\text{LS}}\|_2^2 - \|\bX(\hat{\beta}^{k+1} - \hat{\beta}_{\text{LS}})\|_2^2}{\varepsilon(2-\varepsilon)(k+1)} \ ,
\end{equation}
where the final inequality follows from \eqref{subgrad_sum} in the proof of Theorem \ref{august24-2}. Now letting $i$ be an index achieving the minimum in the left hand side of the above and taking square roots implies that
\begin{equation*}
\|\bX^T\hat{r}^i\|_\infty = |\tilde{u}_{j_i}| \leq \frac{\sqrt{\|\bX\hat{\beta}_{\text{LS}}\|^2_2 - \|\bX\hat{\beta}_{\text{LS}}-\bX\hat{\beta}^{k+1}\|^2_2}}{\sqrt{\varepsilon(2 - \varepsilon)(k+1)}} \ ,
\end{equation*} which is equivalent to the inequality in \eqref{four} for the first right-most term therein.
Directly applying \eqref{step_size_bound} from the proof of Theorem \ref{august24-2} and using the fact that $i$ is an index achieving the minimum in the left hand side of \eqref{desk} yields:
\begin{equation*}
\|\bX^T\hat{r}^i\|_\infty = |\tilde{u}_{j_i}| \leq |\tilde{u}_{j_k}| \leq \|\bX\hat{\beta}_{\text{LS}}\|_2\left(1-\frac{\varepsilon(2-\varepsilon)\lambda_{\pmin}(\bX^T\bX)}{4p} \right)^{k/2} \ ,
\end{equation*} which is equivalent to the inequality in \eqref{four} for the second right-most term therein.

We now prove part {\em (ii)}.  For fixed $k > 0$, let ${\cal J}(\ell)$ denote the set of iteration counters where \ebs~modifies coordinate $\ell$ of $\beta$, namely $${\cal J}(\ell) := \{i: i<k \ \mbox{and} \ j_i = \ell \ \mbox{in~Step (2.)~of~Algorithm~\mbox{\ebs}} \} \ , $$ for $\ell=1, \ldots, p$.  Then $J_\ell = |{\cal J}(\ell)|$, and the sets ${\cal J}(1), \ldots, {\cal J}(p)$ partition the iteration index set $\{0, 1, \ldots, k-1\}$.  We have:
\begin{equation}\label{whynot}\begin{array}{rcl}
\|\hat{\beta}^k\|_2 &\le& \|(\sum_{i \in {\cal J}(1)} \varepsilon \tilde{u}_{j_i}, \ldots, \sum_{i \in {\cal J}(p)} \varepsilon \tilde{u}_{j_i}) \|_2\\ \\
&\le& \left\|\left(\sqrt{J(1)}\sqrt{\sum_{i \in {\cal J}(1)} \varepsilon^2 \tilde{u}_{j_i}^2}, \ldots, \sqrt{J(p)}\sqrt{\sum_{i \in {\cal J}(p)} \varepsilon^2 \tilde{u}_{j_i}^2}\right) \right\|_2\\ \\
&\le& \varepsilon\sqrt{J_{\max}}\left\|\left(\sqrt{\sum_{i \in {\cal J}(1)} \tilde{u}_{j_i}^2}, \ldots, \sqrt{\sum_{i \in {\cal J}(p)} \tilde{u}_{j_i}^2}\right) \right\|_2\\ \\
&=& \varepsilon \sqrt{J_{\max}}\sqrt{\left(\tilde{u}_{j_0}^2 + \ldots + \tilde{u}_{j_{k-1}}^2\right)} \ ,
\end{array}\end{equation}and the proof is completed by applying inequality \eqref{subgrad_sum}.
\end{proof}

Part {\em (i)} of Proposition~\ref{prop-ebs-grad-shr-1} describes the behavior of the gradient of the least squares loss function --- indeed, recall that the dynamics of the gradient are closely linked to that
of the \ebs algorithm and, in particular, to the evolution of the loss function values. To illustrate this connection, let us recall two simple characteristics of
the \ebs algorithm:
\begin{equation}\nonumber
\begin{myarray}[1.0]{ccc}
L_n(\hat{\beta}^{k}) -  L_n(\hat{\beta}^{k+1}) &=&  \tfrac{n}{2}\varepsilon(2-\varepsilon) \|\nabla L_n(\hat{\beta}^k)\|_\infty^2 \\ \\
  \hat{r}^{k+1} -  \hat{r}^{k}  &=& - \varepsilon \left( (\hat{r}^k)^T\bX_{j_k}\right) \M{X}_{j_k} \ ,
\end{myarray}
\end{equation}
which follow from~\eqref{drew3} and Step (3.) of the $\FSe$ algorithm respectively.
The above updates of the \ebs algorithm clearly show that smaller values of the $\ell_{\infty}$ norm of the gradient
slows down the ``progress'' of the residuals and thus the overall algorithm. Larger values of the norm of the gradient, on the other hand,
lead to rapid ``progress''  in the algorithm. Here, we use the term ``progress'' to measure the amount of decrease in training error and the norm of the changes in successive residuals.
Informally speaking, the \ebs algorithm operationally works towards minimizing the unregularized least squares loss function --- and the gradient of the least squares loss function is simultaneously shrunk towards
zero.  Equation~\eqref{four} precisely quantifies the rate at which the $\ell_{\infty}$ norm of the gradient converges to zero.
Observe that the bound is a minimum of two distinct rates: one which decays as $O(\tfrac{1}{\sqrt{k}})$ and another which is linear with parameter $\sqrt{\gamma}$. This is similar to item {\em (v)} of Theorem~\ref{august24-2}. For small values of $k$ the first rate will dominate, until a point is reached where the linear rate begins to dominate. Note that the dependence on the linear rate $\gamma$ suggests that for large values of correlations among the samples, the gradient decays slower than for smaller pairwise correlations among the samples.

The behavior of the \ebs algorithm described above should be contrasted with the $\FSe$ algorithm. In view of Step (3.) of the $\FSe$ algorithm, the successive differences
of the residuals in $\FSe$ are indifferent to the magnitude of the gradient of the least squares loss function --- as long as the gradient is non-zero, then for $\FSe$ it holds that
$ \| \hat{r}^{k+1} - \hat{r}^k \|_{2}  = \varepsilon$.  Thus $\FSe$ undergoes a more erratic evolution, unlike \ebs where the convergence of the residuals is much more ``smooth.''

\subsubsection{Concentration Results for $\lambda_{\pmin}(\bX^T\bX)$ in the High-dimensional Case}\label{random_matrix}
\begin{proposition}
Suppose that $p > n$, let $\tilde\bX \in \bbR^{n \times p}$ be a random matrix whose entries are i.i.d. standard normal random variables, and define $\bX := \frac{1}{\sqrt{n}}\tilde\bX$. Then, it holds that:\begin{equation*}
\bbE[\lambda_{\pmin}(\bX^T\bX)] \geq \frac{1}{n}\left(\sqrt{p} - \sqrt{n}\right)^2 \ .
\end{equation*}
Furthermore, for every $t \geq 0$, with probability at least $1 - 2 \exp(-t^2/2)$, it holds that:
\begin{equation*}
\lambda_{\pmin}(\bX^T\bX) \geq \frac{1}{n}\left(\sqrt{p} - \sqrt{n} - t\right)^2 \ .
\end{equation*}
\end{proposition}
\begin{proof}
Let $\sigma_1(\tilde\bX^T) \geq \sigma_2(\tilde\bX^T) \geq \ldots \geq \sigma_n(\tilde\bX^T)$ denote the ordered singular values of $\tilde\bX^T$ (equivalently of $\tilde\bX$). Then, Theorem 5.32 of \cite{vershynin2010introduction} states that:
\begin{equation*}
\bbE[\sigma_n(\tilde\bX^T)] \geq \sqrt{p} - \sqrt{n} \ ,
\end{equation*}
which thus implies:
\begin{equation*}
\bbE[\lambda_{\pmin}(\bX^T\bX)] = \bbE[(\sigma_n(\bX^T))^2] \geq (\bbE[\sigma_n(\bX^T)])^2 = \frac{1}{n}(\bbE[\sigma_n(\tilde\bX^T)])^2 \geq \frac{1}{n}\left(\sqrt{p} - \sqrt{n}\right)^2 \ ,
\end{equation*}
where the first inequality is Jensen's inequality.

Corollary 5.35 of \cite{vershynin2010introduction} states that, for every $t \geq 0$, with probability at least $1 - 2 \exp(-t^2/2)$ it holds that:
\begin{equation*}
\sigma_n(\tilde\bX^T) \geq \sqrt{p} - \sqrt{n} - t \ ,
\end{equation*}
which implies that:
\begin{equation*}
\lambda_{\pmin}(\bX^T\bX) = (\sigma_n(\bX^T))^2 = \frac{1}{n}(\sigma_n(\tilde\bX^T))^2 \geq \frac{1}{n}\left(\sqrt{p} - \sqrt{n} - t\right)^2 \ .
\end{equation*}
\end{proof}

Note that, in practice, we standardize the model matrix $\bX$ so that its columns have unit $\ell_2$ norm. Supposing that the entries of $\bX$ did originate from an i.i.d. standard normal matrix $\tilde\bX$, standardizing the columns of $\tilde\bX$ is not equivalent to setting $\bX := \frac{1}{\sqrt{n}}\tilde\bX$. But, for large enough $n$, standardizing is a valid approximation to normalizing by $\frac{1}{\sqrt{n}}$, i.e., $\bX \approx \frac{1}{\sqrt{n}}\tilde\bX$, and we may thus apply the above results.

\subsection{Additional Details for Section~3}

\subsubsection{An Elementary Sequence Process Result, and a Proof of Proposition \ref{subgrad}}\label{simple}
Consider the following elementary sequence process:  $x^0 \in \mathbb{R}^n$ is given, and $x^{i+1} \gets x^i - \alpha_i g^i$ for all $i \ge 0$, where $g^i \in \mathbb{R}^n$ and $\alpha_i $ is a nonnegative scalar, for all $i$.  For this process there are {\em no} assumptions on how the vectors $g^i$ might be generated.\medskip

\begin{proposition}\label{simpleprop}
For the elementary sequence process described above, suppose that the $\{g^i\}$ are uniformly bounded, namely $\|g^i\|_2 \leq G$ for all $i \geq 0$. Then for all $k \geq 0$ and for any $x \in \mathrm{R}^n$ it holds that:
\begin{equation}\label{toomuchsnow}
\frac{1}{\sum_{i=0}^k\alpha_i}\sum_{i = 0}^k\alpha_i(g^i)^T(x^i - x) \ \ \leq \ \  \frac{\|x^0-x\|_2^2+ G^2\sum_{i = 0}^k\alpha_i^2}{2\sum_{i=0}^k\alpha_i} \ .
\end{equation}
Indeed, in the case when $\alpha_i = \varepsilon$ for all $i$, it holds that:
\begin{equation}\label{waytoomuchsnow}
\frac{1}{k+1}\sum_{i = 0}^k(g^i)^T(x^i - x) \ \ \leq \ \  \frac{\|x^0-x\|_2^2}{2(k+1)\varepsilon}+ \frac{G^2\varepsilon}{2} \ .
\end{equation}

\end{proposition}\medskip
\begin{proof} Elementary arithmetic yields the following:
$$\begin{array}{rcl}
\|x^{i+1} - x\|_2^2 &=& \|x^{i} -\alpha_i g^i- x\|_2^2\\ \\
&= & \|x^{i} - x \|_2^2 +\alpha_i^2 \|g^i\|^2_2 + 2 \alpha_i (g^i)^T(x-x^i)\\ \\
&\le & \|x^{i} - x \|_2^2 +G^2\alpha_i^2 + 2 \alpha_i (g^i)^T(x-x^i) \ . \\ \\
\end{array}$$ Rearranging and summing these inequalities for $i=0, \ldots, k$ then yields:
$$2 \sum_{i=0}^k \alpha_i (g^i)^T(x^i -x ) \le G^2 \sum_{i=0}^k \alpha_i^2 +\|x^0 - x\|_2^2 - \|x^{k+1}-x\|_2^2 \le G^2 \sum_{i=0}^k \alpha_i^2 +\|x^0 - x\|_2^2 \ , $$ which then rearranges to yield \eqref{toomuchsnow}.  \eqref{waytoomuchsnow} follows from \eqref{toomuchsnow} by direct substitution.\end{proof}\medskip

\noindent {\em Proof of Proposition \ref{subgrad}:}
Consider the subgradient descent method \eqref{subgrad_update2} with arbitrary step-sizes $\alpha_i$ for all $i$.  We will prove the following inequality:
\begin{equation}\label{sd_bound11}
\min_{i \in \{0,\ldots,k\}} f(\x^i)  \ \ \leq \ \ f^* \ + \ \frac{\|\x^0-\x^*\|_2^2+ G^2\sum_{i = 0}^k\alpha_i^2}{2\sum_{i=0}^k\alpha_i} \ ,
\end{equation} from which the proof of Proposition \ref{subgrad} follows by substituting $\alpha_i = \alpha$ for all $i$ and simplifying terms.  Let us now prove \eqref{sd_bound11}.  The subgradient descent method \eqref{subgrad_update2} is applied to instances of problem \eqref{poi1} where $f(\cdot)$ is convex, and where $g^i$ is subgradient of $f(\cdot)$ at $x^i$, for all $i$.  If $x^*$ is an optimal solution of \eqref{poi1}, it therefore holds from the subgradient inequality that
$$ f^* = f(x^*) \ge f(x^i) + (g^i)^T(x-x^i) \ . $$ Substituting this inequality in \eqref{toomuchsnow} for the value of $x = x^*$ yields:
\begin{align*}
\frac{\|x^0-x^*\|_2^2+ G^2\sum_{i = 0}^k\alpha_i^2}{2\sum_{i=0}^k\alpha_i}  &\ge \frac{1}{\sum_{i=0}^k\alpha_i}\sum_{i = 0}^k\alpha_i(g^i)^T(x^i - x^*) \\
&\ge \frac{1}{\sum_{i=0}^k\alpha_i}\sum_{i = 0}^k\alpha_i(f(x^i)- f^*) \ge \min_{i \in \{0, \ldots, k\}} f(x^i) \ .
\end{align*}
\qed\medskip

\subsubsection{Proof of Theorem \ref{December17}}\label{migraine}
We first prove part {\em (i)}.  Note that item {\em (i)} of Proposition~\ref{FSequiv} shows that $\FSe$ is a specific instance of subgradient descent to solve problem~\eqref{FS-problem}, using the constant step-size $\varepsilon$. Therefore
we can apply the computational guarantees associated with the subgradient descent method, particularly Proposition \ref{subgrad}, to the $\FSe$ algorithm.   Examining Proposition \ref{subgrad}, we need to work out the corresponding values of $f^*$, $\|x^0 - x^*\|_2$, $\alpha$, and $G$ in the context of $\FSe$ for solving the CM problem \eqref{FS-problem}.
Note that $f^\ast = 0$ for problem \eqref{FS-problem}.   We bound the distance from the initial residuals to the optimal least-squares residuals as follows:
$$\|\hat r^0 - r^*\|_2 = \|\hat r^0 - \hat r_{LS}\|_2 = \|\by - (\by - \bX\hat{\beta}_{\text{LS}})\|_2 = \|\bX\hat{\beta}_{\text{LS}}\|_2 \ . $$
From Proposition \ref{FSequiv} part {\em (i)} we have $\alpha = \varepsilon$.  Last of all, we need to determine an upper bound $G$ on the norms of subgradients.  We have:
$$\|g^k\|_2 = \|\sgn((\hat{r}^k)^T\bX_{j_k})\bX_{j_k}\|_2 = \|\bX_{j_k}\|_2 = 1 \ ,
$$ since the covariates have been standardized, so we can set $G=1$.  Now suppose algorithm $\FSe$ is run for $k$ iterations.  Proposition \ref{subgrad} then implies that:
\begin{equation}\label{fs_cor_bound}
\min\limits_{i \in \{0, \ldots, k \}}\|\bX^T\hat r^i\|_\infty = \min\limits_{i \in \{0, \ldots, k \}}f(\hat r^i) \leq f^* + \frac{\|\hat r^0 - r^*\|_2^2}{2\alpha(k+1)}+ \frac{\alpha G^2}{2} =  \frac{\|\bX\hat{\beta}_{\text{LS}}\|_2^2}{2\varepsilon(k+1)}+ \frac{\varepsilon}{2} \ .
\end{equation}
The above inequality provides a bound on the best (among the first $k$ residual iterates) empirical correlation between between the residuals $\hat r^i$ and each predictor variable, where the bound depends explicitly on the learning rate $\varepsilon$ and the number of iterations $k$.  Furthermore, invoking \eqref{grad_norm}, the above inequality implies the following upper bound on the norm of the gradient of the least squares loss $L_n(\cdot)$ for the model iterates $\{\hat{\beta}^i\}$ generated by $\FSe$:
\begin{equation}\label{fs_grad_bound}
\min\limits_{i \in \{0, \ldots, k \}}\|\nabla L_n(\hat{\beta}^i)\|_\infty \leq  \frac{\|\bX\hat{\beta}_{\text{LS}}\|_2^2}{2n\varepsilon(k+1)}+ \frac{\varepsilon}{2n} \ .
\end{equation}
Let $i$ be the index where the minimum is attained on the left side of the above inequality.  In a similar vein as in the analysis in Section \ref{LSBsection}, we now use Proposition \ref{qp} which presents two important properties of convex quadratic functions.  Because $L_n(\cdot)$ is a convex quadratic function of the same format as Proposition \ref{qp} with $h(\cdot) \gets L_n(\cdot)$, $Q \gets \tfrac{1}{n}\bX^T\bX$, and $h^* \gets L_n^*$, it follows from the second property of Proposition \ref{qp} that
$$\|\nabla L_n(\hat{\beta}^i)\|_2 \ge \sqrt{\frac{\lambda_\pmin (\tfrac{1}{n}\bX^T\bX) (L_n(\hat{\beta}^i) -L_n^*)}{2}} = \sqrt{\frac{\lambda_\pmin (\bX^T\bX) (L_n(\hat{\beta}^i) -L_n^*)}{2n}} \ ,  $$ where recall that $\lambda_{\pmin}(\bX^T\bX)$ denotes the smallest non-zero (hence positive) eigenvalue of $\bX^T\bX$.  Therefore
$$\|\nabla L_n(\hat{\beta}^i)\|_\infty^2 \ge \tfrac{1}{p}\|\nabla L_n(\hat{\beta}^i)\|_2^2 \ge \frac{\lambda_\pmin (\bX^T\bX) (L_n(\hat{\beta}^i) -L_n^*)}{2np} \ . $$
Substituting this inequality into \eqref{fs_grad_bound} for the index $i$ where the minimum is attained yields after rearranging:
\begin{equation}\label{bluray}
L_n(\hat{\beta}^{i}) - L_n^* \le  \frac{p}{2n\lambda_\pmin (\bX^T\bX)}\left[ \frac{\|\bX\hat{\beta}_{\text{LS}}\|_2^2}{\varepsilon(k+1)}+ \varepsilon \right]^2 \ , \end{equation}
which proves part {\em (i)}.  The proof of part {\em (ii)} follows by noting from the first inequality of Proposition \ref{qp} that there exists a least-squares solution $\hat{\beta}^*$ for which:
$$\|\hat{\beta}^* - \hat{\beta}^i\|_2 \le \sqrt{\frac{2(L_n(\hat{\beta}^i) - L_n^*)}{\lambda_{\pmin}\left(\tfrac{1}{n}\bX^T\bX\right)}}=\sqrt{\frac{2n(L_n(\hat{\beta}^i) - L_n^*)}{\lambda_{\pmin}(\bX^T\bX)}}\le \frac{\sqrt{p}}{\lambda_{\pmin}(\bX^T\bX)}\left[\frac{\|\bX\hat{\beta}_{\text{LS}}\|_2^2}{\varepsilon(k+1)}+ \varepsilon \right] \ ,$$ where the second inequality in the above chain follows using \eqref{bluray}.  The proof of part {\em (iii)} follows by first observing that
$\|\bX(\hat{\beta}^i- \hat{\beta}_{\text{LS}} )\|_2 = \sqrt{2n(L_n(\hat{\beta}^i) - L_n^*)}$ and then substituting the bound on $(L_n(\hat{\beta}^i) - L_n^*)$ from part {\em (i)} and simplifying terms.  Part {\em (iv)} is a restatement of inequality \eqref{fs_cor_bound}.  Finally, parts {\em (v)} and {\em (vi)} are simple and well-known structural properties of $\FSe$ that are re-stated here for completeness. \qed\medskip

\subsubsection{A deeper investigation of the computational guarantees for \ebs and $\FSe$}\label{retread}

Here we show that in theory, \ebs is much more efficient than $\FSe$ if the primary goal is to obtain a model with a certain (pre-specified) data-fidelity. To formalize this notion,
we consider a parameter
$\tau \in (0,1]$.  We say that $\bar\beta$ is at a $\tau$-relative distance to the least squares predictions if $\bar\beta$ satisfies:
\begin{equation}\label{predict_gamma}
\|\bX\bar\beta - \bX \hat{\beta}_{\text{LS}}\|_2 \leq  \tau\| \bX \hat{\beta}_{\text{LS}}\|_2 \ .
\end{equation}

Now let us pose the following question:  if both \ebs and $\FSe$ are allowed to run with an appropriately chosen learning rate $\varepsilon$ for each algorithm,
which algorithm will satisfy~\eqref{predict_gamma} in fewer iterations? We will answer this question by studying closely the computational guarantees of Theorems~\ref{august24-2} and~\ref{December17}.
Since our primary goal is to compute $\bar\beta$ satisfying \eqref{predict_gamma}, we may optimize the learning rate $\varepsilon$, for each algorithm, to achieve this goal with the smallest
number of boosting iterations.

Let us first study \ebs.  As we have seen, a learning rate of $\varepsilon = 1$ achieves the fastest rate of linear convergence for \ebs and is thus optimal with regard to the bound in part {\em (iii)} of
Theorem \ref{august24-2}. If we run \ebs with $\varepsilon = 1$ for
$k^{\text{\ebs}} := \left\lceil\frac{4p}{\lambda_{\pmin}(\bX^T\bX)}\ln\left(\frac{1}{\tau^2}\right)\right\rceil$ iterations,
then it follows from part {\em (iii)} of Theorem \ref{august24-2} that we achieve \eqref{predict_gamma}.  Furthermore, it follows from \eqref{2miles} that the resulting $\ell_{1}$-shrinkage bound will satisfy:

$$\textsc{Sbound}^{\text{\ebs}} \le \|\bX\hat\beta_{\text{LS}}\|_2 \sqrt{k^{\text{\ebs}}}  \ . $$

For $\FSe$, if one works out the arithmetic, the optimal number of boosting iterations to achieve \eqref{predict_gamma} is given by:
$k^{\FSe} := \left\lceil\frac{4p}{\lambda_{\pmin}(\bX^T\bX)}\left(\frac{1}{\tau^2}\right)\right\rceil - 1$ using the learning rate $\varepsilon = \frac{\|\bX\hat{\beta}_{\text{LS}}\|_2}{\sqrt{k^{\FSe}+1}}$.  Also, it follows from part {\em (v)} of Theorem \ref{December17} that the resulting shrinkage bound will satisfy:

$$\textsc{Sbound}^{\FSe} \le \varepsilon \cdot  k^{\FSe}   \approx \|\bX\hat\beta_{\text{LS}}\|_2 \cdot  \sqrt{k^{\FSe}} \ . $$

Observe that $k^{\text{\ebs}} < k^{\FSe}$, whereby \ebs is able to achieve \eqref{predict_gamma} in fewer iterations than $\FSe$.  Indeed, if we let $\eta$ denote the ratio $k^{\text{\ebs}} / k^{\FSe}$, then it holds that

\begin{equation}\label{sunny} \eta := \frac{k^{\text{\ebs}}}{k^{\FSe}} \approx \frac{\ln\left(\frac{1}{\tau^2}\right)}{\frac{1}{\tau^2}} \le \frac{1}{e} \ < \ 0.368 \ . \end{equation}

The left panel of Figure \ref{fig:eta_vs_tau} shows the value of $\eta$ as a function of $\tau$.  For small values of the tolerance parameter $\tau$ we see that $\eta$ is itself close to zero, which means that \ebs will need significantly fewer iterations than $\FSe$ to achieve the condition~\eqref{predict_gamma}.

We can also examine the $\ell_{1}$-shrinkage bounds similarly.  If we let $\vartheta$ denote the ratio of

$\textsc{Sbound}^{\text{\ebs}}$ to $\textsc{Sbound}^{\FSe}$, then it holds that

\begin{equation}\label{sunny2} \vartheta := \frac{\textsc{Sbound}^{\text{\ebs}}}{\textsc{Sbound}^{\FSe}} = \sqrt{\frac{k^{\text{\ebs}}}{k^{\FSe}}} = \frac{\sqrt{\ln\left(\frac{1}{\tau^2}\right)}}{\frac{1}{\tau}} \le \frac{1}{\sqrt{e}} \ < \ 0.607 \ . \end{equation}

This means that if both bounds are relatively tight, then the $\ell_1$-shrinkage of the final model produced by \ebs is smaller than that of the final model produced by $\FSe$, by at least a factor of $0.607$.  The right panel of Figure \ref{fig:eta_vs_tau} shows the value of $\vartheta$ as a function of $\tau$.  For small values of the relative predication error constant $\tau$ we see that $\vartheta$ is itself close to zero.

\begin{figure}[h!]
\begin{center}
\scalebox{.9}[0.8]{\begin{tabular}{ccccc}
\rotatebox{90}{\sf { {~~~~~~~~~~~~~~~~~~~~~$\eta$}}}&\includegraphics[width=0.4\textwidth,height=0.24\textheight,  trim = 1.1cm 1.5cm 0cm 2cm, clip = true ]{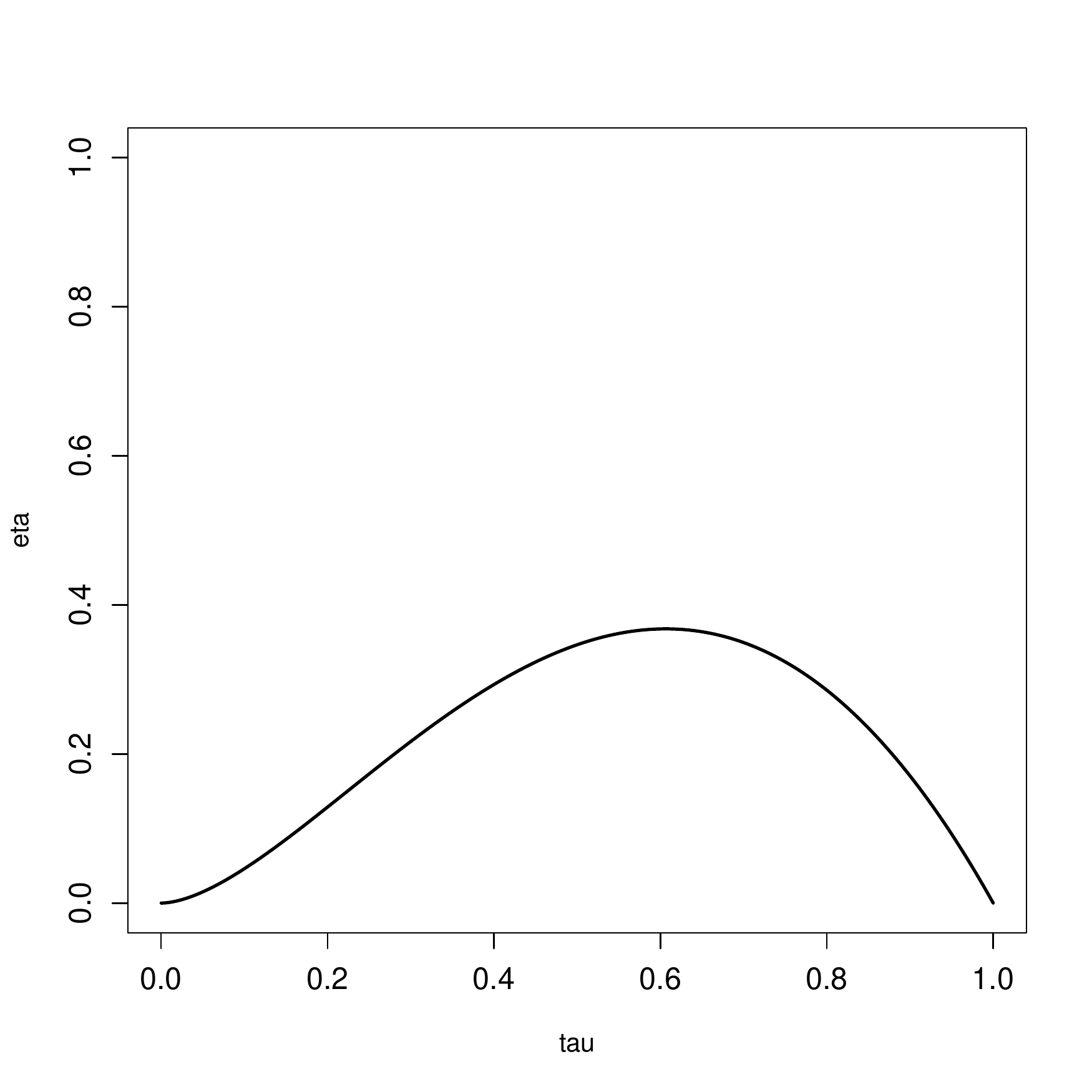}&&
\rotatebox{90}{\sf { {~~~~~~~~~~~~~~~~~~~~~$\vartheta$}}}&\includegraphics[width=0.4\textwidth,height=0.24\textheight,   trim = 1.1cm 1.5cm 0cm 2cm, clip = true ]{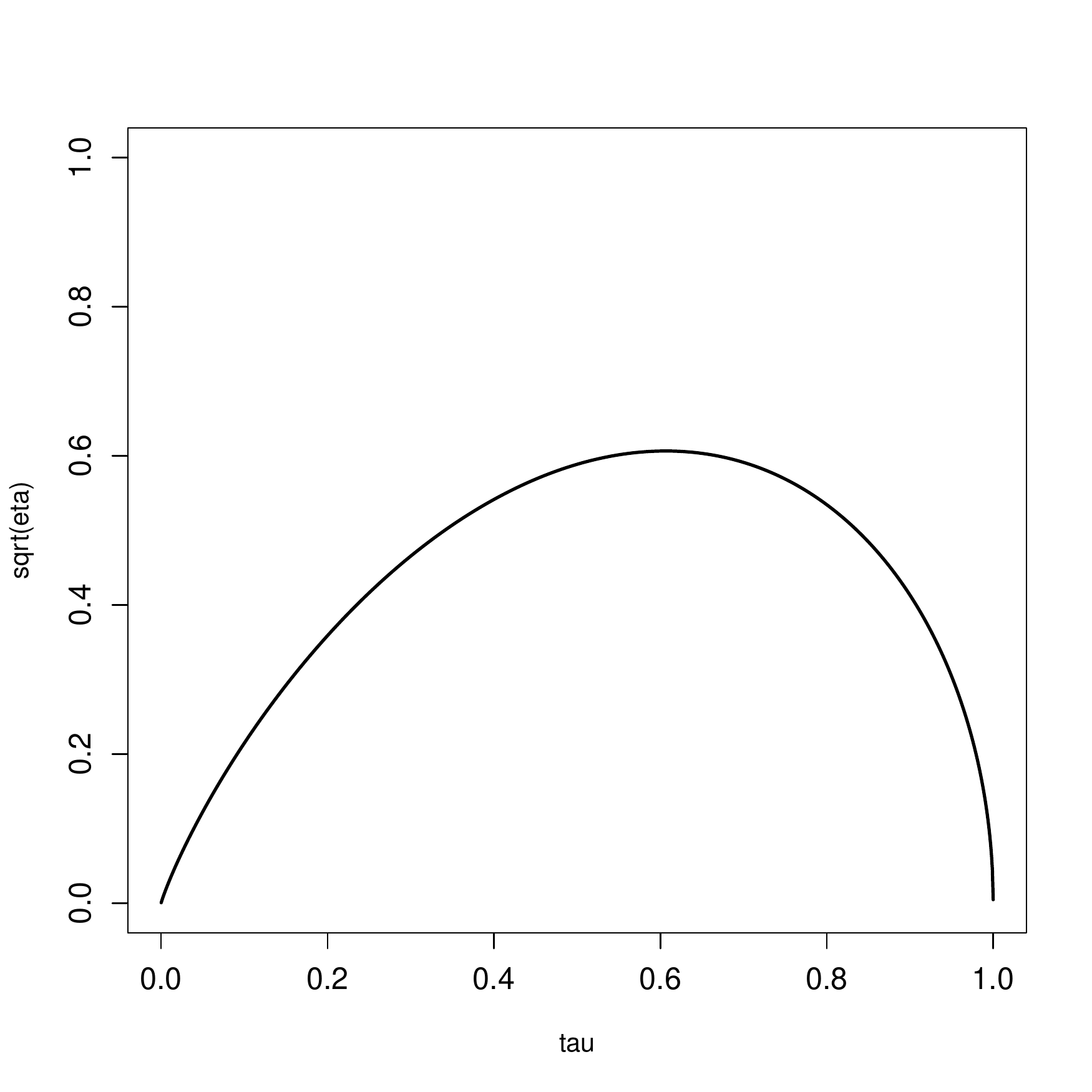}\\
&\sf $\tau$& &&\sf $\tau$ \\
\end{tabular}}
\end{center}
\caption{{\small{Plot showing the value of the ratio $\eta$ of iterations of \ebs to $\FSe$ (equation \eqref{sunny}) versus the target relative prediction error $\tau$ [left panel], and the ratio $\vartheta$ of shrinkage bounds of \ebs to $\FSe$ (equation \eqref{sunny2}) versus the target relative prediction error $\tau$ [right panel]. }} }\label{fig:eta_vs_tau}
\end{figure}

We summarize the above analysis in the following remark.\medskip

\begin{remark}\label{our remark} {\bf (Comparison of efficiency of \ebs and $\FSe$)}  Suppose that the primary goal is to achieve a $\tau$-relative prediction error as defined in \eqref{predict_gamma}, and that \ebs and $\FSe$ are run with appropriately determined learning rates for each algorithm.  Then the ratio of required number of iterations of these methods to achieve \eqref{predict_gamma} satisfies $$\eta := \frac{k^{\text{\ebs}}}{k^{\FSe}} < 0.368 \ . $$Also, the ratio of the shrinkage bounds from running these numbers of iterations satisfies $$\vartheta:= \frac{\textsc{Sbound}^{\text{\ebs}}}{\textsc{Sbound}^{\FSe}} < 0.607 \ , $$where all of the analysis is according to the bounds in Theorems \ref{December17}  and \ref{august24-2}.\end{remark}

We caution the reader that the analysis leading to Remark \ref{our remark} is premised on the singular goal of achieving \eqref{predict_gamma} in as few iterations as possible. As mentioned previously, the models produced in the interior of the boosting profile are more statistically interesting than those produced at the end. Thus for both algorithms it may be beneficial, and may lessen the risk of overfitting, to trace out a smoother profile by selecting the learning rate $\varepsilon$ to be smaller than the prescribed values in this subsection ($\varepsilon = 1$ for \ebs and $\varepsilon = \frac{\|\bX\hat{\beta}_{\text{LS}}\|_2}{\sqrt{k^{\FSe}+1}}$ for $\FSe$). Indeed, considering just \ebs for simplicity, if our goal is to produce a $\tau$-relative prediction error according to \eqref{predict_gamma} with the smallest possible $\ell_1$ shrinkage, then Figure \ref{fig:ls} suggests that this should be accomplished by selecting $\varepsilon$ as small as possible (essentially very slightly larger than 0).

\subsection{Additional Details for Section~4}

\subsubsection{Duality Between Regularized Correlation Minimization and the \lassoperiod}\label{lasso_dual_appendix}
In this section, we precisely state the duality relationship between the RCM problem \eqref{FS-dual-prob} and the \lassoperiod. We first prove the following property of the least squares loss function that will be useful in our analysis.\medskip

\begin{proposition}\label{lasso-rep}
The least squares loss function $L_n(\cdot)$ has the following max representation:
\begin{equation}\label{ls-minmax}
L_n(\beta) = \max\limits_{\tilde r \in {P_{\text{res}}}}\left\{-\tilde r^T(\tfrac{1}{n}\bX)\beta - \tfrac{1}{2n}\|\tilde r - \by\|_2^2 + \tfrac{1}{2n}\|\by\|_2^2\right\} \ ,
\end{equation}where $P_{\mathrm{res}}:= \{r \in \mathbb{R}^n : r=\by-\bX\beta \ \mathrm{for~some~} \beta \in \mathbb{R}^p\}$.
Moreover, the unique optimal solution (as a function of $\beta$) to the subproblem in \eqref{ls-minmax} is $\bar{r} := \by - \bX\beta$.
\end{proposition}
\begin{proof}
For any $\beta \in \mathbb{R}^p$, it is easy to verify through optimality conditions (setting the gradient with respect to $\tilde r$ equal to 0) that $\bar{r}$ solves the subproblem in \eqref{ls-minmax}, i.e., $$\bar{r} = \arg\max\limits_{\tilde r \in {P_{\text{res}}}}\left\{-\tilde r^T(\tfrac{1}{n}\bX)\beta - \tfrac{1}{2n}\|\tilde r - \by\|_2^2 + \tfrac{1}{2n}\|\by\|_2^2\right\} \ . $$ Thus, we have
\begin{align*}
\max\limits_{\tilde r \in {P_{\text{res}}}}\left\{-\tilde r^T(\tfrac{1}{n}\bX)\beta - \tfrac{1}{2n}\|\tilde r - \by\|_2^2 + \tfrac{1}{2n}\|\by\|_2^2\right\} &= \tfrac{1}{n}\left(\tfrac{1}{2}\|\by\|_2^2 - \by^T\bX\beta + \tfrac{1}{2}\|\bX\beta\|_2^2\right)\\
&= \tfrac{1}{2n}\|\by - \bX\beta\|_2^2 \ .
\end{align*}
\end{proof}

The following result demonstrates that RCM \eqref{FS-dual-prob} has a direct interpretation as a (scaled) dual of the \lasso problem \eqref{poi-lasso}. Moreover, in part {\em (iii)} of the below Proposition, we give a bound on the optimality gap for the \lasso problem in terms of a quantity that is closely related to the objective function of RCM.\medskip

\begin{proposition}{\bf (Duality Equivalence of \lasso and $\text{RCM}_\delta$, and Optimality Bounds)}\label{LASSO-dual}
The \lasso problem \eqref{poi-lasso} and the regularized correlation minimization problem $\mathrm{RCM}_\delta$ \eqref{FS-dual-prob} are dual optimization problems modulo the scaling factor $\frac{n}{\delta}$.  In particular:
\begin{itemize}
\item[(i)]  (Weak Duality) If $\beta$ is feasible for the \lasso problem \eqref{poi-lasso}, and if $\tilde r$ is feasible for the regularized correlation minimization problem $\mathrm{RCM}_\delta$ \eqref{FS-dual-prob}, then
$$L_n(\beta) + \tfrac{\delta}{n} f_\delta(\tilde r) \ \ge \ \tfrac{1}{2n}\|\by\|_2^2  \ . $$
\item[(ii)]  (Strong Duality) It holds that: $$L_{n, \delta}^\ast + \tfrac{\delta}{n}f_\delta^* = \tfrac{1}{2n}\|\by\|_2^2 \ . $$ Moreover, for any given parameter value $\delta \ge 0$, there is a unique vector of residuals $\hat r^\ast_\delta $ associated with every \lasso solution $\hat{\beta}^*_\delta$, i.e., $\hat r^*_\delta = \by - \bX\hat{\beta}^*_\delta$, and $\hat r^\ast_\delta$ is the unique optimal solution to the $\text{RCM}_\delta$ problem \eqref{FS-dual-prob}.
\item[(iii)]  (Optimality Condition for \lassoperiod) If $\beta$ is feasible for the \lasso problem \eqref{poi-lasso} and $r=\by - \bX\beta$, then
\begin{equation}\label{wolfe_gap}
\omega_\delta(\beta) := \|\bX^Tr\|_\infty - \frac{r^T\bX\beta}{\delta} \ \ge 0 \ ,
\end{equation} and
$$L_n(\beta) - L_{n, \delta}^\ast ~\leq~ \tfrac{\delta}{n}\cdot\omega_\delta(\beta) \ . $$
Hence, if $\omega_\delta(\beta) = 0$, then $\beta$ is an optimal solution of the \lasso problem \eqref{poi-lasso}.  \qed
\end{itemize}
\end{proposition}
\begin{proof}
Let us first construct the problem $\mathrm{RCM}_\delta$ using basic constructs of minmax duality.  As demonstrated in Proposition \ref{lasso-rep}, the least-squares loss function $L_n(\cdot)$ has the following max representation:
\begin{equation*}
L_n(\beta) = \max\limits_{\tilde r \in {P_{\text{res}}}}\left\{-\tilde r^T(\tfrac{1}{n}\bX)\beta - \tfrac{1}{2n}\|\tilde r - \by\|_2^2 + \tfrac{1}{2n}\|\by\|_2^2\right\} \ .
\end{equation*}
Therefore the \lasso problem \eqref{poi-lasso} can be written as
\begin{equation*}
\begin{array}{rcl}
\min\limits_{\beta \in B_\delta}   & \max\limits_{\tilde r \in {P_{\text{res}}}}\left\{-\tilde r^T(\tfrac{1}{n}\bX)\beta - \tfrac{1}{2n}\|\tilde r - \by\|_2^2 + \tfrac{1}{2n}\|\by\|_2^2\right\}\\
\end{array}
\end{equation*} where $B_\delta := \{\beta \in \mathbb{R}^p : \|\beta\|_1 \le \delta \}$.
We construct a dual of the above problem by interchanging the min and max operators above, yielding the following dual optimization problem:
\begin{equation*}
\max_{\tilde r \in P_{\mathrm{res}}} \ \ \min_{\beta \in B_\delta}\left\{-\tilde r^T(\tfrac{1}{n}\bX)\beta - \tfrac{1}{2n}\|\tilde r - \by\|_2^2 + \tfrac{1}{2n}\|\by\|_2^2\right\} \ .
\end{equation*}
After negating, and dropping the constant term $\tfrac{1}{2n}\|\by\|_2^2$, the above dual problem is equivalent to:
\begin{equation}\label{ben_folds}
\min_{\tilde r \in P_{\mathrm{res}}} \ \ \max_{\beta \in B_\delta}\left\{\tilde r^T(\tfrac{1}{n}\bX)\beta\right\} + \tfrac{1}{2n}\|\tilde r - \by\|_2^2 \ .
\end{equation}
Now notice that
\begin{equation}\label{matt_kim}
 \max_{\beta \in B_\delta}\left\{\tilde r^T(\tfrac{1}{n}\bX)\beta\right\} = \tfrac{\delta}{n}\left( \max_{j \in \{1, \ldots, p\}}|\tilde r^T\bX_j| \right)=\tfrac{\delta}{n}\|\bX^T\tilde r\|_\infty \ ,
\end{equation} from which it follows after scaling by $\frac{n}{\delta}$ that \eqref{ben_folds} is equivalent to \eqref{FS-dual-prob}.

Let us now prove item (i).  Let $\beta$ be feasible for the \lasso problem \eqref{poi-lasso} and $\tilde r$ be feasible for the regularized correlation minimization problem $\mathrm{RCM}_\delta$ \eqref{FS-dual-prob}, and let $r = \by - \bX\beta$ and let $\tilde \beta$ be such that $\tilde r = \by - \bX\tilde \beta$.  Then direct arithmetic manipulation yields the following equality:
\begin{equation}\label{weak_arithmetic}
L_n(\beta) + \tfrac{\delta}{n} f_\delta(\tilde r) = \tfrac{1}{2n}\|\by\|_2^2 + \tfrac{1}{2n}\|r-\tilde r\|_2^2 + \tfrac{\delta}{n}\left(\|\bX^T\tilde r\|_\infty - \frac{\tilde r^T\bX\beta}{\delta} \right)  \ ,
\end{equation}
from which the result follows since $\|r-\tilde r\|_2^2 \ge 0$ and $\tilde r^T\bX\beta \le \|\bX^T\tilde r\|_\infty \|\beta\|_1 \le \delta \|\bX^T\tilde r\|_\infty$ which implies that the last term above is also nonnegative.

To prove item (ii), notice that both the \lasso and $\mathrm{RCM}_\delta$ can be re-cast as optimization problems with a convex quadratic objective function and with linear inequality constraints.  That being the case, the classical strong duality results for linearly-constrained convex quadratic optimization apply, see \cite{avriel} for example.

We now prove (iii). Since $\beta$ is feasible for the \lasso problem, it follows from the Holder inequality that $r^T\bX\beta \le \|\bX^Tr\|_\infty \|\beta\|_1 \le \delta \|\bX^Tr\|_\infty$, from which it then follows that $\omega_\delta(\beta) \ge 0$. Invoking \eqref{weak_arithmetic} with $\tilde r \gets r = \by - \bX\beta$ yields:
\begin{equation*}
L_n(\beta) + \tfrac{\delta}{n} f_\delta(r) = \tfrac{1}{2n}\|\by\|_2^2 + \tfrac{\delta}{n}\cdot\omega_\delta(\beta) \ .
\end{equation*}
Combining the above with strong duality (ii) yields:
\begin{equation*}
L_n(\beta) + \tfrac{\delta}{n} f_\delta(r) = L_{n, \delta}^\ast + \tfrac{\delta}{n}f_\delta^\ast + \tfrac{\delta}{n}\cdot\omega_\delta(\beta) \ .
\end{equation*}
After rearranging we have:
\begin{equation*}
L_n(\beta) - L_{n, \delta}^\ast \leq \tfrac{\delta}{n}f_\delta^\ast - \tfrac{\delta}{n} f_\delta(r) + \tfrac{\delta}{n}\cdot\omega_\delta(\beta) \leq \tfrac{\delta}{n}\cdot\omega_\delta(\beta) \ ,
\end{equation*}
where the last inequality follows since $f_\delta^\ast \leq f_\delta(r)$.
\end{proof}

\subsubsection{Proof of Proposition \ref{RFSequiv}}\label{may3}
Recall the update formula for the residuals in $\RFSe$:
\begin{equation}\label{Rresupdate}
\hat r^{k+1} \gets \hat{r}^k - \varepsilon\left[\sgn((\hat{r}^k)^T\bX_{j_k})\bX_{j_k} + \tfrac{1}{\delta}(\hat{r}^k - \by)\right] \ .
\end{equation}
We first show that $g^k := \sgn((\hat{r}^k)^T\bX_{j_k})\bX_{j_k} + \tfrac{1}{\delta}(\hat{r}^k - \by)$ is a subgradient of $f_\delta(\cdot)$ at $\hat{r}^k$. Recalling the proof of Proposition \ref{FSequiv}, we have that $\sgn((\hat{r}^k)^T\bX_{j_k})\bX_{j_k}$ is a subgradient of $f(r) := \|\bX^Tr\|_\infty$ at $\hat r^k$ since $j_k \in \argmax_{j \in \{1, \ldots, p\}}|(\hat{r}^k)^T\bX_j|$. Therefore, since $f_\delta(r) = f(r) + \tfrac{1}{2\delta}\|r-\by\|_2^2$, it follows from the additive property of subgradients (and gradients) that $g^k = \sgn((\hat{r}^k)^T\bX_{j_k})\bX_{j_k} + \tfrac{1}{\delta}(\hat{r}^k - \by)$ is a subgradient of $f_\delta(r)$ at $r=\hat{r}^k$. Therefore the update \eqref{Rresupdate} is of the form $\hat{r}^{k+1} = \hat{r}^k - \varepsilon g^k$ where $g^k \in \partial f_\delta(\hat{r}^k)$. Finally note that $\hat{r}^k - \varepsilon g^k = \hat{r}^{k+1} = \by - \bX\beta^{k+1} \in P_{\mathrm{res}}$, hence $\Pi_{P_{\mathrm{res}}}(\hat{r}^k - \varepsilon g^k) = \hat{r}^k - \varepsilon g^k$, i.e., the projection step is superfluous here. Therefore $\hat{r}^{k+1} = \Pi_{P_{\mathrm{res}}}(\hat{r}^k - \varepsilon g^k)$, which shows that \eqref{Rresupdate} is precisely the update for the subgradient descent method with step-size $\alpha_k := \varepsilon$.
\qed

\subsubsection{Proof of Theorem \ref{RFSe-guarantees}}\label{migraine2}
Let us first use induction to demonstrate that the following inequality holds:
\begin{equation}\label{geo_shrinkage}
\| \hat{\beta}^k \|_1 \le \varepsilon\sum_{j = 0}^{k-1}\left(1 - \tfrac{\varepsilon}{\delta}\right)^j \ \ \ \mathrm{for~all~} k \ge 0 \ .
\end{equation}
Clearly, \eqref{geo_shrinkage} holds for $k = 0$ since $\hat{\beta}^0 = 0$. Assuming that \eqref{geo_shrinkage} holds for $k$, then the update for $\hat{\beta}^{k+1}$ in step (3.) of algorithm $\RFSe$ can be written as $\hat{\beta}^{k+1} = (1-\tfrac{\varepsilon}{\delta})\hat{\beta}^k + \varepsilon \cdot \sgn((\hat{r}^k)^T\bX_{j_k})e_{j_k}$, from which it holds that
\begin{align*}
\|\hat{\beta}^{k+1} \|_1 &=  \|(1-\tfrac{\varepsilon}{\delta})\hat{\beta}^k + \varepsilon \cdot \sgn((\hat{r}^k)^T\bX_{j_k})e_{j_k} \|_1 \\ \\
&\leq (1-\tfrac{\varepsilon}{\delta})\|\hat{\beta}^k\|_1  + \varepsilon\| e_{j_k} \|_1 \\ \\
&\le (1-\tfrac{\varepsilon}{\delta})\varepsilon\sum_{j = 0}^{k-1}\left(1 - \tfrac{\varepsilon}{\delta}\right)^j \ + \ \varepsilon \\
&= \varepsilon\sum_{j = 0}^{k}\left(1 - \tfrac{\varepsilon}{\delta}\right)^j \ ,
\end{align*}
which completes the induction. Now note that \eqref{geo_shrinkage} is a geometric series and we have:
\begin{equation}\label{geo_shrinkage2}
\| \hat{\beta}^k \|_1 \le~ \varepsilon\sum_{j = 0}^{k-1}\left(1 - \tfrac{\varepsilon}{\delta}\right)^j =~ \delta\left[1 - \left(1 - \tfrac{\varepsilon}{\delta}\right)^k\right] \leq~ \delta \ \ \ \mathrm{for~all~} k \ge 0 \ .
\end{equation}

Recall that we developed the algorithm $\RFSe$ in such a way that it corresponds exactly to an instantiation of the subgradient descent method applied to the RCM problem \eqref{FS-dual-prob}.  Indeed, the update rule for the residuals given in Step (3.) of $\RFSe$ is:  $\hat r^{k+1} \gets \hat{r}^k - \varepsilon g^k$ where $g^k = \left[\sgn((\hat{r}^k)^T\bX_{j_k})\bX_{j_k} + \tfrac{1}{\delta}(\hat{r}^k - \by)\right]$.  We therefore can apply Proposition \ref{simpleprop}, and more specifically the inequality \eqref{waytoomuchsnow}.  In order to do so we need to translate the terms of Proposition \ref{simpleprop} to our setting:  here the variables $x$ are now the residuals $r$, the iterates $x^i$ are now the iterates $\hat r^i$, etc.  The step-sizes of algorithm $\RFSe$ are fixed at $\varepsilon$, so we have $\alpha_i = \varepsilon$ for all $i \ge 0$.  Setting the value of $x$ in  Proposition \ref{simpleprop} to be least-squares residual value, namely $x = \hat r_{LS}$, the left side of \eqref{waytoomuchsnow} is therefore:
\begin{equation}\label{39steps}\begin{array}{rcl}
\frac{1}{k+1}\sum_{i = 0}^k(g^i)^T(x^i - x) &= & \frac{1}{k+1}\sum_{i = 0}^k \left(\bX\left[\sgn((\hat r^i)^T\bX_{j_i})e_{j_i} - \tfrac{1}{\delta}\hat{\beta}^i\right]\right)^T(\hat r^i - \hat r_{LS}) \\ \\
&=&\frac{1}{k+1}\sum_{i = 0}^k\left(\sgn((\hat r^i)^T\bX_{j_i})\bX_{j_i} - \tfrac{1}{\delta}(\bX\hat{\beta}^i)\right)^T\hat r^i \\ \\
&=&\frac{1}{k+1}\sum_{i = 0}^k\left[ \|\bX^T\hat r^i\|_\infty - \tfrac{1}{\delta}(\hat r^i)^T\bX\hat{\beta}^i\right] \\ \\
&=& \frac{1}{k+1}\sum_{i = 0}^k \omega_\delta(\hat{\beta}^i) \ ,
\end{array}\end{equation} where the second equality uses the fact that $\bX^T\hat r_{LS} = 0$ from \eqref{grad2} and the fourth equality uses the definition of $\omega_\delta(\beta) $ from \eqref{wolfe_gap}.

Let us now evaluate the right side of \eqref{waytoomuchsnow}.  We have $\|x^0 - x\|_2 = \|\hat r^0 - \hat r_{LS}\|_2=  \|\by - (\by - \bX\hat{\beta}_{\text{LS}})\|_2 = \|\bX\hat{\beta}_{\text{LS}}\|_2 $.  Also, it holds that
$$\|g^i\|_2 = \|\sgn((\hat r^i)^T\bX_{j_i})\bX_{j_i} - \tfrac{1}{\delta}(\bX\hat{\beta}^i)\|_2 \leq \|\bX_{j_i}\|_2 + \|\bX(\tfrac{\hat{\beta}^i}{\delta})\|_2 \leq 1 + \tfrac{1}{\delta}\|\bX\|_{1,2}\|\hat{\beta}^i\|_1 \le 1 + \|\bX\|_{1,2} \le 2 \ , $$where the third inequality follows since $\|\hat{\beta}^i\|_1 \leq \delta$ from \eqref{geo_shrinkage2} and the second and fourth inequalities follow from the assumption that the columns of $\bX$ have been normalized to have unit $\ell_2$ norm.  Therefore $G=2$ is a uniform bound on $\|g^i\|_2$.  Combining the above, inequality \eqref{waytoomuchsnow} implies that after running $\RFSe$ for $k$ iterations, it holds that:
\begin{equation}\label{lasso_gap_bound}
\min_{i \in \{0, \ldots, k\}}\omega_\delta(\hat{\beta}^i) ~\leq~ \frac{1}{k+1}\sum_{i = 0}^k \omega_\delta(\hat{\beta}^i) ~\leq~ \frac{\|\bX\hat{\beta}_{\text{LS}}\|_2^2}{2(k+1)\varepsilon} + \frac{2^2\varepsilon}{2}   ~=~ \frac{\|\bX\hat{\beta}_{\text{LS}}\|_2^2}{2\varepsilon(k+1)} + 2\varepsilon \ ,
\end{equation}where the first inequality is elementary arithmetic and the second inequality is the application of \eqref{waytoomuchsnow}.  Now let $i$ be the index obtaining the minimum in the left-most side of the above.  Then it follows from part (iii) of Proposition \ref{LASSO-dual} that
\begin{equation}\label{slopes} L_n(\hat{\beta}^i) - L_{n, \delta}^\ast ~\leq~ \tfrac{\delta}{n}\cdot\omega_\delta(\hat{\beta}^i) \le \frac{\delta\|\bX\hat{\beta}_{\text{LS}}\|_2^2}{2n\varepsilon(k+1)} + \frac{2\delta\varepsilon}{n} \ ,  \end{equation}which proves item (i) of the theorem.

To prove item (ii), note first that if $\hat{\beta}^\ast_\delta$ is a solution of the \lasso problem \eqref{poi-lasso}, then it holds that $\|\hat{\beta}^\ast_\delta\|_1 \le \delta$ (feasibility) and $\omega_\delta(\hat{\beta}^\ast_\delta)=0$ (optimality).  This latter condition follows easily from the optimality conditions of linearly constrained convex quadratic problems, see \cite{avriel} for example.  Setting $\hat r^\ast_\delta = \by - \bX\hat{\beta}^\ast_\delta$, the following holds true:
$$\begin{array}{rcl}
\|\bX\hat{\beta}^i - \bX\hat{\beta}^\ast_\delta\|_2^2 &=& 2n\left( L_n(\hat{\beta}^i) - L_n(\hat{\beta}^\ast_\delta) + (\hat r^\ast_\delta)^T\bX(\hat{\beta}^i - \hat{\beta}^\ast_\delta)\right) \\ \\
&=& 2n\left( L_n(\hat{\beta}^i) - L_{n,\delta}^* -\delta \|\bX^T \hat r^\ast_\delta\|_\infty + (\hat r^\ast_\delta)^T\bX\hat{\beta}^i \right) \\ \\
&\le& 2n\left( L_n(\hat{\beta}^i) - L_{n,\delta}^* -\delta \|\bX^T \hat r^\ast_\delta\|_\infty + \|\bX^T \hat r^\ast_\delta\|_\infty \|\hat{\beta}^i\|_1 \right)  \\ \\
&\le& 2n\left( L_n(\hat{\beta}^i) - L_{n,\delta}^* -\delta \|\bX^T \hat r^\ast_\delta\|_\infty + \delta \|\bX^T \hat r^\ast_\delta\|_\infty  \right)  \\ \\
&=& 2n\left( L_n(\hat{\beta}^i) - L_{n,\delta}^* \right) \\ \\
&\le& \frac{\delta\|\bX\hat{\beta}_{\text{LS}}\|_2^2}{\varepsilon(k+1)} + 4\delta\varepsilon \ ,
\end{array}$$where the first equality is from direct arithmetic substitution, the second equality uses the fact that $\omega_\delta(\hat{\beta}^\ast_\delta)=0$ whereby $(\hat r^\ast_\delta)^T\bX\hat{\beta}^\ast_\delta =  \delta \|\bX^T \hat r^\ast_\delta\|_\infty$, the first inequality follows by applying Holder's inequality to the last term of the second equality, and the final inequality is an application of \eqref{slopes}.  Item (ii) then follows by taking square roots of the above.

Item (iii) is essentially just \eqref{geo_shrinkage2}. Indeed, since $i \leq k$ we have:
\begin{equation*}
\| \hat{\beta}^i \|_1 \le~ \varepsilon\sum_{j = 0}^{i-1}\left(1 - \tfrac{\varepsilon}{\delta}\right)^j \leq~ \varepsilon\sum_{j = 0}^{k-1}\left(1 - \tfrac{\varepsilon}{\delta}\right)^j =~ \delta\left[1 - \left(1 - \tfrac{\varepsilon}{\delta}\right)^k\right] \leq~ \delta \ .
\end{equation*}
(Note that we emphasize the dependence on $k$ rather than $i$ in the above since we have direct control over the number of boosting iterations $k$.) Item (iv) of the theorem is just a restatement of the sparsity property of $\RFSe$.\qed

\subsubsection{Regularized Boosting: Related Work and Context} \label{related-work-lasso-boost}

As we have already seen, the $\FSe$ algorithm leads to models that have curious similarities with the \lasso coefficient profile, but in general the profiles are different.
Sufficient conditions under which the coefficient profiles of $\FSe$ (for $\varepsilon \approx 0$) and \lasso are equivalent have been explored in~\cite{hastie06:_forwar}.
A related research question is whether there are
structurally similar algorithmic variants of $\FSe$ that lead to \lasso solutions for arbitrary datasets? In this vein~\cite{zhao2007stagewise} propose \blassoperiod, a corrective version of the forward stagewise algorithm. \blassoperiod, in addition to taking incremental forward steps (as in $\FSe$), also takes backward steps, the result of which is that the algorithm approximates the \lasso coefficient profile under certain assumptions on the data.
 The authors observe that \blasso often leads to models that are sparser and have better predictive accuracy than those produced by $\FSe$.

In \cite{buhlmann2006sparse}, the authors point out that models delivered by boosting methods need not be adequately sparse, and they highlight the importance of obtaining models that have more sparsity, better prediction accuracy, and better variable selection properties.
They propose a sparse variant of $L2$-\textsc{Boost} (see also Section~\ref{sect_intro}) which considers a regularized version of the
squared error loss, penalizing the approximate degrees of freedom of the model.

In \cite{friedman2003importance}, the authors also point out that boosting algorithms often lead to a large collection of nonzero coefficients. They suggest reducing the complexity of the model by some form of
``post-processing" technique---one such proposal is to apply a \lasso regularization on the selected set of coefficients.

A parallel line of work in machine learning~\cite{duchi2009boosting} explores the scope of boosting-like algorithms on $\ell_{1}$-regularized versions of different loss functions arising mainly in the context of classification problems.
The proposal of~\cite{duchi2009boosting}, when adapted to the least squares regression problem with $\ell_{1}$-regularization penalty, leads to the following optimization problem:
\begin{equation}\label{lasso-lag-1}
\min_\beta\;\;  \tfrac{1}{2n} \| \by - \bX \beta \|_{2}^2 + \lambda \| \beta \|_{1} \ ,
\end{equation}
for which the authors~\cite{duchi2009boosting} employ greedy coordinate descent methods. Like the boosting algorithms considered herein, at each iteration the algorithm studied by \cite{duchi2009boosting} selects a certain coefficient $\beta_{j_k}$ to update, leaving all other coefficients $\beta_i$ unchanged. The amount with which to update the coefficient $\beta_{j_k}$ is determined by fully optimizing the loss function \eqref{lasso-lag-1} with respect to $\beta_{j_k}$, again holding all other coefficients constant (note that one recovers \textsc{LS-Boost}$(1)$ if $\lambda = 0$). This way of updating $\beta_{j_k}$ leads to a simple soft-thresholding operation~\cite{donoho1995wavelet} and is \emph{structurally} different from forward stagewise update rules. In contrast, the boosting algorithm $\RFSe$ that we propose here is based on subgradient descent on the dual of the \lasso problem~\eqref{poi-lasso}, i.e., problem~\eqref{FS-dual-prob}.

\subsubsection{Connecting $\RFSe$ to the Frank-Wolfe method}\label{FW_appendix} Although we developed and analyzed $\RFSe$ from the perspective of subgradient descent, one can also interpret $\RFSe$ as the Frank-Wolfe algorithm in convex optimization \cite{frank-wolfe, jaggi2013revisiting, GF-FW} applied to the \lasso \eqref{poi-lasso}. This secondary interpretation can be derived directly from the structure of the updates in $\RFSe$ or as a special case of a more general primal-dual equivalence between subgradient descent and Frank-Wolfe developed in \cite{bach2012duality}. We choose here to focus on the subgradient descent interpretation since it provides a natural unifying framework for a general class of boosting algorithms (including $\FSe$ and $\RFSe$) via a single algorithm applied to a parametric class of objective functions. Other authors have commented on the similarities between boosting algorithms and the Frank-Wolfe method, see for instance \cite{clarkson} and \cite{jaggi2013revisiting}.

\subsection{Additional Details for Section~5}

\subsubsection{Proof of Theorem \ref{PATHe-guarantees}}\label{redfolder}
We first prove the feasibility of $\hat{\beta}^k$ for the \lasso problem with parameter $\bar\delta_k$.  We do so by induction.  The feasibility of $\hat{\beta}^k$ is obviously true for $k = 0$ since $\hat{\beta}^0 = 0$ and hence $\|\hat\beta^0\|_1 = 0 < \bar\delta_0$.  Now suppose it is true for some iteration $k$, i.e., $\|\hat{\beta}^k\|_1 \le \bar\delta_k$.  Then the update for $\hat{\beta}^{k+1}$ in step (3.) of algorithm $\PATHe$ can be written as $\hat{\beta}^{k+1} = (1-\tfrac{\varepsilon}{\bar\delta_k})\hat{\beta}^k + \tfrac{\varepsilon}{\bar\delta_k}(\bar\delta_k \sgn((\hat{r}^k)^T\bX_{j_k})e_{j_k})$, from which it follows that
\begin{align*}
\|\hat{\beta}^{k+1} \|_1 ~&=~  \|(1-\tfrac{\varepsilon}{\bar\delta_k})\hat{\beta}^k + \tfrac{\varepsilon}{\bar\delta_k}(\bar\delta_k \sgn((\hat{r}^k)^T\bX_{j_k})e_{j_k}) \|_1 \\
&\le (1-\tfrac{\varepsilon}{\bar\delta_k})\|\hat{\beta}^k\|_1  + \tfrac{\varepsilon}{\bar\delta_k}\|\bar\delta_k e_{j_k} \|_1 \le (1-\tfrac{\varepsilon}{\bar\delta_k})\bar\delta_k  + \tfrac{\varepsilon}{\bar\delta_k}\bar\delta_k  = \bar\delta_k \leq \bar\delta_{k+1} \ ,
\end{align*}
which completes the induction.

We now prove the bound on the average training error in part {\em (i)}.  In fact, we will prove something stronger than this bound, namely we will prove:
\begin{equation}\label{righthere} \displaystyle\frac{1}{k+1}\sum_{i=0}^k\frac{1}{\bar\delta_i}\left(L_n(\hat{\beta}^i) - L_{n, \bar\delta_i}^\ast\right) ~\leq~ \frac{\|\bX\hat{\beta}_{\text{LS}}\|_2^2}{2n\varepsilon(k+1)} + \frac{2\varepsilon}{n} \ ,
 \end{equation} from which average training error bound of part {\em (i)} follows since $\bar\delta_i \le \bar\delta$ for all $i$.  The update rule for the residuals given in Step (3.) of $\RFSe$ is:  $\hat r^{k+1} \gets \hat{r}^k - \varepsilon g^k$ where $g^k = \left[\sgn((\hat{r}^k)^T\bX_{j_k})\bX_{j_k} + \tfrac{1}{\bar\delta_k}(\hat{r}^k - \by)\right]$.  This update rule is precisely in the format of an elementary sequence process, see Appendix \ref{simple}, and we therefore can apply Proposition \ref{simpleprop}, and more specifically the inequality \eqref{waytoomuchsnow}.  Similar in structure to the proof of Theorem \ref{RFSe-guarantees}, we first need to translate the terms of Proposition \ref{simpleprop} to our setting:  once again the variables $x$ are now the residuals $r$, the iterates $x^i$ are now the iterates $\hat r^i$, etc.  The step-sizes of algorithm $\PATHe$ are fixed at $\varepsilon$, so we have $\alpha_i = \varepsilon$ for all $i \ge 0$.  Setting the value of $x$ in  Proposition \ref{simpleprop} to be least-squares residual value, namely $x = \hat r_{LS}$, and using the exact same logic as in the equations \eqref{39steps}, one obtains the following result about the left side of \eqref{waytoomuchsnow}:
$$\frac{1}{k+1}\sum_{i = 0}^k(g^i)^T(x^i - x) = \frac{1}{k+1}\sum_{i = 0}^k \omega_{\bar\delta_i} (\hat{\beta}^i) \ . $$
Let us now evaluate the right side of \eqref{waytoomuchsnow}.  We have $\|x^0 - x\|_2 = \|\hat r^0 - \hat r_{LS}\|_2=  \|\by - (\by - \bX\hat{\beta}_{\text{LS}})\|_2 = \|\bX\hat{\beta}_{\text{LS}}\|_2 $.  Also, it holds that
$$\|g^i\|_2 = \|\sgn((\hat r^i)^T\bX_{j_i})\bX_{j_i} - \tfrac{1}{\bar\delta_i}(\bX\hat{\beta}^i)\|_2 \leq \|\bX_{j_i}\|_2 + \|\bX(\tfrac{\hat{\beta}^i}{\bar\delta_i})\|_2 \leq 1 + \tfrac{1}{\bar\delta_i}\|\bX\|_{1,2}\|\hat{\beta}^i\|_1 \le 1 + \|\bX\|_{1,2} \le 2 \ , $$where the third inequality follows since $\|\hat{\beta}^i\|_1 \leq \bar\delta_i$ from the feasibility of $\hat\beta^i$ for the \lasso problem with parameter $\bar\delta_i$ proven at the outset, and the second and fourth inequalities follow from the assumption that the columns of $\bX$ have been normalized to have unit $\ell_2$ norm.  Therefore $G=2$ is a uniform bound on $\|g^i\|_2$.  Combining the above, inequality \eqref{waytoomuchsnow} implies that after running $\PATHe$ for $k$ iterations, it holds that:
\begin{equation}\label{lasso_gap_boundtwo}
\frac{1}{k+1}\sum_{i = 0}^k \omega_{\bar\delta_i}(\hat{\beta}^i) ~\leq~ \frac{\|\bX\hat{\beta}_{\text{LS}}\|_2^2}{2(k+1)\varepsilon} + \frac{2^2\varepsilon}{2}   ~=~ \frac{\|\bX\hat{\beta}_{\text{LS}}\|_2^2}{2\varepsilon(k+1)} + 2\varepsilon \ ,
\end{equation}where the inequality is the application of \eqref{waytoomuchsnow}.  From Proposition \ref{LASSO-dual} we have $L_n(\hat{\beta}^i) - L_{n, \bar\delta_i}^\ast ~\leq~ \tfrac{\bar\delta_i}{n}\cdot\omega_{\bar\delta_i}(\hat{\beta}^i)$, which combines with \eqref{lasso_gap_boundtwo} to yield:
$$\displaystyle\frac{1}{k+1}\sum_{i=0}^k\frac{1}{\bar\delta_i}\left(L_n(\hat{\beta}^i) - L_{n, \bar\delta_i}^\ast\right) \ \le \ \displaystyle\frac{1}{(k+1)}\frac{1}{n}\sum_{i=0}^k\omega_{\bar\delta_i}(\hat{\beta}^i) \ \le \ \frac{\|\bX\hat{\beta}_{\text{LS}}\|_2^2}{2n\varepsilon(k+1)} + \frac{2\varepsilon}{n} \ . $$
This proves \eqref{righthere} which then completes the proof of part {\em (i)} through the bounds $\bar\delta_i \le \bar\delta$ for all $i$.

Part {\em (ii)} is a restatement of the feasibility of $\hat{\beta}^k$ for the \lasso problem with parameter $\bar\delta_k$ which was proved at the outset, and is re-written to be consistent with the format and for comparison with Theorem \ref{RFSe-guarantees}  Last of all, part {\em (iii)} follows since at each iteration at most one new coefficient is introduced at a non-zero level.  \qed

\section{Additional Details on the Experiments}\label{sec:append:comp-details}

\setcounter{table}{0}
\renewcommand{\thetable}{B.\arabic{table}}

We describe here some additional details pertaining to the computational results performed in this paper. We first describe in some more detail the real datasets that have been considered in the paper.

\paragraph{Description of datasets considered}

\ \\
We considered four different publicly available microarray datasets as described below.

\paragraph{Leukemia dataset}
This dataset, taken from~\cite{dettling2003boosting}, has binary response with continuous covariates, with $72$ samples and approximately $3500$ covariates.
We further processed the dataset by taking a subsample of $p=500$ covariates, while retaining all $n = 72$ sample points. We artificially generated the response
$\by$ via a linear model with the given covariates $\bX$ (as described in Eg-A in Section \ref{1999}). The true regression coefficient $\beta^{\text{pop}}$ was taken as
$\beta^{\text{pop}}_{i} =1$ for all $i \leq 10$ and zero otherwise.

\paragraph{Golub dataset}
The original dataset was taken from the R package {\texttt{mpm}}, which had 73 samples with approximately 5000 covariates. We reduced this to $p = 500$ covariates (all samples were retained).
Responses $\by$ were generated via a linear model with $\beta^{\text{pop}}$ as above.

\paragraph{Khan dataset}
This dataset was taken from the dataset webpage~\url{http://statweb.stanford.edu/~tibs/ElemStatLearn/datasets/} accompanying the book~\cite{ESLBook}. The
original covariate matrix ({\texttt{khan.xtest}}), which had 73 samples with approximately 5000 covariates, was reduced to $p=500$ covariates (all samples were retained).
Responses $\by$ were generated via a linear model with $\beta^{\text{pop}}$ as above.

\paragraph{Prostate cancer dataset}
This dataset appears in~\cite{LARS} and is available from the R package {\texttt{LARS}}.
The first column {\texttt{lcavol}} was taken as the response (no artificial response was created here).
We generated multiple datasets from this dataset, as follows:
\begin{itemize}
\item[(a)] One of the datasets is the original one with $n=97$ and $p=8$.

\item[(b)] We created another dataset, with $n=97$ and $p=44$ by enhancing the covariate space to include second order interactions.

\item[(c)] We created another dataset, with $n=10$ and $p=44$. We subsampled the dataset from (b), which again was enhanced to include second order interactions.

\end{itemize}

Note that in all the examples above we standardized $\bX$ such that the columns have unit $\ell_{2}$ norm, before running the different algorithms studied herein.\medskip

\begin{figure}[h!]
\centering
\scalebox{0.95}[.8]{\begin{tabular}{l c c c}
&\scriptsize{\sf {Leukemia, SNR=1, p=500}} &\scriptsize{\sf{ Leukemia, SNR=3, p=500} } & \scriptsize{ \sf{Khan, SNR=1, p=500 } } \\
&\scriptsize{\sf {\ebs}} &\scriptsize{\sf{\ebs} } & \scriptsize{ \sf{\ebs } } \\
\rotatebox{90}{\sf {\scriptsize {~~~~~~~Train/Test Errors (in relative scale)}}}&\includegraphics[width=0.31\textwidth,height=0.3\textheight,  trim = 1.0cm 1.3cm 1cm 1.9cm, clip = true ]{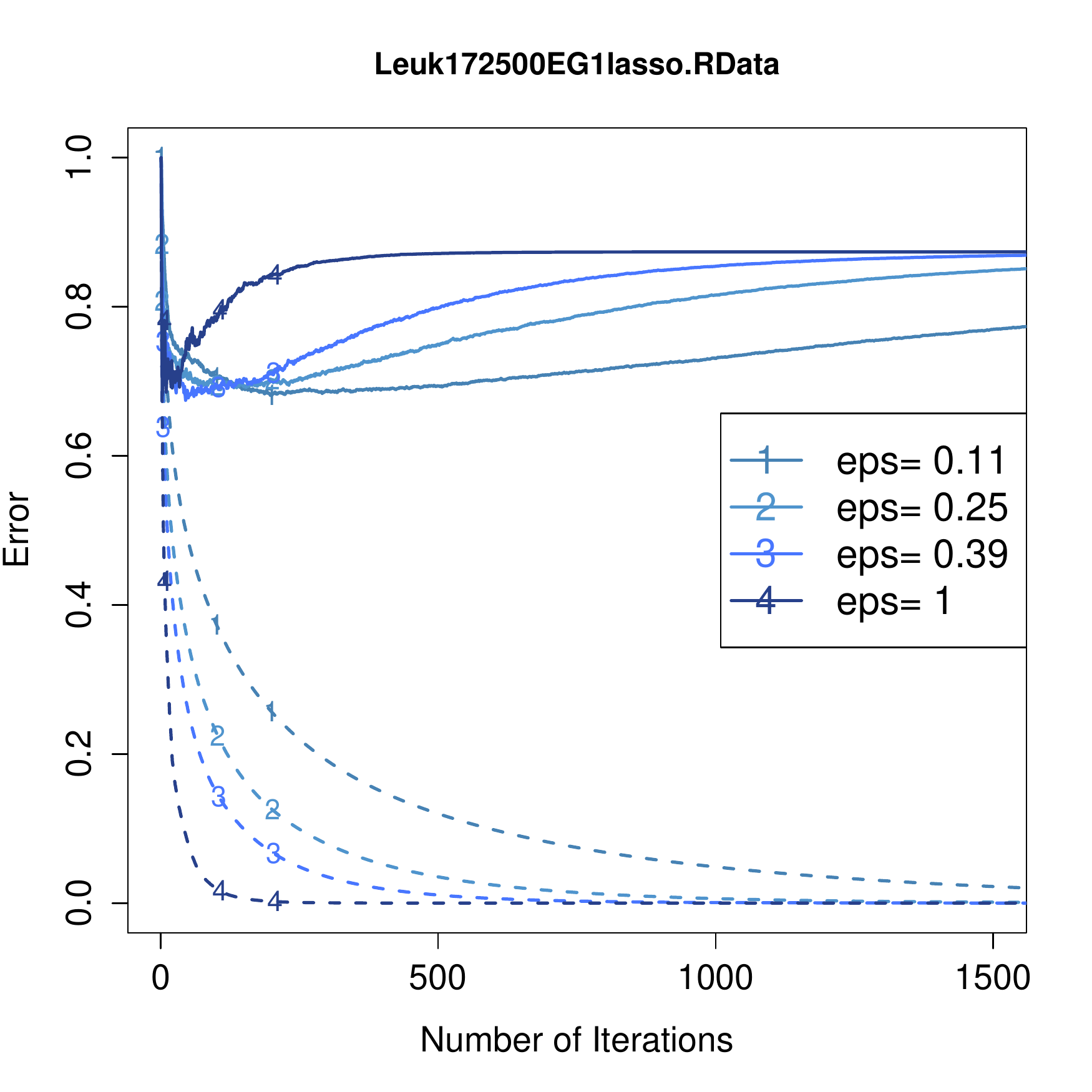}&
\includegraphics[width=0.31\textwidth,height=0.3\textheight,  trim = 1.8cm 1.3cm 1cm 2cm, clip = true ]{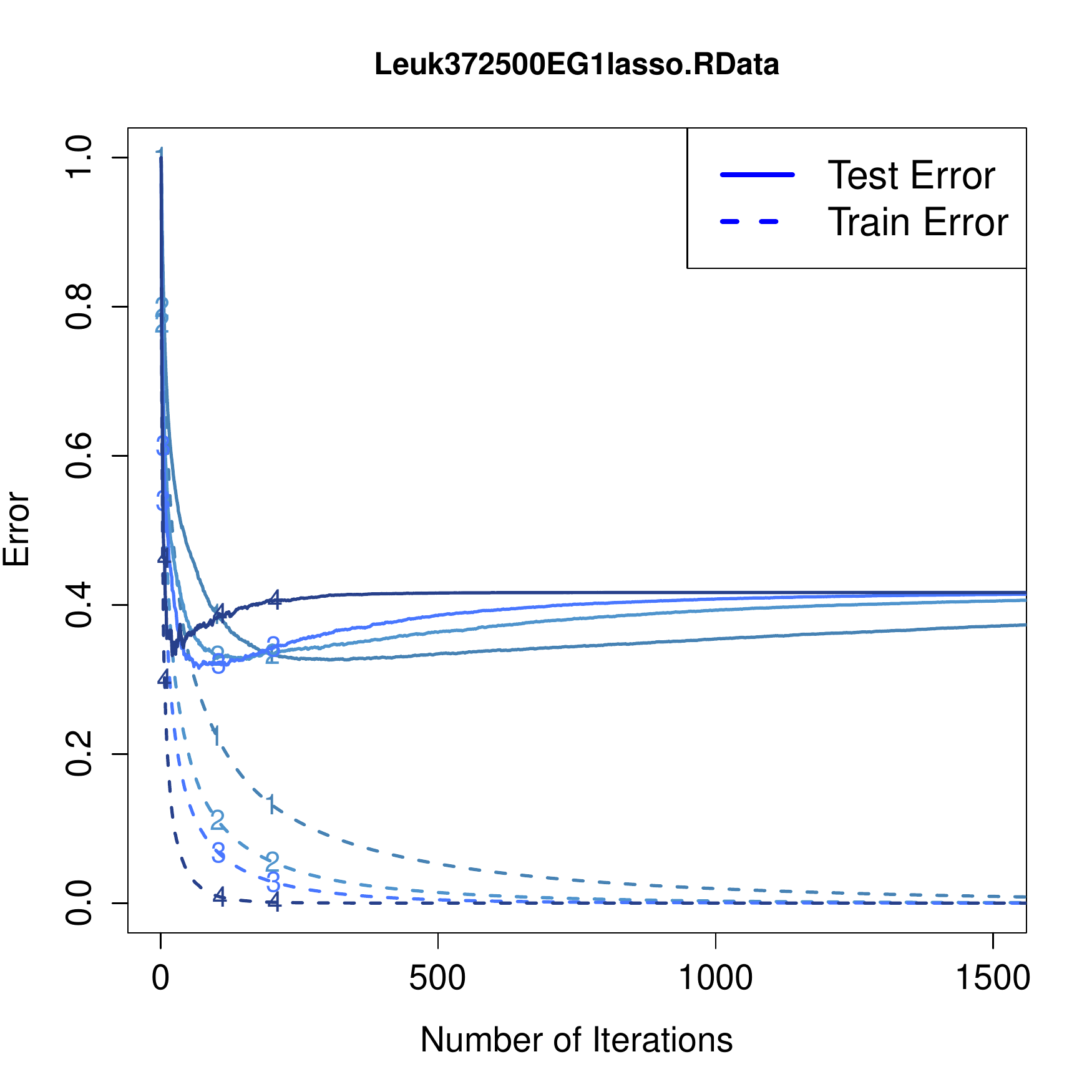}&
\includegraphics[width=0.31\textwidth,height=0.3\textheight,  trim = 1.8cm 1.3cm 1cm 2cm, clip = true ]{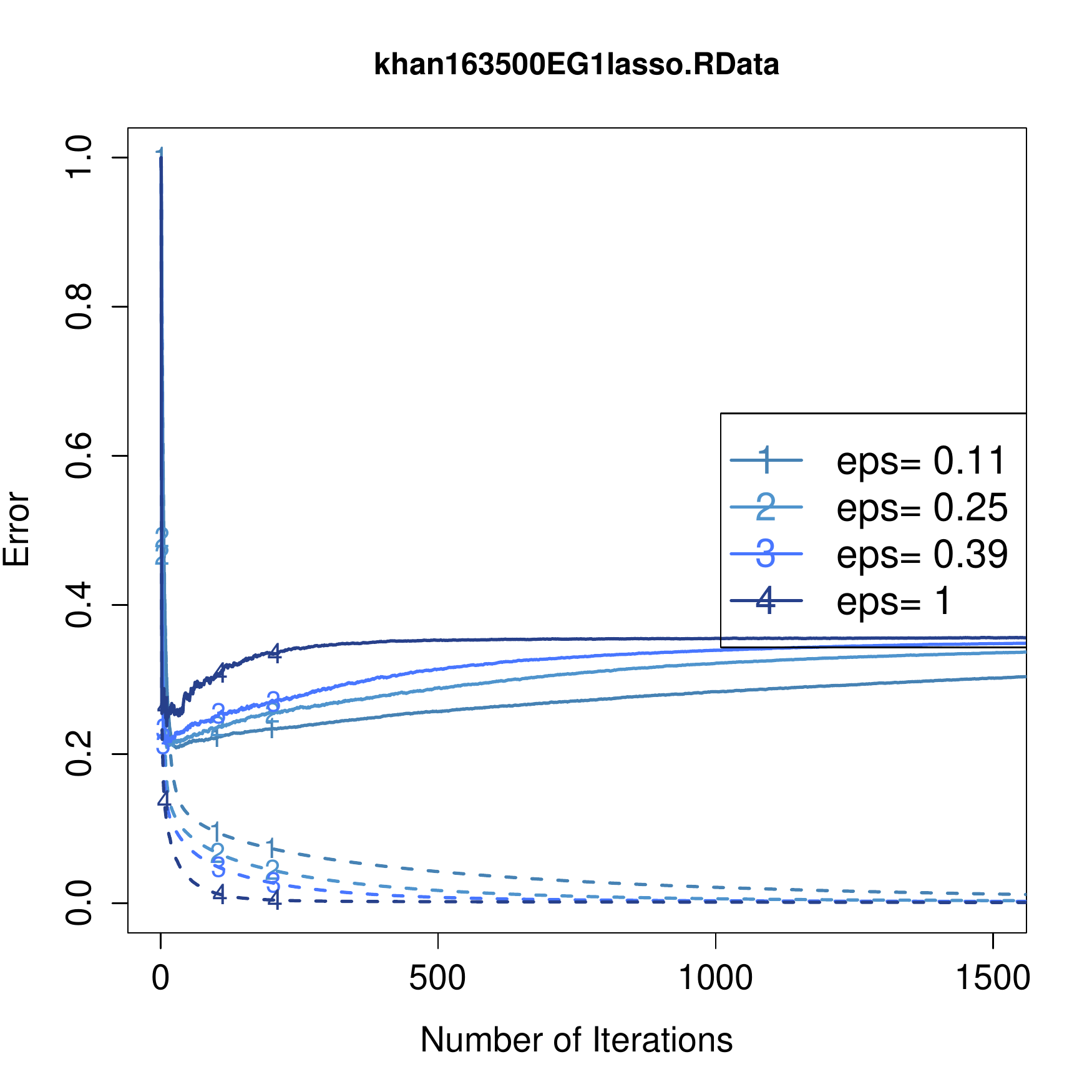} \\
&&&\\
&\scriptsize{\sf {$\FSe$}} &\scriptsize{\sf{$\FSe$} } & \scriptsize{ \sf{$\FSe$ } } \\
\rotatebox{90}{\sf {\scriptsize {~~~~~~~Train/Test Errors (in relative scale)}}}&\includegraphics[width=0.31\textwidth,height=0.3\textheight,  trim = 1.0cm 1.3cm 1cm 1.9cm, clip = true ]{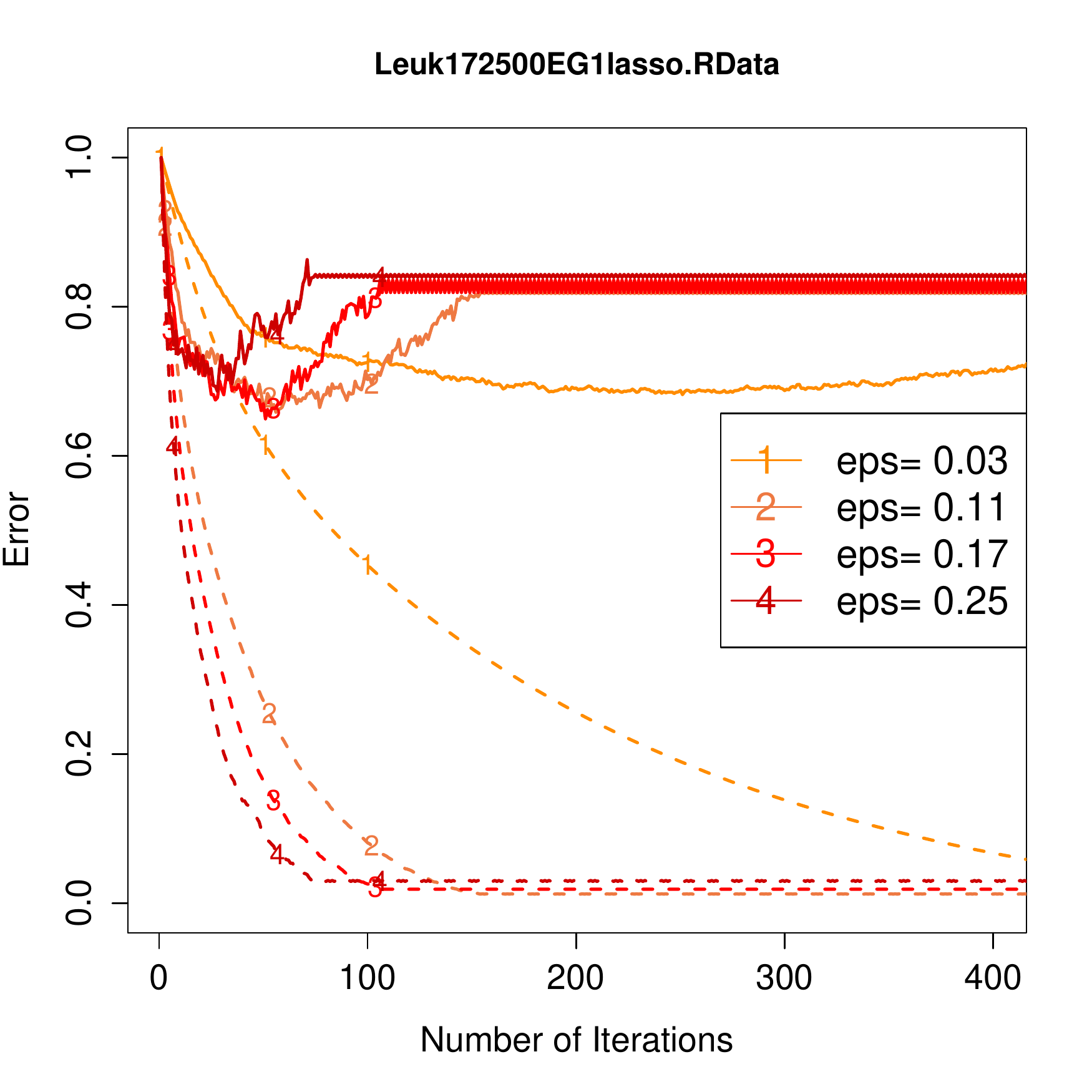}&
\includegraphics[width=0.31\textwidth,height=0.3\textheight,  trim = 1.8cm 1.3cm 1cm 2cm, clip = true ]{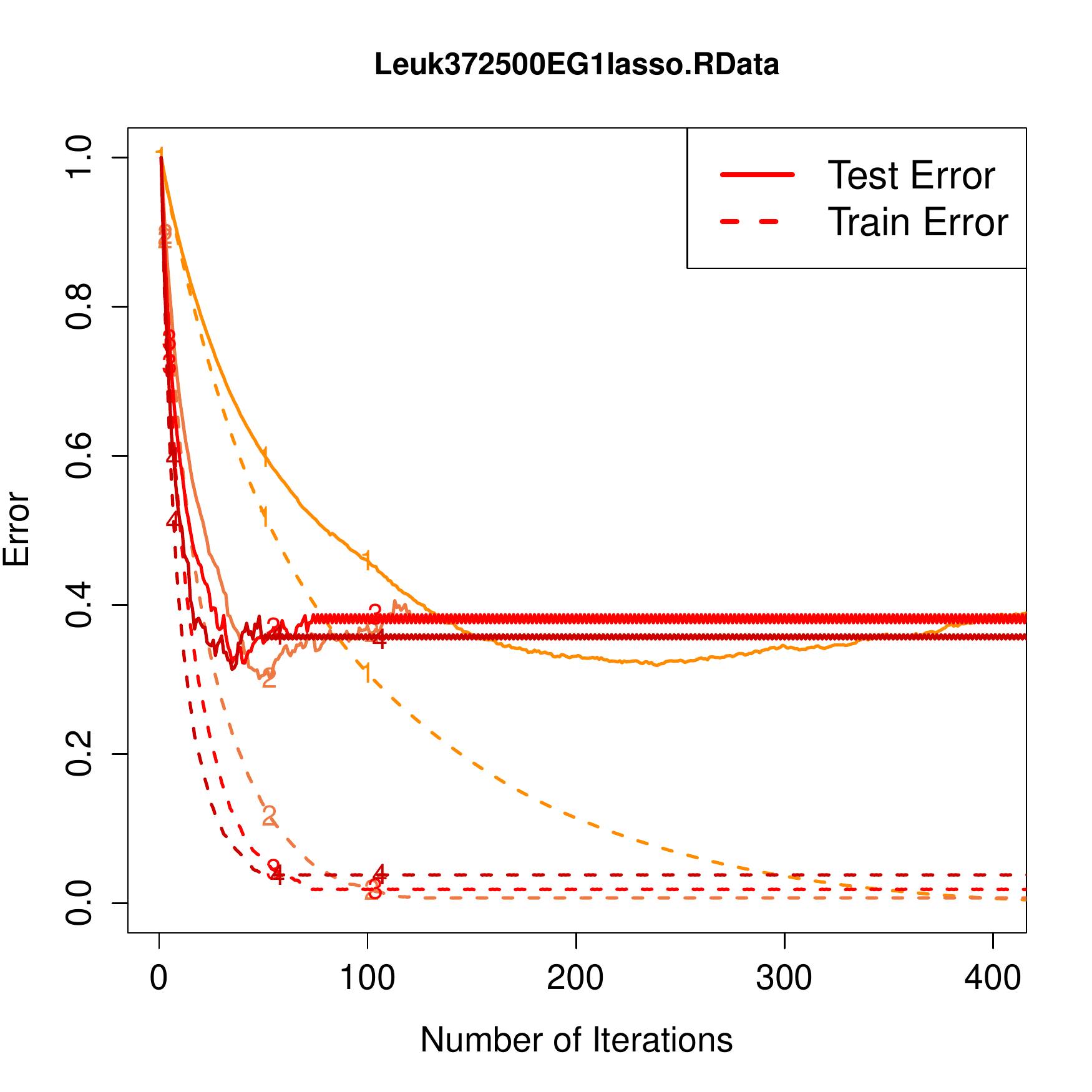}&
\includegraphics[width=0.31\textwidth,height=0.3\textheight,  trim = 1.8cm 1.3cm 1cm 2cm, clip = true ]{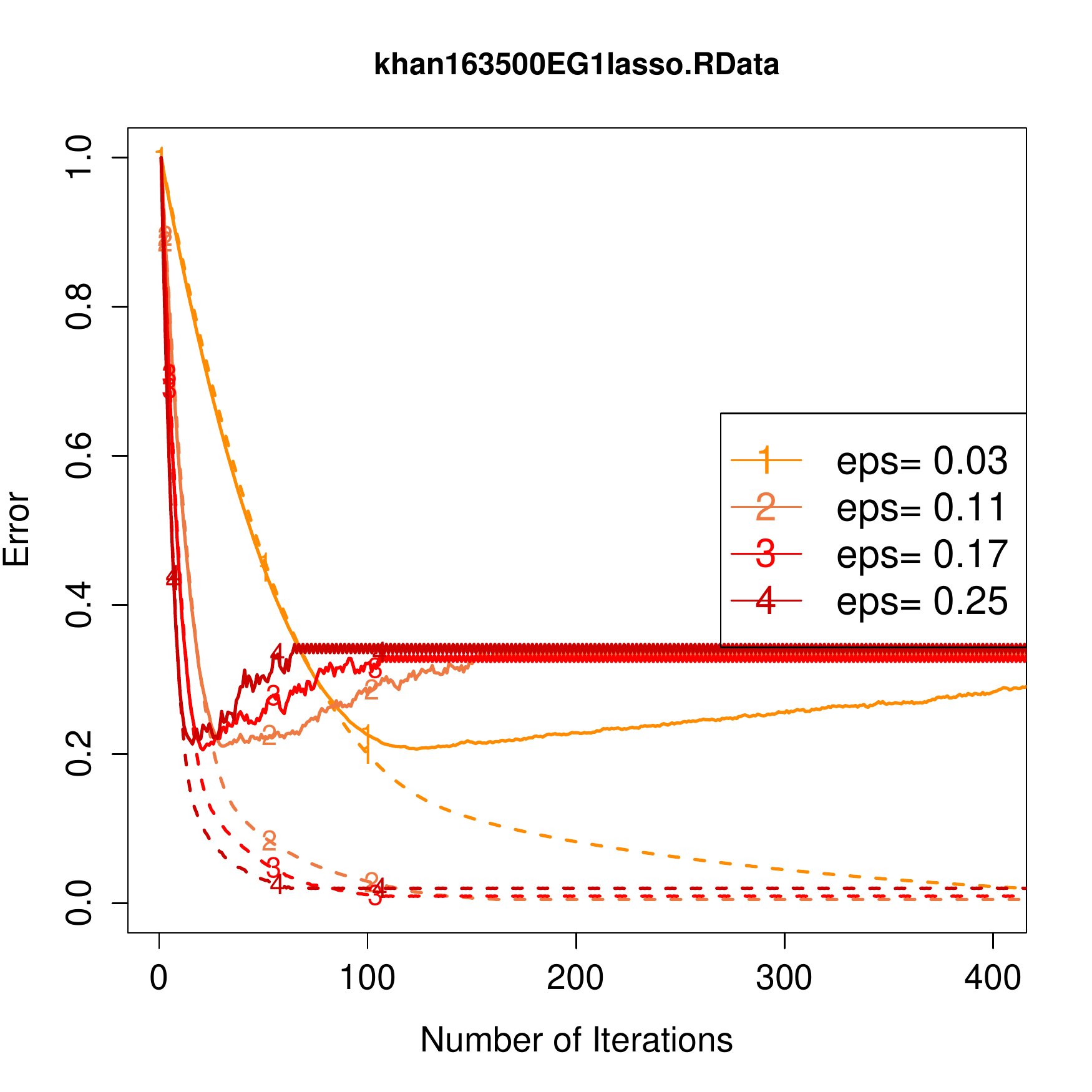} \\
&\scriptsize{\sf {Number of Iterations}} &\scriptsize{\sf {Number of Iterations}}  & \scriptsize{\sf {Number of Iterations}}  \\
\end{tabular}}
\caption{{\small{ Figure showing the training and test errors (in relative scale) as a function of boosting iterations, for both \ebs (top panel) and $\FSe$ (bottom panel).  As the number of iterations increases, the training error shows a global monotone pattern. The test errors however, initially decrease and then start increasing after reaching a minimum. The best test errors obtained are found to be sensitive to the choice of $\varepsilon$. Two different datasets have been considered: the Leukemia dataset (left and middle panels) and the Khan dataset (right panel), as described in Section~\ref{sec:experiments}. }} }\label{fig:train-test-leuk}
\end{figure}

\begin{table}[h!]
\centering
\scalebox{0.85}{\begin{tabular}{c  c c c c c c c }
Dataset  & SNR & n  & \ebs  & $\FSe$ &  $\text{FS}_{0}$ & Stepwise & \lasso \\
     & &  & $\times 10^{-2}$  &  $\times 10^{-2}$  &  $\times 10^{-2}$  &  $\times 10^{-2}$  & $\times 10^{-2}$  \\
  \hline \\
 \multirow{ 3}{*}{\rotatebox[origin=c]{50}{\small {Leukemia}}}  & 1&72 & 65.9525 (1.8221) & 66.7713 (1.8097) & 68.1869 (1.4971) & 74.5487 (2.6439) & 68.3471 (1.584) \\
   & 3&72 & 35.4844 (1.1973) & 35.5704  (0.898) & 35.8385 (0.7165) & 38.9429 (1.8030) & 35.3673 (0.7924) \\
 & 10&72 & 13.5424 (0.4267) & 13.3690 (0.3771) & 13.6298 (0.3945) & 14.8802 (0.4398) & 13.4929 (0.4276) \\ \\
    \multirow{ 3}{*}{\rotatebox[origin=c]{50}{\small {Khan} }}  & 1&63 & 22.3612 (1.1058) & 22.6185 (1.0312) & 22.9128 (1.1209) & 25.2328 (1.0734) & 23.5145 (1.2044) \\
   & 3&63 & 9.3988 (0.4856) & 9.4851 (0.4721) & 9.6571 (0.3813) & 10.8495 (0.3627) & 9.2339 (0.404) \\
 & 10&63 & 3.4061 (0.1272) & 3.4036  (0.1397) & 3.4812 (0.1093) & 3.7986 (0.0914) & 3.1118 (0.1229) \\ \\
    \multirow{ 3}{*}{\rotatebox[origin=c]{50}{\small {Eg-A, $\rho=0.8$}} } & 1  & 50  & 53.1406 (1.5943) & 52.1377 (1.6559) & 53.6286 (1.4464) & 60.3266 (1.9341) & 53.7675 (1.2415) \\
   & 3  & 50 & 29.1960  (1.2555) & 29.2814 (1.0487) & 30.0654 (1.0066) & 33.4318 (0.8780) & 29.8000 (1.2662) \\
  & 10  & 50 & 12.2688 (0.3359) & 12.0845 (0.3668) & 12.6034 (0.5052) & 15.9408 (0.7939) & 12.4262 (0.3660) \\  \\
    \multirow{ 3}{*}{\rotatebox[origin=c]{50}{\small {Eg-A, $\rho=0$}}} & 1 & 50 & 74.1228 (2.1494) & 73.8503 (2.0983) & 75.0705 (2.5759) & 92.8779 (2.7025) & 75.0852  (2.1039)  \\
   & 3  & 50 & 38.1357 (2.7795) & 40.0003 (1.8576) & 41.0643 (1.5503) & 43.9425 (3.9180) & 41.4932  (2.2092) \\
   & 10  & 50 & 14.8867 (0.6994) & 12.9090 (0.5553) & 15.2174 (0.7086) & 12.5502 (0.8256) & 15.0877  (0.7142) \\  \\
\end{tabular}}
\caption{\small{Table showing the prediction errors (in percentages) of different methods:  \ebs, $\FSe$ (both for different values of $\varepsilon$), $\text{FS}_{0}$, (forward) Stepwise regression, and \lassoperiod.
The numbers within parentheses denote standard errors.
 \ebs, $\FSe$ are found to exhibit similar statistical performances as the \lassoperiod, in fact in some examples the boosting methods seem to be marginally better than \lassoperiod.
 The predictive performance of the models were also found to be
 sensitive to the choice of the learning rate $\varepsilon$.
  For  $\text{FS}_{0}$  and Stepwise we used the~{\texttt{R}} package {\texttt{LARS}}~\cite{LARS} to compute the solutions. For all the cases, $p=500$. For Eg-A, we took $n=50$.
Both \ebs and $\FSe$ were run for a few values of $\varepsilon$ in the range $[0.001-0.8]$ -- in all cases,
the optimal models (see the text for details) for \ebs and $\FSe$ were achieved at a value of $\varepsilon$ larger than its limiting version $\varepsilon = 0+$, thereby suggesting the sensitivity of the best predictive model to the learning rate $\varepsilon$.}}
\label{table-sensitive-eps-1}
\end{table}

\paragraph{Sensitivity of the Learning Rate in \ebs and $\FSe$}
We performed several experiments running \ebs and $\FSe$ on an array of real and synthetic datasets, to explore how the training and test errors change as a function of
the number of boosting iterations and the learning rate.  Some of the results appear in Figure~\ref{fig:train-test-leuk}.
The training errors were found to decrease with increasing number of boosting iterations.
The rate of decay, however, is very sensitive to the value of $\varepsilon$, with smaller values of $\varepsilon$
leading to slower convergence behavior to the least squares fit, as expected.
The test errors were found to decrease and then increase after reaching a minimum; furthermore, the best predictive models were found to be sensitive to the choice of
$\varepsilon$.

In addition to the above, we also performed a series of experiments on both real and synthetic datasets comparing the performance of \ebs and $\FSe$ to
other sparse learning methods, namely \lasso, stepwise regression~\cite{LARS} and $\FS_0$~\cite{LARS}.
Our results are presented in
Table~\ref{table-sensitive-eps-1}. In all the cases, we found that the performance of $\FSe$ and \ebs were at least as good as \lassoperiod. And in some cases, the performances of $\FSe$ and \ebs were superior.
The best predictive models achieved by \ebs and $\FSe$ correspond to values of $\varepsilon$ that are larger than zero or even close to one -- this suggests that a proper choice of $\varepsilon$ can lead to superior models.

\paragraph{Statistical properties of  $\RFSe$, \lasso and $\FSe$: an empirical study} We performed some experiments to evaluate the performance of $\RFSe$, in terms of predictive accuracy and
sparsity of the optimal model, versus the more widely known methods $\FSe$ and \lassoperiod. In all the cases, we took a small value of $\varepsilon=10^{-3}$. We ran $\RFSe$ on a grid of twenty $\delta$ values, with the limiting solution corresponding to the \lasso estimate at the particular value of $\delta$ selected. In all cases, we found that when $\delta$ was large, i.e., larger than the best $\delta$ for the \lasso (in terms of obtaining a model with the best predictive performance), $\RFSe$ delivered a model
with excellent statistical properties -- $\RFSe$ led to sparse solutions (the sparsity was similar to that of the best \lasso model) and the predictive performance was as good as, and in some cases better than, the \lasso solution.
This suggests that the choice of $\delta$ does not play a very crucial role in the $\RFSe$ algorithm, once it is chosen to be reasonably large; indeed the number of boosting iterations play a more important role in obtaining good quality statistical
estimates.
When compared to $\FSe$ (i.e., the version of $\RFSe$ with $\delta=\infty$) we observed that the best models delivered by $\RFSe$ were more sparse (i.e., with fewer non-zeros) than the best
$\FSe$ solutions. This complements a popular belief about boosting in that it delivers models that are quite dense -- see the discussion herein in Section~\ref{related-work-lasso-boost}. Furthermore, it shows that the particular form of regularized boosting that we consider, $\RFSe$, does indeed induce sparser solutions.
Our detailed results are presented in Table~\ref{tab:pred-perf-1-rfse}.

\paragraph{Comments on Table~\ref{table-sensitive-eps-1}}
In this experiment, we ran $\FSe$ and \ebs for thirty different values of $\eps$ in the range $0.001$ to $0.8$. The entire regularization paths for the \lasso, $\text{FS}_{0}$, and the more aggressive Stepwise regression were computed with
the~{\texttt{LARS}} package. First, we observe that Stepwise regression, which is quite fast in reaching an unconstrained least squares solution, does not perform well in terms of obtaining a model with good predictive performance.
The slowly learning boosting methods perform quite well -- in fact their performances are quite similar to the best \lasso solutions.
A closer inspection shows that $\FSe$ almost always delivers the best predictive models when $\varepsilon$ is allowed to be flexible. While a good automated method to find the optimal value of $\varepsilon$ is certainly worth investigating, we leave this for future work (of course, there are excellent heuristics for choosing the optimal $\varepsilon$ in practice, such as cross validation, etc.).
However, we do highlight that in practice a strictly non-zero learning rate $\varepsilon$ may lead to better models than its limiting version $\varepsilon = 0 +$.

For Eg-A ($\rho=0.8$), both \ebs and $\FSe$ achieved the best model at $\varepsilon=10^{-3}$. For Eg-A ($\rho=0$), \ebs achieved the best model at $\varepsilon = 0.1, 0.7, 0.8$ and $\FSe$ achieved the best model at $\varepsilon = 10^{-3},0.7,0.8$ (both for SNR values 1, 3, 10 respectively). For the Leukemia dataset, \ebs achieved the best model at $\varepsilon = 0.6,0.7,0.02$ and $\FSe$ achieved the best model at $\varepsilon = 0.6,0.02,0.02$ (both for SNR values 1, 3, 10 respectively). For the Khan dataset, \ebs achieved the best model at $\varepsilon = 0.001,0.001,0.02$ and $\FSe$ achieved the best model at  $\varepsilon = 0.001,0.02,0.001$ (both for SNR values 1, 3, 10 respectively).\medskip

\begin{table}[h!]
\centering
\scalebox{0.99}{\centering{\begin{tabular}{ccc cc ccc}
  \hline
\multicolumn{8}{c}{Real Data Example: Leukemia} \\

 Method    & n & p & SNR & Test Error  & Sparsity & ${\| \hat{\beta}^{\text{opt}} \|_{1}}/{\| \hat{\beta}^{*} \|_{1}}$ &  ${\delta}/{\delta_{\max}}$ \\
  \hline
$\FSe$& 72 & 500 & 1 & 0.3431 (0.0087) & 28 & 0.2339 &- \\
  $\RFSe$& 72 & 500 & 1 & 0.3411 (0.0086) & 25 & 0.1829 & 0.56 \\
  \lasso & 72 & 500 & 1 & 0.3460 (0.0086) & 30 & 1 &  0.11 \\  \\
  $\FSe$& 72 & 500 & 10 & 0.0681 (0.0014) & 67 & 0.7116 &-\\
  $\RFSe$ & 72 & 500 & 10 & 0.0659 (0.0014) & 60 & 0.5323 &  0.56\\
  \lasso & 72 & 500 & 10 & 0.0677 (0.0015) & 61 & 1 & 0.29 \\   \hline \\
  \multicolumn{8}{c}{Synthetic Data Examples: Eg-B (SNR=1)} \\
 Method & n & p & $\rho$ & Test Error  & Sparsity& ${\| \hat{\beta}^{\text{opt}} \|_{1}}/{\| \hat{\beta}^{*} \|_{1}}$ &  ${\delta}/{\delta_{\max}}$ \\  \hline
  $\FSe$& 50 & 500 & 0 & 0.19001 (0.0057) & 56 & 0.9753 &  - \\
  $\RFSe$ & 50 & 500 & 0 & 0.18692 (0.0057) & 51 & 0.5386 &  0.71 \\
  \lasso & 50 & 500 & 0 & 0.19163 (0.0059) & 47 & 1 & 0.38 \\  \\
     $\FSe$& 50 & 500 & 0.5& 0.20902 (0.0057) & 14 & 0.9171 &  - \\
  $\RFSe$ & 50 & 500 & 0.5& 0.20636 (0.0055) & 10 & 0.1505 &  0.46 \\
  \lasso & 50 & 500 & 0.5& 0.21413 (0.0059) & 13 & 1 & 0.07 \\  \\
  $\FSe$& 50 & 500 & 0.9& 0.05581 (0.0015) & 4 & 0.9739 & - \\
  $\RFSe$ & 50 & 500 & 0.9& 0.05507 (0.0015) & 4 & 0.0446 &  0.63 \\
  \lasso & 50 & 500 & 0.9& 0.09137 (0.0025) & 5 & 1 & 0.04 \\  \hline\\
  \end{tabular}}}
  \caption{\small{Table showing the statistical properties of $\RFSe$ as compared to \lasso and $\FSe$. Both $\RFSe$ and $\FSe$ use $\varepsilon=0.001$.
  The model that achieved the best predictive performance (test-error) corresponds to $\hat{\beta}^{\text{opt}}$. The limiting model (as the number of boosting iterations is taken to be infinitely large) for each method is denoted by $\hat{\beta}^{*}$.
  ``Sparsity" denotes the number of coefficients in $\hat{\beta}^{\text{opt}}$ larger than $10^{-5}$ in absolute value. $\delta_{\max}$ is the $\ell_{1}$-norm of the least squares solution
  with minimal $\ell_{1}$-norm.   Both $\RFSe$ and \lasso were run for a few $\delta$ values of the form $\eta \delta_{\max}$, where $\eta$ takes on twenty values in $[0.01,0.8]$.
  For the real data instances, $\RFSe$ and \lasso were run for a maximum of 30,000 iterations, and $\FSe$ was run for 20,000 iterations. For the synthetic examples,
  all methods were run for a maximum of 10,000 iterations. The best models for $\RFSe$ and $\FSe$ were all obtained in the interior of the path. The best models delivered by $\RFSe$ are seen to be more sparse and have better predictive performance than the best models obtained by $\FSe$.
  The performances of \lasso and $\RFSe$ are found to be quite similar, though in some cases
  $\RFSe$ is seen to be at an advantage in terms of better predictive accuracy. }}\label{tab:pred-perf-1-rfse}
\end{table}

\end{document}